\documentclass[11pt]{article}
\usepackage{subfigure}
\usepackage[font={small}]{caption}
\usepackage{amsfonts}
\usepackage{amssymb}
\usepackage{amsthm}
\usepackage{amsmath}
\usepackage[a4paper, margin=2.2cm]{geometry}
\usepackage{mathrsfs}
\usepackage{epsfig}
\usepackage{tikz,fp,ifthen}
\usepackage{float}
\usepackage[all]{xy}
\usetikzlibrary{decorations.pathreplacing}
\usepackage{algorithm}
\usepackage{algorithmic}
\usepackage{bbm}
\usepackage{todonotes}
\usepackage{color}
\definecolor{hanblue}{rgb}{0.27, 0.42, 0.81}
\definecolor{red}{rgb}{1.0, 0.0, 0.0}
\usepackage[colorlinks, citecolor=hanblue,linkcolor=red]{hyperref}
\usepackage[capitalise]{cleveref}
\crefformat{equation}{(#2#1#3)}
\usepackage{enumitem}

\usepackage[backend=bibtex,maxnames=4,url=false,style=numeric-comp,isbn=false, giveninits=true,eprint=false,sorting=nyt,doi=false,date=year]{biblatex}

\bibliography{lit_dat.bib}

\usepackage{scalerel}

\usepackage{graphicx}
\DeclareMathOperator*{\argmax}{arg\,max}
\DeclareMathOperator*{\argmin}{arg\,min}

\newcommand{\domain}{\mathrm{dom}}
\newcommand{\weakstar}{\stackrel{*}{\rightharpoonup}}

\newcommand{\e}{\varepsilon}

\newcommand{\lsc}{lower semicontinuous}
\newcommand{\abs}[1]{|#1|}
\newcommand{\predual}[1]{{#1}_*}
\newcommand{\preadj}[1]{{#1}_*}
\DeclareMathOperator{\sign}{sign}
\renewcommand{\phi}{\varphi}
\let\sp\relax
\newcommand{\sp}[1]{\left\langle #1 \right\rangle}
\newcommand{\wsto}{\weakstar}
\newcommand{\defeq}{:=}

\usepackage{mathtools}
\usepackage{hyperref}

\newcommand{\res}{\mathop{\hbox{\vrule height 7pt width .5pt depth 0pt
\vrule height .5pt width 6pt depth 0pt}}\nolimits}

\newcommand{\R}{\mathbb{R}}
\newcommand{\Z}{\mathbb{Z}}

\newcommand{\N}{\mathbb{N}}
\newcommand{\M}{\mathcal{M}}

\newcommand{\Ext}{{\rm Ext}}

\newcommand{\Rea}{{\rm Re}}
\newcommand{\norm}[1]{\|#1\|}

\newcommand\restr[2]{{
  \left.\kern-\nulldelimiterspace 
  #1 
  \vphantom{\big|} 
  \right|_{#2} 
  }}

\newtheorem{thm}{Theorem}

\newtheorem{lemma}[thm]{Lemma}
\newtheorem{prop}[thm]{Proposition}
\newtheorem{cor}[thm]{Corollary}

\theoremstyle{definition}
\newtheorem{dfnz}[thm]{Definition}

\newtheorem{rem}[thm]{Remark}

\numberwithin{equation}{section}
\numberwithin{thm}{section}

\title{A sparse optimization approach to infinite infimal convolution regularization
}
\author{Kristian Bredies\footnote{Kristian Bredies, Institute of Mathematics and Scientific Computing,
    University of Graz, Heinrichstra\ss{}e 36, A-8010 Graz,
    Austria. Email: \texttt{kristian.bredies@uni-graz.at}} \and Marcello Carioni\footnote{Marcello Carioni, Department of Applied Mathematics, University of Twente, P.O. Box 217, 7500 AE Enschede, The Netherlands. Email: \texttt{m.c.carioni@utwente.nl}} \and Martin Holler \footnote{Martin Holler, Institute of Mathematics and Scientific Computing,  University of Graz, Heinrichstra\ss{}e 36, A-8010 Graz, Austria. Email: \texttt{martin.holler@uni-graz.at}
    } \and Yury Korolev\footnote{Yury Korolev, Department of Mathematical Sciences, University of Bath,  Bath BA2 7AY, UK. 
Email: \texttt{ymk30@bath.ac.uk}}  \and Carola-Bibiane Sch\"onlieb\footnote{Carola-Bibiane Sch\"onlieb. Centre for Mathematical Sciences, University of Cambridge, Wilberforce Road, CB3 0WA, Cambridge, UK.
Email: \texttt{cbs31@cam.ac.uk}}}
\date{}
\begin{document}
\maketitle

\begin{abstract}
\sloppy
In this paper we introduce the class of \emph{infinite infimal convolution functionals} and apply these functionals to the regularization of ill-posed inverse problems. The proposed regularization involves an infimal convolution of a continuously parametrized family of convex, positively one-homogeneous functionals defined on a common Banach space $X$. We show that, under mild assumptions, this functional admits an equivalent convex lifting in the space of measures with values in $X$. This reformulation allows us to prove well-posedness of a Tikhonov regularized inverse problem and opens the door to a sparse analysis of the solutions. In the case of finite-dimensional measurements we prove a \emph{representer} theorem, showing that there exists a solution of the inverse problem that is sparse, in the sense that it can be represented as a linear combination of the extremal points of the ball of the lifted infinite infimal convolution functional. Then, we design a generalized conditional gradient method for computing solutions of the inverse problem without relying on an a priori discretization of the parameter space and of the Banach space $X$. The iterates are constructed as linear combinations of the extremal points of the lifted infinite infimal convolution functional. We prove a sublinear rate of convergence for our algorithm and  apply it to denoising of signals and images using, as regularizer, infinite infimal convolutions of fractional-Laplacian-type operators with adaptive orders of smoothness and anisotropies.



\end{abstract}

{
  \hypersetup{linkcolor=black}
  \tableofcontents
}

\section{Introduction}
The infimal convolution of two convex functions $f_1,f_2 : X \rightarrow (-\infty,+\infty]$ is a classical convexity-preserving operation defined as follows
\begin{equation}
(f_1 \square f_2)(x) = \inf_{x_1 + x_2 = x} f_1(x_1) + f_2(x_2)\,.
\end{equation} 
Many optimization problems can be formulated using this notion: not only in economics, where infimal convolutions are  related to the so-called Pareto optimality \cite{hochman1969pareto, stiglitz1987pareto,warr1982pareto}, but also in other fields such as engineering,  physics, social sciences, and mechanics \cite{balle2015computation}.
Infimal convolutions of convex functionals have also been used to regularize inverse problems with applications to image and signal processing. 
One of the first successful approaches in this direction has been proposed in \cite{cl} in the context of image denoising. It is well-known that reconstructions obtained using variational models that penalize the $L^1$-norm of the gradient of the image, such as the ROF model \cite{rof}, achieve good performances to reconstruct sharp edges, but exhibit artifacts in the smooth regions of the image (the so-called staircasing effect). On the other hand, higher-order regularizers tend to smooth out the edges but perform better in the regions where a small gradient is expected. Following these considerations, the authors in \cite{cl} introduced a regularizer defined as the infimal convolution of a first and a second order functional as follows 
\begin{equation}
R(v)=  \inf_{v_1 + v_2 = v} \alpha \int_{\Omega} |\nabla v_1|\, dx + \beta \int_{\Omega} |\nabla (\nabla v_2)|\, dx\,,
\end{equation}
where $\alpha,\beta > 0$ are fixed parameters, with the goal of automatically selecting the regions of the image where it is convenient to apply a first order regularizer and the regions where the second order regularizer is more suitable.
In the following years, further regularizers defined as the infimal convolution of functionals were introduced, achieving striking results in solving inverse problems. 
One classical example is the so-called cartoon-texture decomposition~\cite{meyer2001oscillating, vese2003modeling, aujol2006structure}, where an infimal convolution of the total variation and Meyer's G-norm is used.
Combinations of higher-order functionals were explored improving on the work of \cite{cl}, see \cite{bkp, chan2010fourth, setzer2011infimal}, and spatio-temporal infimal convolutions were used for solving dynamic inverse problems \cite{hk,shsbs}. Other more recent examples can be found in \cite{burger2016infimal_finite_p, burger2015infimal_infty, gao2018infimal, burger2019total, Gao19infimal, Kutyniok13_mh, reggraph, carioni2022extremal}.

In this work we focus on a generalization of the infimal convolution functional \cite[Section 9,
Equation 9.49]{rock:1974} that is performed on an infinite (and even uncountable) collection of convex, non-negative, positively one-homogeneous, and coercive functionals $J(\cdot, s)$ defined on a common dual Banach space $X$ and parametrized by an index $s\in S$, where $S$ is a compact set in the Euclidean space. We additionally suppose that $J$ is lower semicontinuous in $X\times S$, where $X$ is endowed with the weak* topology. We remark that this generalization was first proposed by Rockafellar in \cite[Section 9,
Equation 9.49]{rock:1974} for functionals defined on reflexive Banach spaces.
Given a positive measure $\sigma \in \mathcal{M}_+(S)$ that models which convex functional is ``active" in the infimal convolution, following \cite[Section 9,
Equation 9.49]{rock:1974}, we define the operator $R : X \times \mathcal{M}_+(S) \rightarrow [0,\infty]$ as follows
\begin{align*}
R(v,\sigma) \defeq \inf\left\{  \int_S J(u(s),s)\, d\sigma(s): u \in L^1((S,\sigma);X) \text{ with }\int_S u(s)\, d\sigma(s) = v \right\},
\end{align*}
where $X$ is a given Banach space and $L^1((S,\sigma);X)$ is the Bochner space of $L^1$ functions defined on $(S,\sigma)$ with values in $X$. We call $R$ the \emph{infinite infimal convolution functional} to stress the fact that we are allowing for an infinite family of convex functionals continuously parametrized by $s \in S$ and we search for the optimal decomposition of $v$ with respect to $\sigma$, i.e., $\int_S u(s)\, d\sigma(s) = v$, among all $u\in L^1((S,\sigma);X)$. We remark that, since each $J(\cdot,s)$ is positively one-homogeneous, if $\sigma$ is chosen as a finite sum of  Dirac deltas concentrated in the points $s_1, \ldots, s_N \in S$, then $R(v,\sigma)$ is the standard infimal convolution of the functionals $J(\cdot,s_1) ,\ldots,  J(\cdot,s_N)$. 
Our aim is to use $R(v,\sigma)$ as a regularizer for inverse problems $Av = f$, where $A : X \rightarrow Y$ is a linear measurement operator and $f \in Y$ is the observation. In particular, we solve the Tikhonov regularized problem
\begin{equation}\label{eq:inverseproblemintro}
  \inf_{\substack{v \in X, \\ \sigma\in \mathcal{M}_+(S)}} \frac{1}{2}\|Av - f\|^2_Y + \alpha R(v,\sigma)\,,
\end{equation}
where $\alpha > 0$ is a positive parameter. Note that in \eqref{eq:inverseproblemintro} we are simultaneously optimizing over $v \in X$ and $\sigma\in \mathcal{M}_+(S)$. In this way, \eqref{eq:inverseproblemintro} is automatically selecting the best combination of functionals $J(\cdot, s)$ that optimally regularizes the reconstruction of a given measurement $f$. Our setting is presented in more detail in Section~\ref{sec:prelim}.

We start our analysis by proving well-posedness of \eqref{eq:inverseproblemintro}, which is done in Section~\ref{sec:convexlifting}. Compared to the  classical infimal convolution functional, this is not straightforward. Indeed, since the Bochner space $L^1((S,\sigma);X)$ does not admit a predual and the functionals $J(\cdot,s)$ are only assumed to be coercive, standard methods based on the Banach--Alaoglu theorem cannot be applied. To circumvent this obstacle we leverage the positive one-homogeneity of $J(\cdot,s)$ to construct a convex lifting of the infinite infimal convolution functional to the space of Radon measures with values in $X$, denoted by $\mathcal{M}(S;X)$. Formally, such convex lifting is obtained by setting $\mu = u \sigma \in \mathcal{M}(S;X)$ defined as the Bochner integration of $u \in L^1((S,\sigma);X)$ against the measure $\sigma \in \mathcal{M}_+(S)$. We define the convex lifting of the functional $F_\sigma(u) := \int_S J(u(s),s)\, d\sigma(s)$ as 
$\hat F : \mathcal{M}(S;X) \rightarrow [0,+\infty]$ defined by
\begin{equation}
\hat F(\mu)=  \int_S J\left(\frac{\mu}{|\mu|}(s),s\right)\, d|\mu|(s)\,,
\end{equation}
where $|\mu| \in \mathcal{M}_+(S)$ is the variation (measure) of $\mu \in \mathcal{M}(S;X)$ defined for every Borel set $E \subset S$ as 
\begin{equation*}
|\mu|(E) = \sup \left\{\sum_{i=1}^n  \|\mu(A_i)\|_X: \ n \in \N,\,  (A_i)_{i=1}^n \ \text{partition of } E\right\}
\end{equation*}
and $\frac{\mu}{|\mu|} \in L^1((S,\sigma);X)$ is the density of $\mu$ with respect to $|\mu|$ defined by the relation $\mu(E) = \int_E \frac{\mu}{|\mu|}(s) \, d|\mu|(s)$ for every Borel set $E \subset S$. 
We then rewrite the lifted inverse problem as
\begin{equation}\label{eq:liftingintro}
\inf_{\mu \in \mathcal{M}(S;X)} \frac{1}{2}\|A\mu(S) - f\|_Y^2 + \alpha  \hat F(\mu)\,.
\end{equation}
Relying on the positive one-homogeneity of $J(\cdot, s)$ we show that  \eqref{eq:liftingintro} and \eqref{eq:inverseproblemintro} are equivalent.
Moreover, after proving that $\hat F$ is coercive and weakly-* lower semicontinuous, we obtain the well-posedness of \eqref{eq:liftingintro} and \eqref{eq:inverseproblemintro} using standard arguments of calculus of variations.

The convex lifting of the infinite infimal convolution functional is not only useful to show well-posedness of the original problem, but it also sheds light on the sparsity properties of the regularized inverse problem. The key observation is that the extremal points of the unit ball of $\hat F$, denoted by 
\begin{equation*}
    B = \{\mu \in M(S;X) : \hat F(\mu)\leq 1\},
\end{equation*}
can be characterized as follows:
\begin{itemize}
\item If $0 \in \Ext(B_s)$ for each $s \in S$, then $\Ext(B) = \{v \delta_s : s \in S,\, v \in \Ext(B_s)\}$.  
    \item If there exists $s \in S$ such that $0 \notin \Ext(B_s)$, then $\Ext(B) = \{v \delta_s : s \in S,\, v \in \Ext(B_s) \smallsetminus\{0\}\}$,
\end{itemize}
where $B_s$ is the unit ball of $J(\cdot,s)$ and $\Ext(B_s)$ is the set of its extremal points.
As shown in \cite{bc} and \cite{chambolle}, the extremal points of the unit ball of the regularizer  provide information on the sparse structure of the minimizers. In particular, such extremal points are precisely the atoms used to represent sparse solutions in case of finite-dimensional data. Such results are called representer theorems and have recently become relevant in the context of inverse problems \cite{bc,chambolle}, optimization \cite{unsersplines}, and machine learning  \cite{unser2, bartolucci2021understanding}.
Relying on the result of \cite{bc} and  the characterization of the extremal points of the lifted regularizer,  we prove a representer theorem for inverse problems regularized with the infimal convolution functional in Section \ref{sec:representer}. In particular, we show that if the data is  $Q$-dimensional for $Q \in \N$, then there exists a solution $(v^*, \sigma^*)$ of \eqref{eq:inverseproblemintro} that can be written as  follows
\begin{equation}
v^* = \sum_{i = 1}^M \lambda_i v_i \,, \quad \sigma^* =  \sum_{s \in \{s_i : \,i =1,\ldots,M\}} \bigg\|\sum_{j \in \{1,\ldots, M :\, s_j = s\}}  \lambda_j v_j\bigg\|_X\delta_{s}\,, 
\end{equation}
where $\lambda_i >0$, $s_i \in S$, $v_i \in \Ext(B_{s_i})$, and $M \leq Q$. 

In Section \ref{sec:gcg} we leverage the knowledge of the atoms of the lifted regularization functional to design a generalized conditional gradient method for solving \eqref{eq:inverseproblemintro}. Generalized conditional gradient methods (GCG) are infinite-dimensional generalizations of the classical Frank--Wolfe algorithms \cite{frank1956algorithm}. In the context of inverse problems they have been used to design algorithms that do not require an a priori discretization of the domain $X$ (also called off-the-grid algorithms), since optimization steps are performed by inserting iteratively suitably chosen elements of $X$.  Some applications of generalized conditional gradient methods to  inverse problems can be found, for example, in \cite{bredieslorenz2008, bredies2009}. More recently, GCG methods have been studied as sparse optimization algorithms in infinite-dimensional spaces. Indeed, it has been shown that it is possible to construct the $k$th iterate of a GCG method as the following sparse element 
\begin{align*}
u^k = \sum_{i = 1}^{N_k} \lambda^k_i u^k_i,
\end{align*}
where $N_k \in \N$, $\lambda^k_i >0$ and $u^k_i$ are extremal points of the unit ball of the convex regularizer. 
Such observation has proved useful, for example, in solving super-resolution problems in the space of measures \cite{bp, pieper2021linear, boyd2017alternating, denoyelle2019sliding}, dynamic inverse problems with optimal transport regularization \cite{bredies2022generalized, duval2021dynamical, carioni2022extremal} and for the recovery of superpositions of trajectories \cite{laville2023off} and shapes \cite{de2023towards}. Moreover, they can be seen as a particular case of so-called exchange algorithms \cite{reemtsen1998numerical, flinth2021linear, eftekhari2018bridge, remes1934procede}.

In this paper we use the characterization of the extremal points of the unit ball of $\hat F$ given in Section \ref{sec:convexlifting} to design a GCG for the lifted  problem \eqref{eq:liftingintro}. We follow the general procedure described in \cite{bredies2021linear} and  adapt the results of \cite{bredies2021linear} to prove sublinear convergence of the GCG algorithm. Moreover,  we use the equivalence of  \eqref{eq:inverseproblemintro} and \eqref{eq:liftingintro} to rewrite the GCG algorithm for \eqref{eq:liftingintro} as a converging algorithm for the inverse problem regularized with the infinite infimal convolution functional \eqref{eq:inverseproblemintro} with the same rate of convergence.

Finally, in Section \ref{sec:numerics} we provide examples of signal and image denoising problems regularized with an infinite infimal convolution  functional, where we apply our generalized conditional gradient method. First we consider the task of denoising a signal using infimal convolutions of $L^2$-norms of fractional-Laplacian-type operators with adaptive order $s \in [0,1]$. We refer to \cite{antil2017spectral, bartels2020parameter} for examples of applications of fractional-Laplacian regularizers to image denoising. We construct the regularizer in such a way that the infimal convolution is automatically selecting the predominant frequencies of the data and is assigning a lower-order regularization to them.
As a second application we consider  denoising images using infinite infimal convolutions of $L^2$-norms of anisotropic fractional-Laplacian-type operators. Here the anisotropy direction is regulated by a parameter $s \in [0,\pi]$ providing the direction $(\cos(s), \sin(s)) \in \mathbb{S}^1$. The infinite infimal convolution regularizer learns the directions of higher oscillation of the noisy image $f$ and applies a lower regularization in such directions. In this way the natural anisotropy of the data is preserved and the noise that is supposed to be isotropic is removed. Directional regularizers that are aware of dominant anisotropy directions in the image were proposed in \cite{parisotto2020higher,parisotto2020higher2, kongskov2019directional, gao2018infimal, hk}. 
While the model in \cite{kongskov2019directional} is limited to one dominant direction that is fixed a priori by the user, \cite{gao2018infimal, hk} incorporate multiple directions of anisotropy in the same spirit as the present paper, in the sense that they use an infimal convolution of multiple functionals, each with one dominant direction of anisotropy. But also in those works the bias toward a finite, user-defined number of directions remains.
In \cite{parisotto2020higher,parisotto2020higher2} the anisotropy is encoded in a space-dependent tensor that is either estimated a priori or optimized in the algorithm by an alternating procedure. In contrast, our model is able to automatically select a distribution of directions and does not require a discretization of the parameter space; the set of directions is optimized together with the reconstruction through a provably convergent algorithm.

\medskip

\noindent \textbf{Outline of the main contributions:}

\begin{enumerate}
\item[i)] We introduce the \emph{infinite infimal convolution functional} as a regularizer to solve ill-posed inverse problems defined in a Banach space $X$.
\item[ii)] Under mild assumptions that are suitable for sparse regularization we prove well-posedness of the regularized inverse problem by constructing a convex lifting to the space of measures with values in $X$.
\item[iii)] We use the convex lifting to gain insights into the existence and the structure of sparse solutions by proving a representer theorem in case of finite-dimensional data.
\item[iv)] We propose a generalized conditional gradient method for solving the regularized problem, relying on the sparse structure of the iterates, and prove a sublinear rate of convergence.
\item[v)] We demonstrate the performance of our model and the generalized conditional gradient method on several denoising problems using the infimal convolution of $L^2$-norms of fractional-Laplacian-type operators.
\end{enumerate}

\section{The infinite infimal convolution functional}\label{sec:prelim}


In this section we define the regularizer studied in this paper and state the corresponding regularized inverse problem.

\subsection{Infinite infimal convolution} \label{sec:inf-inf-conv}
Let $X$ be a separable Banach space that admits a predual $\predual{X}$. 
Let $Y$ be a separable Hilbert space and  $A : X \rightarrow Y$ be a weak*-weak continuous operator (that is, $A$ is the adjoint of some linear bounded operator $\preadj{A} \colon Y \to X_*$).  
Let $d \in \N$ and $S \subset \R^d$ be a non-empty compact set. We consider a proper Borel-measurable functional $J: X \times S \rightarrow [0,+\infty]$ satisfying the following  assumptions:

\begin{enumerate}[label=(H\arabic*)]
\item \label{ass:convexlsc} $J(\cdot,s)$ is proper, convex, and positively one-homogeneous for every $s\in S$. Note that by positively one-homogeneous we mean that $J(\lambda v,s) = \lambda J(v,s)$ for every $\lambda \geq 0$, $v\in X$ and $s\in S$.
\item \label{ass:lsc} $J$ is sequentially lower semicontinuous in $X\times S$, where $X$ is endowed with the weak* topology.
\item \label{ass:coercivity} There exists a constant $C>0$ such that 
\begin{equation}
\|v\|_X \leq CJ(v,s)
\end{equation}
for all $v \in X$ and for every $s\in S$. 
\end{enumerate}

\begin{rem}\label{rem:rnholds}
Note that since $X$ is separable, its predual $\predual{X}$ is separable as well. Moreover, thanks to \cref{eq:DP}, $X$ has the Radon--Nikod\'ym-property, see Definition \ref{def:rd}. 
\end{rem}

Denote by $\mathcal{M}_+(S)$ the space of finite positive Radon measures on $S$. Given $\sigma \in \mathcal{M}_+(S)$ we define $F_\sigma : L^1((S,\sigma);X) \rightarrow [0,\infty]$ as follows
\begin{equation}\label{eq:integralfunctional}
F_\sigma(u) := \int_S J(u(s),s)\, d\sigma(s)\,,
\end{equation}
where $L^1((S,\sigma);X)$ is the standard Bochner space of $L^1$-functions with values in the Banach space $X$, cf. Section \ref{app:bochner}. 
Note that the map $s \mapsto J(u(s),s)$ is $\sigma$-measurable for every $\sigma \in \mathcal{M}_+(S)$ and every $u \in L^1((S,\sigma);X)$, due to the measurability of $J$ and the fact that $\sigma$ is a Radon measure.

We are now ready to introduce the infinite infimal convolution functional.
\begin{dfnz}[Infinite infimal convolution] \label{def:infiniteinfimal}
The functional $R: X \times \mathcal{M}_+(S) \rightarrow [0,\infty]$ given by 
\begin{equation}\label{eq:defff}
R(v,\sigma) \defeq \inf\left\{ F_\sigma (u): u \in L^1((S,\sigma);X) \text{ with }\int_S u(s)\, d\sigma(s) = v \right\}
\end{equation}
is called the infinite infimal convolution of the functionals $J(\cdot,s)$ over $s \in S$. 
The constraint $\int_S u(s)\, d\sigma(s) = v$ is defined as a Bochner integral. 
\end{dfnz}


\begin{rem}
One can see that if we choose $\sigma = \sum_{i=1}^N \delta_{s_i}$ and $S = \{s_1, \ldots, s_N\}$, where $N\in \N$ is fixed, the regularizer $R(v,\sigma)$ becomes the infimal convolution of the $N$ functionals $v \mapsto J(v,s_i)$ for $i=1,\ldots,N$. 
\end{rem}

\subsection{Tikhonov regularized inverse problems}\label{sec:inverse}
For a given noisy measurement $f \in Y$ we are interested in inverse problems of the form
\begin{align}\label{eq:inversep}
    Av = f
\end{align}
where $A : X \rightarrow Y$ is a linear continuous operator. 

In this paper we propose to solve \eqref{eq:inversep}
by setting up a Tikhonov problem regularized with the infinite infimal convolution functional. First, for a fixed $\sigma \in \mathcal{M}_+(S)$ we define the variational problem
\begin{equation}\label{eq:proe}
  \inf_{\substack{v \in X}} \frac{1}{2}\|Av - f\|^2_Y + \alpha R(v,\sigma)\,,
\end{equation}
where $\alpha > 0$ is a positive parameter.
Since we are interested in optimizing \eqref{eq:proe} also over $\sigma \in \mathcal{M}_+(S)$, in order to find the best combination of functionals $J(\cdot, s)$ for $s\in S$ that optimally regularizes the reconstruction of a given measurement $f$, we additional minimize \eqref{eq:proe} with respect to $\sigma \in \mathcal{M}_+(S)$ obtaining the following Tikhonov problem  
\begin{equation}\label{eq:inverseproblem}
  \inf_{\substack{v \in X, \\ \sigma\in \mathcal{M}_+(S)}} \frac{1}{2}\|Av - f\|^2_Y + \alpha R(v,\sigma)\,.
\end{equation}
Note that, using the definition of $R(v, \sigma)$, we can rewrite \eqref{eq:inverseproblem}  as
\begin{align}\label{eq:inverseproblem2}
\inf_{\substack{\sigma\in \mathcal{M}_+(S), \\
u \in L^1((S,\sigma);X)  }}   \frac{1}{2}\left\|A\int_S u(s)\, d\sigma(s) - f\right\|^2_Y + \alpha  \int_S J(u(s),s)\, d\sigma(s) \,.
\end{align}
Clearly, \eqref{eq:inverseproblem} and \eqref{eq:inverseproblem2} have the same infimum. Moreover, it is easy to check that if $(u,\sigma)$ is a minimizer for \eqref{eq:inverseproblem2}, then $(\int_S u(s)\, d\sigma(s), \sigma)$ is a minimizer for  \eqref{eq:inverseproblem}. Conversely, if $(v,\sigma)$ is a minimizer for \eqref{eq:inverseproblem} and the infimum in \eqref{eq:defff} is achieved at $u \in L^1((S,\sigma);X)$, then $(u,\sigma)$ is a minimizer for \eqref{eq:inverseproblem2}. It is our goal in the next section to show that these minima are indeed attained.

 The lack of a predual of the space $L^1((S,\mu);X)$ and the relatively weak coercivity assumption  \ref{ass:coercivity}  do not allow one to show existence of minimizers for the infinite infimal convolution $R(v,\sigma)$ as in \cite{rock:1974}. A possible solution is to assume that $\|v\|^p_X \leq CJ(v,s)$ for every $s \in S$, $v \in X$ and $p>1$. However, such additional assumption imposes a severe restriction on the class of regularizers one can consider, excluding, for example, sparsity-promoting regularizers. Our approach relies on the positive one-homogeneity of the functions $v \mapsto J(v,s)$ for each $s \in S$ (Assumption \ref{ass:convexlsc}) and a lifting procedure to a space of $X$-valued measures over $S$. This yields a convex lifting of~\eqref{eq:inverseproblem2} for which we can prove existence of minimizers. 
 We also show that this implies existence of minimizers in~\eqref{eq:inverseproblem} and~\eqref{eq:inverseproblem2}.

\section{Convex lifting to the space of $X$-valued Radon measures}\label{sec:convexlifting}

The goal of this section is to construct a convex lifting of the considered inverse problem to the space of $X$-valued Radon measures $\mathcal{M}(S;X)$. We collect necessary background on vector-valued measures in \cref{sec:preliminaries}.


\subsection{Convex lifting of the regularizer}
We first introduce a convex lifting of the map $(u,\sigma) \mapsto F_\sigma(u)$ defined in~\eqref{eq:integralfunctional} to $\mathcal{M}(S;X)$. 
We define $\hat F : \mathcal{M}(S;X) \rightarrow [0,+\infty]$ as follows
\begin{equation}
\hat F(\mu)=  \int_S J\left(\frac{\mu}{|\mu|}(s),s\right)\, d|\mu|(s)\,,
\end{equation}
where $\abs{\mu} \in \mathcal{M}_+(S)$ is the variation measure of $\mu$ as defined in \eqref{eq:totalvariation}. Note that Remark \ref{rem:rnholds} ensures the validity of the Radon--Nikod\'ym property for $X$. Therefore, for every $\mu \in \mathcal{M}(S;X)$ the density $\frac{\mu}{|\mu|}$ exists and belongs to $L^1((S,|\mu|); X)$, cf. Definition \ref{def:rd}.

\begin{prop}\label{prop:convexityintegral}
The functional $\hat F \colon \mathcal{M}(S;X) \to [0,+\infty]$ is convex, positively one-homo\-geneous and coercive, i.e., there exists $C>0$ such that 
\begin{equation}\label{eq:estcoercivity}
\hat F(\mu) \geq C \|\mu\|_\mathcal{M}\,,
\end{equation}
for all $\mu \in \mathcal{M}(S;X)$. Moreover, given $\mu,\nu \in  \mathcal{M}(S;X)$ such that $|\mu|$ and $|\nu|$ are mutually singular we have that 
\begin{equation}\label{eq:mutuallysingularthesis}
\hat F(\mu+\nu) =  \hat F(\mu) + \hat F(\nu)\,.
\end{equation}
\end{prop}
\begin{proof}
We start with the second statement. Let $\mu,\nu \in \M(S;X)$ be such that
$|\mu|$ and $|\nu|$ are mutually singular.
By mutual singularity, there exists a Borel set $A\subset S$ such that $|\nu|(A) = 0$ and $|\mu|(S \smallsetminus A) = 0$.
We now show that $\frac{\mu + \nu}{|\mu + \nu|} = \frac{\mu}{|\mu|}$ holds $|\mu|$-a.e. in $A$. For every $E \subset A$ Borel we observe that
\begin{equation}\label{eq:dec1}
\int_E \frac{\mu + \nu}{|\mu + \nu|}(s)\, d|\mu + \nu|(s)= \mu(E) + \nu(E) = \mu(E) = \int_E \frac{\mu}{|\mu|}(s)\, d|\mu|(s)
\end{equation}
and
\begin{align}\label{eq:dec2}
|\mu+\nu|(E)  = \sup \left\{\sum_{i=1}^n  \|\mu(A_i)\|_X:\ n \in \N, \  (A_i)_{i=1}^n \ \text{partition of } E\right\} = |\mu|(E)\,.
\end{align}
From \eqref{eq:dec1} and \eqref{eq:dec2} and Corollary 5 in \cite{DiestelUhl} we conclude that $\frac{\mu + \nu}{|\mu + \nu|} = \frac{\mu}{|\mu|}$ for $|\mu|$-a.e. in $A$. Similarly, it can be shown that $\frac{\mu + \nu}{|\mu + \nu|} = \frac{\nu}{|\nu|}$ for $|\nu|$-a.e. in $S \smallsetminus A$. So,
\begin{align*}
\hat F(\mu+\nu) & = \int_S J\left(\frac{\mu+\nu}{|\mu+\nu|}(s),s\right)\, d|\mu+\nu|(s) \\
& = \int_A J\left(\frac{\mu+\nu}{|\mu+\nu|}(s),s\right)\, d|\mu+\nu|(s) + \int_{S\smallsetminus A} J\left(\frac{\mu+\nu}{|\mu+\nu|}(s),s\right)\, d|\mu+\nu|(s)\\
& = \int_A J\left(\frac{\mu}{|\mu|}(s),s\right)\, d|\mu|(s) + \int_{S\smallsetminus A} J\left(\frac{\nu}{|\nu|}(s),s\right)\, d|\nu|(s)= \hat F(\mu) + \hat F(\nu)\,,
\end{align*}
proving \eqref{eq:mutuallysingularthesis}.

Now, it is clear that $\hat F$ is positively one-homogeneous. Indeed, by the definition of $\hat F$ one immediately sees that $\hat F(0) = 0$. Moreover, for every $\lambda >0$,  it holds that 
\begin{align*}
\lambda \int_A  \dfrac{\mu}{|\mu|}(s) d|\mu|(s) = \lambda \mu(A) = \int_A \dfrac{\lambda \mu}{|\lambda \mu|}(s) d|\lambda \mu|(s) =  \lambda \int_A \dfrac{\lambda \mu}{|\lambda \mu|}(s) d|\mu|(s)
\end{align*}
for every Borel set $A \subset S$, implying that 
$\dfrac{\lambda \mu}{|\lambda \mu|}(s) =    \dfrac{\mu}{|\mu|}(s)$ for $|\mu|-a.e.$ $s \in S$. Therefore, 
\begin{equation*}
\hat F(\lambda \mu) = \int_S J\left(\frac{\mu}{|\mu|}(s),s\right)\, d(|\lambda\mu|)(s) = \lambda \hat F( \mu)\,.
\end{equation*}
Since $\hat F$ is positively one-homogeneous, convexity is equivalent to subadditivity, i.e., it is enough to show that $\hat F(\mu + \nu) \leq \hat F(\mu) + \hat F(\nu)$ for all $\mu, \nu \in \mathcal{M}(S;X)$.
Let $\mu, \nu \in \mathcal{M}(S;X)$ be arbitrary. First note that 
\begin{align}\label{eq:de}
    \frac{\mu + \nu}{|\mu| + |\nu|}(s) = \frac{\mu + \nu}{|\mu + \nu|}(s)  \frac{|\mu + \nu|}{|\mu| + |\nu|}(s) \quad \text{for } \ (|\mu| + |\nu|)-a.e. \ s \in S \,,
\end{align}
where all densities are well-defined since $X$ has the Radon--Nikod\'ym property, cf.~Remark \ref{rem:rnholds}. Indeed, since $|\mu + \nu| = \frac{|\mu + \nu|}{|\mu| + |\nu|}(|\mu| + |\nu|)$ 
 by Radon--Nikod\'ym property, we have that \begin{align*}
 \int_E \frac{\mu + \nu}{|\mu| + |\nu|}(s) d (|\mu| + |\nu|)(s)& =  (\mu + \nu)(E) =  \int_E \frac{\mu + \nu}{|\mu +\nu|}(s) \, d(|\mu + \nu|)(s) \\& = \int_E \frac{\mu + \nu}{|\mu + \nu|}(s) \frac{|\mu + \nu|}{|\mu| + |\nu|}(s)  \, d(|\mu| + |\nu|)(s)
\end{align*}
for every $E \subset S$ Borel, showing \eqref{eq:de} thanks to \cite[Corollary 5]{DiestelUhl}. Similarly, it also holds that 
\begin{align}
    & \frac{\mu}{|\mu| + |\nu|}(s) = \frac{\mu}{|\mu|}(s)  \frac{|\mu|}{|\mu| + |\nu|}(s)\quad \text{for } \ (|\mu| + |\nu|)-a.e. \ s \in S \label{eq:over1}\,, \\ 
     & \frac{\nu}{|\mu| + |\nu|}(s) = \frac{\nu}{|\nu|}(s)  \frac{|\nu|}{|\mu| + |\nu|}(s) \quad \text{for } \ (|\mu| + |\nu|)-a.e. \ s \in S \label{eq:over2}\,, \\
  &  \frac{\mu + \nu}{|\mu| + |\nu|}(s) = \frac{\mu}{|\mu| + |\nu|}(s) + \frac{\nu}{|\mu| + |\nu|}(s) \quad \text{for } \ (|\mu| + |\nu|)-a.e. \ s \in S \label{eq:ads} \,.
\end{align}
Therefore, 
\begin{align*}
    \hat F & (\mu + \nu)  = \int_S J \left(\frac{\mu + \nu}{|\mu + \nu|}(s),s\right) \, d |\mu + \nu|(s) = \int_S J \left(\frac{\mu + \nu}{|\mu + \nu|}(s),s\right) \frac{|\mu + \nu|}{|\mu| + |\nu|}(s)\, d (|\mu| + |\nu|)(s) \\
    & = \int_S J \left(\frac{\mu + \nu}{|\mu| + |\nu|}(s),s\right) \, d (|\mu| + |\nu|)(s) \\
    & \leq \int_S J \left(\frac{\mu}{|\mu| + |\nu|}(s),s\right) \, d (|\mu| + |\nu|)(s) + \int_S J \left(\frac{\nu}{|\mu| + |\nu|}(s),s\right) \, d (|\mu| + |\nu|)(s)\\
    & = \int_S J \left(\frac{\mu}{|\mu|}(s),s\right) \frac{|\mu|}{|\mu| + |\nu|}(s)\, d (|\mu| + |\nu|)(s) + \int_S J \left(\frac{\nu}{|\nu|}(s),s\right) \frac{|\nu|}{|\mu| + |\nu|}(s) \, d (|\mu| + |\nu|)(s)\\
    & = \hat F(\mu) + \hat F(\nu)\,,
\end{align*}
where in the second equality we applied Radon--Nikod\'ym's theorem and in the third equality we used \eqref{eq:de} and the positive one-homogeneity of $J(\cdot,s)$, noting that $\frac{|\mu + \nu|}{|\mu| + |\nu|}$ is non-negative for $(|\mu| + |\nu|)$-a.e. $s\in S$; additionally, in the first inequality we used the subadditivity of $J(\cdot,s)$ and \eqref{eq:ads}, and in fourth equality we used \eqref{eq:over1}, \eqref{eq:over2} and the positive one-homogeneity of $J(\cdot,s)$, noting that both $\frac{|\mu|}{|\mu| + |\nu|}$ and $\frac{|\nu|}{|\mu| + |\nu|}$ are non-negative for $(|\mu| + |\nu|)$-a.e. $s\in S$; finally the last equality follows again from Radon--Nikod\'ym's theorem.

To show coercivity \eqref{eq:estcoercivity}, we use Assumption~\ref{ass:coercivity} and \cref{prop:totalvariationproduct}, and obtain 
\begin{equation*}
\hat F(\mu)=  \int_S J\left(\frac{\mu}{|\mu|}(s),s\right)\, d|\mu|(s) \geq \int_S \frac1C \left\|\frac{\mu}{|\mu|}(s)\right\|_{X} \, d|\mu|(s) = \frac1C \|\mu\|_{\mathcal{M}}\,,
\end{equation*}
where $C$ is the constant from Assumption~\ref{ass:coercivity}.
\end{proof}

Next we establish weak* lower semicontinuity of the functional $\hat F \colon \M(S;X) \to [0,+\infty]$. We remind the reader that the weak* topology on the space $\mathcal{M}(S;X)$ is understood in the sense of  duality with  $C(S;\predual{X})$, see \cref{thm:duality}.
The proof uses the same techniques as in the proof of the classical Reshetnyak semicontinuity theorem, see \cite{Reshetnyak} and \cite[Theorem 2.38]{afp}, adapted to measures with values in Banach spaces. The same result is proven in \cite[Theorem 8.2.2]{castaing} for  a separable and reflexive  space $X$. For completeness, we provide a complete proof.

\begin{thm}\label{thm:lsc}
The functional $\hat F$ is weakly-* sequentially \lsc{} in $\M(S;X)$.
\end{thm}
\begin{proof}
Let $\mu^n \wsto \mu$ in $\M(S;X)$, i.e., for any $\phi \in C(S;\predual{X})$,
\begin{equation*}
	\sp{\mu^n,\phi} \to \sp{\mu,\phi}.
\end{equation*}
By the Radon--Nikod\'ym theorem, there exists a sequence $\{u^n\}_n$ of Borel measurable functions such that $u^n \in L^1((S,|\mu^n|);X)$ and an element $u \in L^1((S,|\mu|);X)$  such that $\mu^n = u^n |\mu^n|$ for all $n$ and $\mu = u |\mu|$. By Proposition \ref{prop:totalvariationproduct} we have that $\|u^n(s)\|_X = 1$ for $|\mu^n|$-almost every $s \in S$.
Since weakly-* convergent sequences are bounded, we have $\sup_n \|\mu^n\|_{\mathcal{M}} <\infty$.

Denote by $B_X := \{u \in X : \|u\|_X \leq 1\}$ the unit ball in $X$. By the Banach--Alaoglu theorem, $B_X$ is weakly-* compact and thanks to the separability of $X$, the weak* topology is metrizable. Since $S$ is compact, the metric space $B_X \times S$ (where $B_X$ is endowed with the metric that metrizes its weak* topology) is compact and therefore separable. 
Let the sequence $\{\nu^n\}_{n}$ in $\mathcal{M}(B_X \times S)$ be defined as 
\begin{align}\label{eq:disin}
    \int_{B_X \times S} \varphi(x,s) \, d\nu^n(x,s) = \int_{S} \varphi(u^n(s),s)\, d|\mu^n|(s)\, \quad \forall \varphi \in C(B_X \times S)\,.
\end{align}
Since, using \eqref{eq:disin}, we have that $\sup_n \|\nu^n\|_\mathcal{M} = \sup_n \|\mu^n\|_{\mathcal{M}} <\infty$, the sequence $\{\nu^n\}_n$ is bounded and therefore contains a weakly-* converging subsequence (which we do not relabel) 
\begin{equation*}
    \nu^n \wsto \nu \in \M(B_X \times S)\,.
\end{equation*}
%
Let $\pi : X\times S \rightarrow S$ be the projection on $S$. Then we have $\pi_\# \nu^n  = |\mu^n|$ and $|\mu^n| \weakstar \pi_\# \nu$ by the continuity of the push-forward with respect to weak* convergence. Moreover, defining $\lambda := \pi_\# \nu$ and applying \cite[Proposition 1.62]{afp}, we get that $\lambda \geq |\mu|$. Applying the disintegration theorem for measures defined on metric spaces (see for example \cite[Theorem A.4]{superpositioninhomogeneous}),
there exists a Borel family of measures $\{\nu_s\}_{s \in S}$ in $\mathcal{M}(B_X)$ such that for every $f \in L^1(B_X \times S, \nu)$ 
\begin{align}\label{eq:disi}
\int_{B_X\times S} f(x,s)\, d\nu(x,s) = \int_{S} \int_{B_X} f(x,s)\, d\nu_s(x)\, d\lambda(s)\,.
\end{align}
We now analyze the barycenter  of $\nu_s$ in $B_X$ for every $s \in S$, that is \begin{align*}
\text{bar}(\nu_s) := \int_{B_X} x  \ \, d \nu_s(x)\,,
\end{align*}
where the previous integral is well-defined as a Bochner integral since the identity map is continuous on the separable metric space $B_X$ and hence, is in $L^1(B_X, \nu_s)$ for every $s\in S$.

Let us define $f:(x,s) \mapsto \langle \psi(s) , x\rangle \in C(B_X \times S)$, where $\psi \in C(S; \predual{X})$. 
Using \eqref{eq:disi} and  standard properties of the Bochner integral we get the following equality for the sequences $\nu^n \weakstar \nu$ and $\mu^n \weakstar \mu:$
\begin{align*}
\int_S \langle \psi(s),   \int_{B_X}  & x \, d\nu_s(x)\rangle\, d\lambda(s)   = \int_S \int_{B_X}  \langle \psi(s) , x\rangle \, d\nu_s(x)\, d\lambda(s)   =  \int_{ B_X \times S}\langle \psi(s) , x\rangle\, d\nu(x,s) \\
& = \lim_{n \rightarrow +\infty}   \int_{B_X \times S} \langle \psi(s) , x\rangle\, d\nu^n(x,s)  = \lim_{n \rightarrow +\infty}   \int_{S} \langle \psi(s) , u^n(s)\rangle\, d|\mu^n|(s)  \\
& = \lim_{n \rightarrow +\infty}  \int_{S}\psi(s)\, d\mu^n(s) = \int_{S}\psi(s)\, d\mu(s) = \int_{S}\psi(s) u(s) \frac{|\mu|}{\lambda}(s) \, d\lambda(s)\,.
\end{align*}
In particular, it holds that
\begin{align}\label{eq:bar}
\text{bar} (\nu_s) = u(s) \frac{|\mu|}{\lambda}(s) \qquad \text{for }\lambda - a.e. \ s\in S\,.
\end{align}
Therefore, by \cite[Proposition A.3]{superpositioninhomogeneous} (see also \cite[Section 5.1.7]{ags}) and the lower semicontinuity of $J$ on $B_X\times S$ (see Assumption \ref{ass:lsc}) we have
\begin{align*}
\liminf_{n \rightarrow +\infty} & \int_S J(u^n(s),s)\, d|\mu^n|(s)   = \liminf_{n \rightarrow +\infty} \int_{B_X \times S} J(x,s)\, d\nu^n(x,s) \geq \int_{B_X \times S} J(x,s)\, d\nu(x,s)\\
& = \int_{S} \int_{B_X} J(x,s)\, d\nu_s(x)\, d\lambda(s) \geq \int_{S} J\left(u(s)\frac{|\mu|}{\lambda}(s),s\right)\, d\lambda(s)= \int_{S} J(u(s),s)\, d|\mu|(s)\,,
\end{align*}
where we used \eqref{eq:disi}, \eqref{eq:bar}, Jensen's inequality (see Assumption \ref{ass:convexlsc}). The proof is complete.
\end{proof}

We conclude the section showing that $\hat F(\mu)$ coincides with $\int_S J(u(s),s)  \, d\sigma(s)$ when computed on measures $\mu \in \mathcal{M}(S; X)$ of the form $\mu(E) = \int_E u(s)\, d\sigma(s)$ for $E \subset S$ Borel, where $\sigma \in \mathcal{M}_+(S)$ and $u \in L^1((S,\sigma); X)$. From now on we will denote the measure $\mu\in \mathcal{M}(S; X)$ constructed as above by $\mu = u \sigma$. We refer the reader to Section \ref{sec:rrd} for more details.

\begin{lemma}\label{lem:ins}
Let $\sigma \in \mathcal{M}_+(S)$ and $u \in L^1((S,\sigma); X)$. Then, 
it holds that
\begin{equation}\label{eq:equivalence}
\hat F(u\sigma)  =  \int_S J(u(s),s)  \, d\sigma(s)\,.
\end{equation}  
\end{lemma}

\begin{proof}
Using  \cref{prop:totalvariationproduct}, the fact that $\|u\|_X \neq 0$ $|\mu|$-a.e., and the positive one-homogeneity of $J(\cdot,s)$  for  every $s\in S$ (see Assumption \ref{ass:convexlsc}), we get
\begin{align*}
\hat F(u\sigma) & =  \int_S J \left(\frac{u(s)}{\|u(s)\|_X},s\right) \|u(s)\|_X \, d\sigma(s) =  \int_S J(u(s),s)  \, d\sigma(s) 
\end{align*}
as we wanted to prove.
\end{proof}
%
%

\subsection{Convex lifting of the variational problem}

We are now ready to lift the problem \eqref{eq:inverseproblem} (and therefore also \eqref{eq:inverseproblem2}) to the space of $X$-valued measures $\mathcal{M}(S;X)$. Consider
\begin{equation}\label{eq:lifting}
\inf_{\mu \in \mathcal{M}(S;X)} \frac{1}{2}\|A\mu(S) - f\|_Y^2 + \alpha  \int_S J\left(\frac{\mu}{|\mu|}(s),s\right)\, d|\mu|(s)\,.
\end{equation}

The following result shows the equivalence between the Tikhonov regularized inverse problem \eqref{eq:inverseproblem2} and the lifted problem \eqref{eq:lifting}.
\begin{prop}\label{prop:equivalence}
Problems \eqref{eq:inverseproblem2} and \eqref{eq:lifting} have the same infimum. Moreover, if the pair $(\bar u,\bar \sigma)$ minimizes \eqref{eq:inverseproblem2}, then $\bar \mu := \bar u \bar \sigma$
is a minimizer for \eqref{eq:lifting}. Conversely, if  $\bar \mu$ is a minimizer for \eqref{eq:lifting}, then the pair $\left(\frac{\bar \mu}{|\bar \mu|}, |\bar \mu|\right)$ is a minimizer for \eqref{eq:inverseproblem2} and $(\bar \mu(S), |\bar \mu|)$ is a minimizer for \eqref{eq:inverseproblem}.
\end{prop}

\begin{proof}
Denote by $(\bar u,\bar \sigma)$ a pair minimizing \eqref{eq:inverseproblem2} and define $\bar \mu \in \mathcal{M}(S;X)$ as $\bar \mu = \bar u \bar \sigma$.  Given $\mu \in \mathcal{M}(S;X)$, since $X$ satisfies the Radon--Nikod\'ym property, we have that $\mu = \frac{\mu}{|\mu|} |\mu|$.
Therefore, using \cref{lem:ins} and the optimality of $(\bar u, \sigma)$ for \eqref{eq:inverseproblem2} we have
\begin{align*}
\frac{1}{2}\|A\mu(S) - f\|_Y^2 & + \alpha  \int_S J\left(\frac{\mu}{|\mu|}(s),s\right)\, d|\mu|(s)  \\
& = \frac{1}{2}\left\|A \int_S\frac{\mu}{|\mu|}(s)\, d|\mu|(s) - f\right\|_Y^2  + \alpha  \int_S J\left(\frac{\mu}{|\mu|}(s),s\right)\, d|\mu|(s) \\
 & \geq  \frac{1}{2}\left\|A\int_S\bar u(s)\, d\bar \sigma(s) - f\right\|_Y^2 + \alpha  \int_S J(\bar u(s),s)\, d\bar \sigma(s) \\
&  =  \frac{1}{2}\|A\bar \mu(S) - f\|_Y^2 + \alpha  \int_S J\left(\frac{\bar \mu}{|\bar \mu|}(s),s\right)\, d|\bar \mu|(s)  \,,
\end{align*}
implying that $\bar \mu =  \bar u \bar \sigma$ is a minimizer of \eqref{eq:lifting}.

Conversely, let  $\bar \mu$ be a minimizer for \eqref{eq:lifting}. Consider  $\sigma \in \mathcal{M}_+(S)$,  $u \in L^1((S,\sigma); X)$, and define the measure $\mu \in \mathcal{M}(S;X)$ as $\mu = u \sigma$. Setting
$(\bar u, \bar \sigma) = \left(\frac{\bar \mu}{|\bar \mu|} , |\bar \mu|\right)$ we have $\bar \mu = \bar u \bar \sigma$, and using \cref{lem:ins} and the optimality of $\bar \mu$ for \eqref{eq:lifting} we deduce that
%
\begin{align*}
 \frac{1}{2}\left\|A\int_S u(s)\, d\sigma(s) - f\right\|^2  + \alpha  \int_S J(u(s),s)\, d\sigma(s)  & = \frac{1}{2}\|A\mu(S) - f\|_Y^2  + \alpha  \int_S J\left(\frac{ \mu}{|\mu|}(s),s\right)\, d| \mu|(s) \\ 
&\geq \frac{1}{2}\|A\bar \mu(S) - f\|_Y^2  + \alpha  \int_S J\left(\frac{\bar \mu}{|\bar \mu|}(s),s\right)\, d| \bar \mu|(s) \\
& = \frac{1}{2}\left\|A\int_S \bar u(s)\, d\bar \sigma(s) - f\right\|_Y^2  + \alpha  \int_S J(\bar u(s),s)\, d \bar \sigma(s)\,,
\end{align*}
implying that $\left(\frac{\bar \mu}{|\bar \mu|}, |\bar \mu|\right)$ is a minimizer for \eqref{eq:inverseproblem2} and therefore the pair $(\bar \mu(S), |\bar \mu|)$ is a minimizer for \eqref{eq:inverseproblem}.

Finally, we show that \eqref{eq:inverseproblem2} and \eqref{eq:lifting} have the same infimum. Given any $\mu \in \mathcal{M}(S;X)$ define 
 $\tilde \sigma := |\mu| \in \mathcal{M}_+(S)$ and $\tilde u := \frac{\mu}{|\mu|}\in L^1((S,\tilde \sigma); X)$. Then, reasoning as above, we have that
\begin{align*}
   \inf_{\substack{\sigma\in \mathcal{M}_+(S),\\
   u \in L^1((S,\sigma);X)}}  & \frac{1}{2}\left\|A\int_S u(s)\, d\sigma(s) - f\right\|^2_Y  + \alpha  \int_S J(u(s),s)\, d\sigma(s) \\
   & \leq \frac{1}{2}\left\|A\int_S \tilde u(s)\, d\tilde \sigma(s) - f\right\|^2 + \alpha \int_S J(\tilde u(s),s)\, d\tilde \sigma(s) \\
   &= \frac{1}{2}\|A\mu(S) - f\|_Y^2  + \alpha  \int_S J\left(\frac{ \mu}{|\mu|}(s),s\right)\, d| \mu|(s)\,,
\end{align*}
implying that the infimum of \eqref{eq:inverseproblem2} is less or equal than the infimum of \eqref{eq:lifting}. Similarly, given any $\sigma \in \mathcal{M}_+(S)$ and $u \in L^1((S,\sigma); X)$, and defining $\tilde \mu := u \sigma \in \mathcal{M}(S;X)$, it holds  
\begin{align*}
  \inf_{\mu \in \mathcal{M}(S;X)} \frac{1}{2}\|A\mu(S) - f\|_Y^2 & + \alpha  \int_S J\left(\frac{\mu}{|\mu|},s\right)\, d|\mu|(s) \leq \frac{1}{2}\|A\tilde \mu(S) - f\|_Y^2  + \alpha  \int_S J\left(\frac{ \tilde \mu}{|\tilde \mu|}(s),s\right)\, d| \tilde \mu|(s) \\
   & = \frac{1}{2}\left\|A \int_S u(s)\, d\sigma(s)  - f\right\|_Y^2  + \alpha  \int_S J\left(u(s),s\right)\, d\sigma(s)\,,
\end{align*}
implying that the infimum of \eqref{eq:lifting} is less or equal than the infimum of \eqref{eq:inverseproblem2} and thus concluding the proof.
\end{proof}
\begin{rem}
We recall that Problems \eqref{eq:inverseproblem} and \eqref{eq:inverseproblem2} are equivalent as discussed in Section \ref{sec:inverse}. Therefore, Proposition \ref{prop:equivalence} implies that \eqref{eq:inverseproblem}, \eqref{eq:inverseproblem2} and \eqref{eq:lifting} have the same infimum and, given a minimizer of one of these problems, one can construct explicitly a minimizer of the others. In the next section we make sure that such minimizers actually exist.
\end{rem}

%

\subsection{Existence of minimizers}

 We are now ready to show that \eqref{eq:inverseproblem} and  \eqref{eq:inverseproblem2} admit a minimizer. First we show that the lifted problem \eqref{eq:lifting} admits a solution and then thanks to \cref{prop:equivalence} we immediately infer that \eqref{eq:inverseproblem2} and \eqref{eq:inverseproblem} admit minimizers as well. 
 
 \begin{prop}\label{prop:existencelifting}
 There exists $\mu \in \mathcal{M}(S;X)$ that minimizes the lifted problem \eqref{eq:lifting}.
 \end{prop}

\begin{proof}
Consider a minimizing sequence $\{\mu^n\}_n$ in  $\mathcal{M}(S;X)$ for  \eqref{eq:lifting}. Then, thanks to \cref{prop:convexityintegral}, the sequence  $\{\mu^n\}_n$ is uniformly bounded in total variation. Hence, by the Banach--Alaoglu theorem and \cref{thm:duality}, there exists $\mu \in \mathcal{M}(S;X)$ such that, up to a subsequence, $\mu^n \weakstar \mu$. 
Thanks to the weak*-to-weak continuity of $A$, applying Remark \ref{rem:consweakstar} and the sequential weak* lower semicontinuity of $\hat F$ provided by \cref{thm:lsc}, we conclude that $\mu$ is a minimizer of~\eqref{eq:lifting}. 
\end{proof}

\begin{cor}
There exists a pair $(u,\sigma)$ with $\sigma \in \mathcal{M}_+(S)$ and $u \in L^1((S,\sigma); X)$ that minimizes \eqref{eq:inverseproblem2}. Consequently, the pair $(\int_S u(s)\, d\sigma(s), \sigma) \in X \times \mathcal{M}_+(S)$ minimizes \eqref{eq:inverseproblem}.
\end{cor}

\begin{proof}
The proof follows directly from \cref{prop:equivalence} and \cref{prop:existencelifting}.
\end{proof}

\subsection{Extremal points of the lifted regularizer}
We conclude this section by characterizing the set of extremal points of the unit ball of $\hat F$, that is 
\begin{equation}\label{eq:ball}
B := \{\mu \in \mathcal{M}(S;X): \hat F(\mu)\leq 1\}\,,
\end{equation}
in terms of the extremal points of 
\begin{equation}
B_s := \{v \in X: J(v,s) \leq 1\}
\end{equation}
for $s \in S$. 
This result will be fundamental in Section \ref{sec:representer} to prove a representer theorem for inverse problems regularized with the infinite infimal convolution functional and in  Section \ref{sec:gcg} to set up a generalized conditional gradient method for its numerical solution. 
Note that the novelty of Theorem~\ref{thm:extremal} is that it covers infinite infimal convolution functionals. The finite case is studied in~\cite{carioni2022extremal, iglesias2021extremal}.

\begin{thm}\label{thm:extremal}
The extremal points of $B$ can be characterized as follows:
\begin{itemize}
    \item If $0 \in \Ext(B_r)$ for each $r \in S$, then $\Ext(B) = \{v \delta_r : r \in S,\, v \in \Ext(B_r)\}$.  
    \item If there exists $r \in S$ such that $0 \notin \Ext(B_r)$, then $\Ext(B) = \{v \delta_r : r \in S,\, v \in \Ext(B_r) \smallsetminus\{0\}\}$.
\end{itemize}

\end{thm}

\begin{proof}
Note that it is enough to prove the statement for $\alpha = 1$, since, due the positive one-homogeneity of $J(\cdot,s)$ for every $s \in S$ and of $\hat F$, the general case follows from a scaling argument.
For every $r \in S$ and $v \in B_r$, using the one-homogeneity of $J$, we have 
\begin{equation}
\hat F(v\delta_r) = J\left(\frac{v}{\norm{v}_X},r\right)\norm{v}_X = J(v,r)\,,
\end{equation}
which shows that $v \delta_r \in B$.

First, we show that $\Ext(B) \subset \{v \delta_{r} : r \in S, \,v \in \Ext(B_r)\}$ and in case there exists $\hat r \in S$ such that $0 \notin \Ext(B_{\hat r})$ it also holds that $\Ext(B) \subset \{v \delta_{r} : r \in S, \,v \in \Ext(B_r)\smallsetminus \{0\}\}$. Note that, given $\mu \in \Ext(B)$, either $\hat F (\mu) = 0$ or $\hat F(\mu) = 1$, since $\hat F$ is positively one-homogeneous. 
Suppose that $\hat F (\mu) = 0$, so that $\mu = 0$ by Proposition \ref{prop:convexityintegral}. If $0 \in \Ext(B_r)$ for each $r \in S$, then we immediately have that $0 \in \{v \delta_{r} : r \in S, \,v \in \Ext(B_r)\}$ as we wanted to prove. Suppose instead that there exists $\hat r \in S$ such that $0 \notin \Ext(B_{\hat r})$. Then,
\begin{align}
0 = \lambda v_1 + (1 - \lambda) v_2
\end{align}
for $v_1,v_2 \in B_{\hat r}$, $v_1 \neq v_2$ and $0<\lambda<1$. Therefore the convex decomposition $0 = \lambda v_1 \delta_{\hat r} + (1 - \lambda) v_2 \delta_{\hat r}$ 
shows that $0$ is not an extremal point of $B$ since $v_1 \delta_{\hat r}, v_2 \delta_{\hat r} \in B$.

Suppose now that $\hat F (\mu) = 1$. Note that, since $\hat F (\mu) = 1$ implies that $\mu \neq 0$, it is enough to show that $ \mu \in \{v \delta_{r} : r \in S, \,v \in \Ext(B_r)\}$. We start by showing that $\mu$ is supported on a singleton. If this does not hold, then there exists a Borel set $A \subset S$ such that $0< \mu(A) < \mu(S)$. Define the constants
\begin{equation*}
C(A) \defeq \int_A J\left(\frac{\mu}{|\mu|}(s),s\right)\, d|\mu|(s) \quad \text{and} \quad C(S\smallsetminus A) \defeq \int_{S\smallsetminus A} J\left(\frac{\mu}{|\mu|}(s),s\right)\, d|\mu|(s) \,.
\end{equation*}
Since $\|v\|_X \leq C J(v,s)$ for each $v \in X$, $s\in S$ by \ref{ass:coercivity}, there exists a  constant $c>0$ such that $C(A) \geq c |\mu|(A) > 0$ and $C(S\smallsetminus A) \geq c |\mu|(S\smallsetminus A) > 0$. Moreover, since $\hat F(\mu) =1$, it holds that $C(A) + C(S\smallsetminus A) = 1$ and in particular, they are both finite.
Defining then the measures
\begin{equation*}
\mu_1 := \frac{\mu \res A}{C(A)} \quad \text{ and } \quad \mu_2 := \frac{\mu \res (S\smallsetminus A)}{C(S\smallsetminus A)},
\end{equation*}
we notice that
\begin{equation*}
|\mu_1| = \frac{|\mu| \res A}{C(A)}, \quad \frac{\mu_1}{|\mu_1|} =  \frac{\mu}{|\mu|}\ \  |\mu_1|-a.e.\,,\quad
|\mu_2| = \frac{|\mu| \res (S \smallsetminus A)}{C(S\smallsetminus A)}, \quad  \frac{\mu_2}{|\mu_2|} =  \frac{\mu}{|\mu|} \ \  |\mu_2|-a.e. \,.
\end{equation*}
In particular, we have that $\hat F(\mu_1) = \hat F(\mu_2) = 1$
and thus
\begin{equation*}
\mu = C(A) \mu_1 + C(S\smallsetminus A) \mu_2
\end{equation*}
gives a non-trivial decomposition, which is a contradiction to $\mu \in \Ext(B)$.
Therefore, $\mu$ is supported on a singleton and thus, $\mu = v \delta_r$ for some $v \in X$ and $r \in S$. We now show that $v \in \Ext(B_r)$. If this does not hold, then there exist $0 < \lambda <1$ and $v_1,v_2 \in B_r$ such that  $v_1 \neq v_2$ and
\begin{equation}
v = \lambda v_1 + (1-\lambda) v_2\,.
\end{equation}
Multiplying the previous equation by $\delta_r$ we obtain that
$\mu = \lambda v_1 \delta_{r} + (1- \lambda) v_2 \delta_{r}$
and, since $v_1 \delta_r, v_2 \delta_r \in B$,
we conclude that $\mu$ is not an extremal point of $B$, which is a contradiction.

We now show that $\Ext(B) \supset \{v \delta_{r} : r \in S,\  v \in \Ext(B_r) \smallsetminus \{0\}\}$. Consider $r\in S$, $v \in \Ext(B_r)\smallsetminus \{0\}$, $0 < \lambda < 1$ and $\mu_1,\mu_2 \in B$ such that 
\begin{equation}\label{eq:cv}
v\delta_r = \lambda \mu_1 + (1-\lambda) \mu_2\,.
\end{equation}
Note that, since $v \neq 0$ and $J(\cdot,r)$ is coercive (Assumption \ref{ass:coercivity}) and positively one-homogeneous, we have $J(v,r) = 1$, as otherwise $v \notin \Ext(B_r)$. Therefore
\begin{align}\label{eq:cont}
    1  = J(v,r) = \hat F(v \delta_r) & = \hat F([\lambda \mu_1 + (1-\lambda) \mu_2] \res \{r\}) + \hat F([\lambda \mu_1 + (1-\lambda) \mu_2] \res (S \smallsetminus \{r\})) \nonumber\\
    & \leq \lambda \hat F(\mu_1 \res \{r\}) + (1-\lambda)  \hat F(\mu_2 \res \{r\}) \leq 1\,,
\end{align}
where in the third equality we apply \eqref{eq:mutuallysingularthesis} in Proposition \ref{prop:convexityintegral}, in the first inequality we use that $\hat F(0) = 0$ as well as the convexity of $\hat F$, and in the last inequality we use that $\hat F(\mu_j \res \{r\}) \leq\hat F(\mu_j)$ for $j=1,2$, that is true due to  \eqref{eq:mutuallysingularthesis}, since $\hat F \geq 0$. Note now that $\mu_j \res \{r\} = v_j \delta_r$ with $v_j \in B_r$ for $j=1,2$, and $\hat F(v_j \delta_r) = J(v_j,r) =1$ for $j=1,2$, since otherwise we would contradict \eqref{eq:cont}. Moreover, thanks to the estimate $1 \leq \hat F(\mu_j \res (S\smallsetminus \{r\})) + \hat F(v_j \delta_r) = \hat F(\mu_j) \leq 1$, we deduce that $\hat F(\mu_j \res (S\smallsetminus \{r\})) = 0$ and thus $\mu_j \res (S\smallsetminus \{r\}) = 0$ for $j=1,2$, due to \eqref{eq:estcoercivity} in Proposition \ref{prop:convexityintegral}. In particular, we obtain $\mu_j = v_j \delta_r$ for $j=1,2$.
Finally, substituting $\mu_j = v_j \delta_r$ for $j=1,2$ in \eqref{eq:cv} we deduce that $v = \lambda v_1 + (1-\lambda)v_2$, which implies that $v_1 = v_2 = v$, due to the extremality of $v$. We then conclude that $\mu_1 = \mu_2$, implying that $v\delta_r \in \Ext(B)$.

We conclude the proof by showing that if $0 \in \Ext(B_r)$ for every $r\in S$, then $0\in \Ext(B)$.
If $0 = \lambda \mu_1 + (1-\lambda) \mu_2$ with $0<\lambda<1$ and $\mu_1,\mu_2 \in B$, note that $\mu_j \ll |\mu_1| + |\mu_2|$ for $j=1,2$. Therefore, applying the Radon–Nikod\'ym property of $X$, cf. Definition \ref{def:rd}, we have that for $j=1,2$ it holds
\begin{align*}
   1 \geq \hat F(\mu_j) & = \int_S J\left(\frac{\mu_j}{|\mu_j|}(s),s\right)\, d|\mu_j|(s) = \int_S J\left(\frac{\mu_j}{|\mu_j|}(s),s\right)\frac{|\mu_j|}{|\mu_1| + |\mu_2|}(s)\, d(|\mu_1| + |\mu_2|)(s) \\
    & = \int_S J\left(\frac{\mu_j}{|\mu_1| + |\mu_2|}(s),s\right)\, d(|\mu_1| + |\mu_2|)(s)  \,,  
\end{align*}
where we have also used the positive one-homogeneity of $J(\cdot,s)$ and \eqref{eq:over1}, \eqref{eq:over2} with $\mu_1 = \mu$ and $\mu_2 = \nu$. In particular,  $J\left(\frac{\mu_j}{|\mu_1| + |\mu_2|}(s),s\right) < +\infty$ for $(|\mu_1| + |\mu_2|)$-a.e. $s \in S$ and $j=1,2$. Now, from the convex decomposition $0 = \lambda \mu_1 + (1-\lambda) \mu_2$, by applying the Radon–Nikod\'ym property of $X$ we have that $0 = \left(\lambda \frac{\mu_1}{|\mu_1| + |\mu_2|} + (1 - \lambda) \frac{\mu_2}{|\mu_1| + |\mu_2|}\right)(|\mu_1| + |\mu_2|)$
and thus 
\begin{align}\label{eq:conx}
0 = \lambda \frac{\mu_1}{|\mu_1| + |\mu_2|}(r) + (1 - \lambda) \frac{\mu_2}{|\mu_1| + |\mu_2|}(r) \quad  \text{for} \ \ (|\mu_1| + |\mu_2|)-a.e. \ \, r \in S\,.     
\end{align}
Define $c(r) = J\left(\frac{\mu_1}{|\mu_1| + |\mu_2|}(r),r\right) + J\left(\frac{\mu_2}{|\mu_1| + |\mu_2|}(r),r\right) < \infty$ for $(|\mu_1| + |\mu_2|)$-a.e. $r \in S$. By Assumption \ref{ass:coercivity}, it holds that $c(r) >0$ for $(|\mu_1| + |\mu_2|)$-a.e. $r \in S$ such that $\frac{\mu_1}{|\mu_1| + |\mu_2|}(r) \neq 0$ or $\frac{\mu_2}{|\mu_1| + |\mu_2|}(r)  \neq 0$. Therefore, up to setting $c(r)$ equal to an arbitrary positive number on the points such that $\frac{\mu_j}{|\mu_1| + |\mu_2|}(r) = 0$ for $j=1,2$, we obtain from \eqref{eq:conx} that
\begin{align}
    0 = \lambda \frac{1}{c(r)}\frac{\mu_1}{|\mu_1| + |\mu_2|}(r) + (1 - \lambda) \frac{1}{c(r)}\frac{\mu_2}{|\mu_1| + |\mu_2|}(r) \quad \text{for}\  (|\mu_1| + |\mu_2|)-a.e. \ \,r \in S\,.
\end{align}
Thanks to the positive one-homogeneity of $J(\cdot,r)$, we have that $\frac{1}{c(r)}\frac{\mu_j}{|\mu_1| + |\mu_2|}(r) \in B_r$ for $(|\mu_1| + |\mu_2|)$-a.e. $r \in S$, $j=1,2$. Therefore, since $0 \in \Ext(B_r)$
for every $r \in S$, it holds that $0 = \frac{\mu_1}{|\mu_1| + |\mu_2|}(r) = \frac{\mu_2}{|\mu_1| + |\mu_2|}(r)$ for $(|\mu_1| + |\mu_2|)$-a.e. $r \in S$ and thus $\mu_j = \frac{\mu_j}{|\mu_1| + |\mu_2|} (|\mu_1| + |\mu_2|) = 0$ for $j=1,2$ as we wanted to prove. 
\end{proof}
In the following sections, we will apply Theorem \ref{thm:extremal} to the rescaled balls 
\begin{equation}\label{eq:ball}
B^\alpha := \{\mu \in \mathcal{M}(S;X): \alpha\hat F(\mu)\leq 1\} = \alpha B\,,
\end{equation}
and
\begin{equation}
B^\alpha_s := \{v \in X: \alpha J(v,s) \leq 1\} = \alpha B_s
\end{equation}
for $\alpha >0$.
Note that the statement of Theorem \ref{thm:extremal} adapts by simply replacing $B$ with $B^\alpha$ and $B_r$ with $B_r^\alpha$.

\section{A representer theorem for the infinite infimal convolution functional}\label{sec:representer}

In this section we analyze sparse solutions for inverse problems regularized with the infinite infimal convolution functional introduced in Section \ref{sec:inverse}, namely 
\begin{equation}\label{eq:invrepr}
  \inf_{\substack{v \in X, \\ \sigma\in \mathcal{M}_+(S)}} \frac{1}{2}\|Av - f\|^2_Y + \alpha R(v,\sigma)\,,
\end{equation}
where we used the same notations as in Section \ref{sec:inverse}.

Sparse solutions of infinite dimensional inverse problems regularized with convex functionals have been recently studied in \cite{bc, chambolle, unsersplines}. In particular, it has been shown that in a general inverse problem, the extremal points of the ball of the convex regularizer can be seen as the atoms of the problem, in analogy with more standard sparsity promoting regularizers, such as the $\ell^1$ norm in a finite dimensional setting and the total variation norm for Radon measures. This claim can be justified by invoking so-called representer theorems that, in case  of finite-dimensional measurements, ensure the existence of a solution that is representable as a finite linear combination of the mentioned extremal points.
In the next theorem, we take advantage of the characterization of extremal points for the ball of the lifted regularizer, see Theorem \ref{thm:extremal}, and we use the results in \cite{bc} to 
prove a representer theorem for \eqref{eq:invrepr} assuming finite-dimensional measurements. In particular, we apply the results in \cite{bc} to the lifted problem \eqref{eq:lifting} and then, thanks to the equivalence of \eqref{eq:lifting} with \eqref{eq:invrepr}, we deduce the representer theorem for \eqref{eq:invrepr}. Note that, in order to apply the results in \cite{bc}, we need to additionally assume that $J(\cdot,s)$ are seminorms for every $s \in S$. Even if such assumption is often satisfied, there are relevant cases for which a more general result would be desirable. For example, in case of  inverse problems for dynamic measures regularized with the Benamou-Brenier energy \cite{bcfr}, it can be checked that the regularizer is convex, but not positively $1$-homogeneous, since it is equal $+\infty$ for non-positive measures. More in general, the presence of convex constraints in the regularization could hinder the positive $1$-homogeneity of the regularizer.  
Having said that, we believe that it would be possible to remove the seminorm assumption by applying the results in \cite{chambolle}. 

\begin{thm}\label{thm:representer}
Assume that $J(\lambda v,s) = |\lambda| J(v,s)$ for every $\lambda \in \R$, $v\in X$ and $s\in S$. Suppose that the measurement space $Y$ is finite-dimensional with ${\rm dim}(Y) = Q \in \N$ and that ${\rm coni}(\{Av: v \in \bigcup_{s\in S} \domain J(\cdot,s)\}) = Y$, where ${\rm coni}(C)$ denotes the conical hull of the set $C$. Then, there exists $(v^*,\sigma^*) \in X \times \mathcal{M}_+(S)$ that minimizes \eqref{eq:inverseproblem} such that 
\begin{equation}
v^* = \sum_{i = 1}^M \lambda_i v_i \,, \quad \sigma^* =  \sum_{s \in \{s_i : \,i =1,\ldots,M\}} \bigg\|\sum_{j \in \{1,\ldots, M :\, s_j = s\}}  \lambda_j v_j\bigg\|_X\delta_{s}\,,
\end{equation}
where $M \leq Q$, $s_i \in S$, $v_i \in \Ext(B^\alpha_{s_i}) \smallsetminus \{0\}$ and $\lambda_i >0$ for every $i = 1,\ldots, M$.
\end{thm}

\begin{proof}
We rewrite \eqref{eq:invrepr} in $\mathcal{M}(S;X)$ according to \eqref{eq:lifting}

\begin{equation}\label{eq:liftingrep}
\inf_{\mu \in \mathcal{M}(S;X)} \frac{1}{2}\|K\mu - f\|_Y^2 +   \alpha \int_S J\left(\frac{\mu}{|\mu|}(s),s\right)\, d|\mu|(s)\,,
\end{equation}
where $K := A \circ T$ and $T\mu := \mu(S)$. Note that $K : \mathcal{M}(S;X) \rightarrow Y$ is a linear operator and it is weak*-to-strong continuous. Indeed, $K$ is weak*-to-weak continuous by Remark \ref{rem:consweakstar}, and weak topology and strong topology on $Y$ coincide by the finite dimensionality of $Y$.
Note that $\hat F(\lambda \mu) = |\lambda| \hat F(\mu)$ for every $\mu \in \mathcal{M}(S;X)$ and $\lambda \in \R$. Indeed, for every $\lambda  \neq 0$,  \cref{prop:totalvariationproduct} yields that $\dfrac{\lambda \mu}{|\lambda \mu|}(s) =    \sign(\lambda)\dfrac{\mu}{|\mu|}(s)$ for $|\mu|$-a.e. $s \in S$  and hence, using that $J(\lambda v,s) = |\lambda| J(v,s)$ for every $\lambda \in \R$, $v\in X$ and $s\in S$, we have
\begin{equation*}
\hat F(\lambda \mu) = \int_S J\left(\sign(\lambda)\frac{\mu}{|\mu|}(s),s\right)\, d(|\lambda\mu|)(s) = |\lambda| \hat F( \mu)\,.
\end{equation*}
Therefore, \cref{prop:convexityintegral} and \cref{thm:lsc} imply that $\hat F$ is a lower semicontinuous seminorm on  $\mathcal{M}(S;X)$ equipped with the weak* topology. Moreover, \cref{prop:convexityintegral} and the Banach--Alaoglu theorem (see also \cref{thm:duality}) imply that the sublevel set of $\hat F$ are compact in the weak* topology of $\mathcal{M}(S;X)$.

 We now show that $K(\domain \hat F) = Y$. Note that it is enough to prove that  $Y \subset K(\domain \hat F) = A(T(\domain \hat F))$. Given $y \in Y$, using that  ${\rm coni}(\{Av: v \in \bigcup_{s\in S} \domain J(\cdot,s)\}) = Y$, we can write $y$ as  $y = Av$ for $v = \sum_{j=1}^N c_j v_{s_j}$ with $N \in \N$, $s_j \in S$, $v_{s_j} \in \domain J(\cdot,s_j)$ and $c_j \geq 0$. Define then $\mu = \sum_{j=1}^N c_j v_{s_j}\delta_{s_j} \in \mathcal{M}(S;X)$ and note that $\mu \in \domain(\hat F)$, since by \eqref{eq:mutuallysingularthesis} in Proposition \ref{prop:convexityintegral}, the convexity and the positive one-homogeneity of $J(\cdot,s)$, we have
 \begin{align}
     \hat F(\mu) = \hat F\left(\sum_{j=1}^N c_j v_{s_j}\delta_{s_j}\right) \leq \sum_{j=1}^N c_j J(v_{s_j}, s_j) < \infty \,.
 \end{align}
We then conclude that $y \in A(T(\domain \hat F))$ by observing that $T\mu = v$. 

Therefore, we can apply the representer theorem \cite[Theorem 3.3]{bc} 
to ensure the existence of a minimizer $\mu^* \in \mathcal{M}(S;X)$ of \eqref{eq:liftingrep} that can be written as
\begin{equation}\label{eq:zerolin}
\mu^* = \sum_{i = 1}^{M'} \lambda_i \mu_i\,,
\end{equation}
where $M' \leq Q$, $\lambda_i >0$ and $\mu_i \in \text{Ext}(B^\alpha)$ for $i = 1,\ldots, M'$.
Using \cref{thm:extremal}, we infer that there exist $s_i \in S$ and $v_i \in \text{Ext}(B^\alpha_{s_i})$ such that $\mu_i = v_i \delta_{s_i}$, so that, by removing from the linear combination \eqref{eq:zerolin} the measures $\mu_i$ that are equal to zero and possibly rearranging the indices, we obtain that 
\begin{equation}
\mu^* = \sum_{i = 1}^M \lambda_i  v_i \delta_{s_i}\,,
\end{equation}
where $M \leq M'$, $s_i \in S$ and $v_i \in \text{Ext}(B^\alpha_{s_i}) \smallsetminus \{0\}$.    
Finally, applying \cref{prop:equivalence} we deduce that the pair $(v^*,\sigma^*) \in X \times \mathcal{M}_+(S)$ defined as
\begin{align*}
v^*  = \sum_{i=1}^{M}  \lambda_i v_i \, ,\qquad \sigma^* =  \sum_{s \in \{s_i : \,i =1,\ldots,M\}} \bigg\|\sum_{j \in \{1,\ldots, M :\, s_j = s\}}  \lambda_j v_j\bigg\|_X\delta_{s} 
\end{align*}
is a minimizer of \eqref{eq:invrepr} as required.
\end{proof}

\section{A generalized conditional gradient method}\label{sec:gcg}

In this section we will describe how to solve the lifted inverse problem defined in  \eqref{eq:lifting} by a generalized conditional gradient method. As a consequence of the equivalence provided by \cref{prop:equivalence}, we will then be able to derive a conditional gradient method for \eqref{eq:inverseproblem} and \eqref{eq:inverseproblem2}. 
We recall that the lifted inverse problem consists in minimizing the following functional:
\begin{equation}\label{eq:inverseproblemconditional}
 \mathcal{G}_\alpha(\mu) = \frac{1}{2}\|A\mu(S) - f\|_Y^2 +   \alpha \hat F(\mu)\,,
\end{equation}
where
\begin{align*}
\hat F(\mu) = \int_S J\left(\frac{\mu}{|\mu|}(s),s\right)\, d|\mu|(s)\,.
\end{align*}
As already done in Section \ref{sec:representer}, denoting by $T : \mathcal{M}(S;X) \rightarrow X$ the linear operator defined as $T\mu := \mu(S)$, we rewrite \eqref{eq:inverseproblemconditional}  equivalently as
\begin{equation}\label{eq:classical}
\mathcal{G}_\alpha(\mu)  = \frac{1}{2}\|K\mu - f\|_Y^2 +   \alpha \hat F(\mu)\,,
\end{equation}
where $K := A\circ T$. Note again that $K$ is weak*-to-weak continuous by Remark \ref{rem:consweakstar}. The pre-adjoint of the linear operator $T : \mathcal{M}(S;X) \rightarrow X$ can be characterized as follows.

\begin{lemma}\label{lem:predual}
The pre-adjoint of the linear operator $T$ is $T_* : \predual{X} \rightarrow C(S;\predual{X})$ is given by
\begin{equation}\label{eq:thesispredual}
(T_* v_* )(s) = v_* \quad \forall s \in S,\,\ \  \forall v_* \in X_*\,.
\end{equation}
\end{lemma}
\begin{proof}
For every $\mu \in \mathcal{M}(S;X)$ and $v_* \in \predual{X}$ there holds
\begin{equation}\label{eq:chainpredual}
\langle (s \mapsto v_*) , \mu \rangle_{\mathcal{M}} = \langle v_* , \mu(S) \rangle_{(\predual{X},X)} = \langle v_* , T
\mu \rangle_{(\predual{X},X)}\,,
\end{equation}
showing \eqref{eq:thesispredual}. Notice that the first equality in \eqref{eq:chainpredual} follows from the definition of the integral with respect to the vector measure $\mu$. 
\end{proof}

\subsection{Description of the algorithm}
A generalized conditional gradient algorithm for minimizing $\mathcal{G}_\alpha$ is based on the construction of a sparse iterate $\mu^k \in \mathcal{M}(S; X)$ given by a linear combination of extremal points of $B$, i.e.,
\begin{align}\label{eq:gcgiterate}
\mu^k = \sum_{i=1}^{N_k} c^k_i \mu^k_i\,,
\end{align}
where $N_k \in \N$, $c^k_i >0$ and $\mu^k_i \in \Ext(B^\alpha)$ pairwise distinct, for every $i$.  At each iterate $k \in \N$, the algorithm performs two steps: the insertion step and the weights optimization step. In the insertion step a suitable extremal point of $B^\alpha$, denoted by $\mu_{N_k + 1}^k$,  is chosen in such a way that it minimizes a variational problem obtained from \eqref{eq:inverseproblemconditional} by linearizing the fidelity term $\frac{1}{2}\|K\mu - f\|_Y^2$ in the previous iterate $\mu^k$.
In the weights optimization step the coefficients of the conical combination $\left[\sum_{i=1}^{N_k} c^k_i \mu^k_i\right] + c  \mu_{N_k + 1}^k$ are optimized with respect to $\mathcal{G}_\alpha$. The next iterate is then constructed by defining
\begin{align*}
\mu^{k+1} = \left[\sum_{i=1}^{N_k} \hat c^k_i \mu^k_i\right] + \hat c^k_{N_k + 1}  \mu_{N_k + 1}^k\,,
\end{align*}
where $(\hat c^k_1, \ldots, \hat c^k_{N_k}, \hat c^k_{N_k + 1}) \in (0,+\infty)^{N_k + 1}$ is the solution of the weights optimization step. In the following part of this section we will provide more details on the insertion step and the coefficient optimization step for our problem. Then, we will formulate the generalized conditional gradient algorithm for the lifted inverse problem, see Algorithm \ref{alg:liftedalgorithm}, and we discuss how to equivalently reformulate this algorithm to compute minimizers of \eqref{eq:inverseproblem} and \eqref{eq:inverseproblem2}, see Algorithm \ref{alg:corealgorithm}.

\subsubsection{Insertion step}\label{sec:insertions}

Given a sparse iterate $\mu^k \in \mathcal{M}(S;X)$ that is a conical combination of extremal points of $B^\alpha$ as in \eqref{eq:gcgiterate}, 
we define the dual variable associated to the $k$-th iteration $p^k \in C(S;\predual{X})$ as
\begin{equation}\label{eq:dualvar}
p^k = -K_*(K\mu^k - f) \,.
\end{equation}
Such dual variable $p^k$ represents the negative derivative of the fidelity term in of $\mathcal{G}_\alpha$ defined in \eqref{eq:classical} computed at the point $\mu^k$.
We then find the new extremal point  of $B^\alpha$ by solving the linearized problem defined as follows.
\begin{dfnz}\label{def:linproblem}
With $p^k$ as in \eqref{eq:dualvar} we define the linearized problem in $\mu^k \in \mathcal{M}(S;X)$ as 
\begin{equation}\label{eq:dualmin}
\max_{\substack{\mu \in \Ext(B^\alpha)}} \langle p^k,\mu\rangle\,.
\end{equation}
\end{dfnz}
Note that the linearized problem in Definition \ref{def:linproblem} is, in our setting, equivalent to  maximizing  the linearization of the objective function as in the standard Frank--Wolfe algorithm \cite{frank1956algorithm}:
\begin{align}\label{eq:originallinearization}
\max_{\mu \in B^\alpha}\, \langle p^k,\mu\rangle \,.   
\end{align}
Indeed, a simple convexity argument shows that it is always possible to find a maximizer of \eqref{eq:originallinearization} in the set of extremal points of $B$; see for example \cite[Lemma 3.1]{bredies2021linear}.
As a consequence, thanks to Proposition \ref{prop:convexityintegral} and Theorem \ref{thm:lsc}, the problem \eqref{eq:dualmin} admits a solution.   

Rephrasing the insertion step as in Definition \ref{def:linproblem} allows to constrain the iterate of our conditional gradient method to be a linear combination of the extremal points of $B^\alpha$, highlighting the sparsity promoting structure of the regularizer. As a consequence,
we can choose $\mu^k_{N_k +1} \in \argmax_{\substack{\mu \in \Ext(B^\alpha)}} \langle p^k,\mu\rangle$ to be the new extremal point added to the linear combination \eqref{eq:gcgiterate}.



\subsubsection{Weights optimization step}
In the insertion step we have selected the new extremal point $\mu^k_{N_k +1} \in \argmax_{\substack{\mu \in \Ext(B^\alpha)}} \langle p^k,\mu\rangle$ to add to the iterate $\mu^k$. The goal of the weights optimization step is to optimize the coefficients of the sparse conical combination $\left[\sum_{i=1}^{N_k} c^k_i \mu^k_i\right] + c \mu^k_{N_k +1}$ with respect to the energy $\mathcal{G}_\alpha$, i.e. to solve
 \begin{equation}\label{eq:optimizationcoefficient-1}
\argmin_{c_i \in \R_+^{N_k+1}} \frac{1}{2}\left\|\sum_{i=1}^{N_k + 1} c_i K\mu_i^k- f\right\|_Y^2 +  \alpha\hat F \left(\sum_{i=1}^{N_k+1} c_i \mu_i^k \right)\,.
\end{equation}
In practice (see \cite[Section 3]{bredies2021linear} for more details), it is enough to optimize the following upper bound of \eqref{eq:optimizationcoefficient-1}:
\begin{equation}\label{eq:optimizationcoefficient}
\hat c^k \in \argmin_{c_i \in \R_+^{N_k+1}} \frac{1}{2}\left\|\sum_{i=1}^{N_k+1} c_i K\mu_i^k - f\right\|_Y^2 +  \alpha\sum_{i=1}^{N_k+1} c_i \hat F (\mu_i^k)\,,
\end{equation}
where the bound $\hat F \left(\sum_{i=1}^{N_k+1} c_i \mu_i^k \right) \leq  \sum_{i=1}^{N_k+1} c_i \hat F (\mu_i^k)$ is ensured by the convexity and the positive one-homogeneity of $\hat F$. 
Note that, since $\hat F$ is positively one-homogeneous, we have $\hat F (\mu_i^k) \in \{0,\alpha^{-1}\}$. Therefore, \eqref{eq:estcoercivity} in Theorem \ref{prop:convexityintegral} implies that, if all inserted extremal points are different from zero, then the insertion step can be rewritten as
\begin{equation}\label{eq:optimizationcoefficient2}
\hat c^k \in \argmin_{c_i \in \R_+^{N_k+1}} \frac{1}{2}\left\|\sum_{i=1}^{N_k+1} c_i K\mu_i^k - f\right\|_Y^2 +  \sum_{i=1}^{N_k+1} c_i\,.
\end{equation}
Since we cannot exclude that all inserted extremal points are different from zero at all iterations, we will employ the optimization problem in \eqref{eq:optimizationcoefficient-1}. However, 
we refer to Remark \ref{rem:nonzero}, where we discuss 
in details how and when the insertion of null extremal points can be avoided.

\begin{rem}
We point out that \eqref{eq:optimizationcoefficient} can be a strict upper bound for \eqref{eq:optimizationcoefficient-1} depending on the choice of the functionals $J(\cdot,s)$. To demonstrate this, take $\mu^k = \sum_{i=1}^{N_k+1} c^k_i \mu^k_i$, where $N_k \in \N$, $c^k_i >0$ and $\mu^k_i \in \Ext(B^\alpha)$ pairwise distinct, for every $i$. Thanks to Theorem \ref{thm:extremal}, it holds that $\mu_i^k = v_i^k \delta_{s^k_i}$ where $s^k_i \in S$ and $v_i^k \in \Ext(B^\alpha_{s^k_i})$ for $i = 1, \ldots, N_k+1$. 
Using Proposition \ref{prop:convexityintegral}, we can write
\begin{align}\label{eq:lsa}
   \hat F\left(\sum_{i=1}^{N_k+1} c_i \mu_i^k \right) & = \hat F\Bigg(\sum_{s \in \{s^k_i : i =1, \ldots, N_k+1\}} \Bigg(\sum_{j\in \{1,\ldots, N_k +1 :\, s^k_j = s\}} c_j v_j^k \Bigg) \delta_{s} \Bigg) \nonumber\\
   & =  \sum_{s \in \{s^k_i : i =1, \ldots, N_k+1\}} J\Bigg(\sum_{j\in \{1,\ldots, N_k+1 :\, s^k_j = s\}} c_j v_j^k, s \Bigg)\,.
\end{align}
In particular, \eqref{eq:lsa} shows that if for every $s \in S$ the functional $J(\cdot,s)$ is linear on conical combinations of elements in $\Ext(B^\alpha_s)$, then  \eqref{eq:optimizationcoefficient} is equivalent to \eqref{eq:optimizationcoefficient-1}. However, this is in general not true even for simple examples such as $J(v,s) = \|v\|_{L^2}$.
\end{rem}

Owing to the linearity of $K$, the optimization in \eqref{eq:optimizationcoefficient} is a convex quadratic problem in $\R_+^{N_k+1}$ and therefore can be solved efficiently with standard solvers.
We then construct the next iterate $\mu^{k+1} \in \mathcal{M}(S;X)$ as 
$\mu^{k+1} = \sum_{i=1}^{N_k+1} \hat c^k_i \mu^k_i$, where $\hat c^k_i$ are the optimal weights determined as in \eqref{eq:optimizationcoefficient-1}. 
Algorithm \ref{alg:liftedalgorithm} summarizes these steps.

\subsubsection{Generalized conditional gradient method for the lifted problem} 

 Algorithm \ref{alg:liftedalgorithm} is a generalized conditional gradient method that computes a minimizer of the lifted inverse problem \eqref{eq:inverseproblemconditional}. Note that we keep track of the extremal points in the representation of the iterate $\mu^k \in \mathcal{M}(S;X)$ in the set $\mathcal{A}^k$. This is useful because during the weight optimization step some of the optimal $\hat c^k_i \in \R_+$ could be zero. If this is the case we want to remove the corresponding extremal point from the linear combination. Such operation is performed by removing from $\mathcal{A}^{k+1}$ the extremal points $\mu^k_i$ whose corresponding weight $\hat c^k_i$ is equal to zero.
 
We choose, for simplicity, to initialize Algorithm \ref{alg:liftedalgorithm} with $\mu^0 = 0$ and $\mathcal{A}^0 = \emptyset$. However, we point out that it is possible to have, as initialization, any sparse measure $\mu^0 = \sum_{i=1}^{N_0} c^0_i \mu^0_i$ with $c_i >0$ and $\mathcal{A}^0 = \{\mu^0_1, \ldots, \mu^0_{N_0}\}$.

It is possible to add a stopping criterion to Algorithm \ref{alg:liftedalgorithm} as in \cite[Proposition 3.3]{bredies2021linear}. In particular, we set the algorithm to terminate at the $k$-th iteration for $k \geq 1$ if 
 \begin{align}\label{eq:stopp}
 \max_{\substack{\mu \in \Ext(B^\alpha)}} \langle p^k,\mu\rangle \leq 1
 \end{align}
 or if the newly inserted atom is already among the atoms composing $\mu^k$, i.e., $\mu^k_{N_k+1} \in \mathcal{A}^k$. 

 \begin{rem}\label{rem:nonzero}
 Note that the stopping criterion \eqref{eq:stopp} is met, in particular, if 
 \begin{align}\label{eq:zeroin}
     0 \in \argmax_{\substack{\mu \in \Ext(B^\alpha)}} \, \langle p^k,\mu\rangle\,.
 \end{align}
Therefore, for all the iterations $k \geq 1$ the algorithm is never inserting $0$, since in this case the algorithm would stop.  
At $k = 0$, the stopping criterion is not given; however, by initializing Algorithm  \ref{alg:liftedalgorithm} with the zero measure, one has that $p^0 = K_*f$ and thus, $0 \in \argmax_{\substack{\mu \in \Ext(B^\alpha)}} \langle p^0,\mu \rangle$ implies that $0$ is a minimizer of \eqref{eq:classical} due to general optimality conditions derived, for example, in \cite[Proposition 2.3]{bredies2021linear} and the equivalence between \eqref{eq:dualmin} and \eqref{eq:originallinearization} \cite[Lemma 3.1]{bredies2021linear}. 
If we initialize Algorithm  \ref{alg:liftedalgorithm} with a general sparse measure $\mu^0 = \sum_{i=1}^{N_0} c^0_i \mu^0_i$ with weights $c^0_i$ that are optimal with respect to the weights optimization step, the stopping criterion can be also given for the first iteration $k=0$ and the insertion of zero would cause the algorithm to stop.
Due to these considerations, from now on we will assume that $0$ is never inserted in the linear combination at any iteration.
 \end{rem}

 \begin{algorithm}[ht!]
  \caption{Generalized conditional gradient method for the lifted inverse problem}\label{alg:liftedalgorithm}
  \begin{algorithmic}
    \STATE Initialize: $\mu_0 = 0$,  $\mathcal{A}^0 = \emptyset$.
    \STATE \textbf{for} \(k=0,1,2\dots \ \textbf{do}\)
  \STATE \qquad $p^k \leftarrow -K_* (K\mu^k - f)$
  \STATE \qquad \(\mu^k_{N_k +1} \in \argmax_{\substack{\mu \in \Ext(B^\alpha)}} \langle p^k,\mu\rangle \)
   \STATE \qquad \textbf{if}  $k \geq 1$, $\langle p^k,\mu^k_{N_k +1}\rangle  \leq 1$ or $\mu^k_{N_k +1} \in \mathcal{A}^k$ \textbf{then}
      \STATE \qquad \qquad \textbf{return} $\mu^k$
        \STATE \qquad  \textbf{end if}
        \STATE \qquad $\hat c^k \in  \argmin_{c \in \R_+^{N_k+1}}\frac{1}{2}\|\sum_{i=1}^{N_k+1} c_i K\mu_i^k - f\|_Y^2 + \alpha \sum_{i=1}^{N_k+1} c_i \hat F(\mu_i^k)$
    \STATE \qquad $\mu^{k+1} \leftarrow  \sum_{i=1}^{N_k+1} \hat c^k_i \mu^k_i$
   \STATE \qquad $\mathcal{A}^{k+1} \leftarrow \{\mu_i^k : \hat c^k_i > 0\}$, \ $N_{k+1} \leftarrow \# \mathcal{A}^{k+1}$ 
    \STATE \qquad $k \leftarrow k+1$
 \STATE \textbf{end for} 
 \STATE \textbf{return} $\mu^k$ 
  \end{algorithmic}
  \end{algorithm}

\subsubsection{Generalized conditional gradient method for the infinite infimal convolution functional}

Using the characterization of extremal points of $B^\alpha$ from Theorem \ref{thm:extremal} together with Lemma \ref{lem:predual}, we can rewrite Algorithm \ref{alg:liftedalgorithm} in the space $X \times \mathcal{M}_+(S)$ to obtain an implementable reformulation and a sequence of iterations that converges to a minimizer of \eqref{eq:inverseproblem}, see Section \ref{sec:convergence}.
By Theorem \ref{thm:extremal} the iterate $\mu^k$ can be written as
\begin{equation}\label{eq:linearcombination}
\mu^k = \sum_{i=1}^{N_k} c^k_i v^k_i \delta_{s^k_i}\,,
\end{equation}
with $N_k \in \N$, $c^k_i > 0$, $s^k_i \in S$ and $v^k_i \in \Ext(B^\alpha_{s_i^k})\smallsetminus\{0\}$ for $i=1,\ldots, N_k$, where $v^k_i$ are different from zero due to Remark \ref{rem:nonzero}. 
%
%
By applying \cref{lem:predual}, the dual variable $p^k \in C(S;\predual{X})$ defined in \eqref{eq:dualvar} can be rewritten as
\begin{equation}\label{eq:constantdual0}
p^k(s) = -A_* (AT \mu^k - f)  \quad \forall s \in S\,,
\end{equation}
and thus, applying the definition of $T$ we get
\begin{equation}\label{eq:constantdual}
p^k(s) = -A_* \left( \sum_{i=1}^{N_k} c^k_i A v^k_i - f\right)  \quad \forall s \in S\,.
\end{equation}
We choose a new extremal point of $B^\alpha$, denoted by $v^k_{N_k+1}\delta_{s^k_{N_k+1}}$, to be a conical combination \eqref{eq:linearcombination} that minimizes the linearized problem defined in Definition \ref{def:linproblem}. More precisely, using \eqref{eq:constantdual}, the insertion step can be written as follows
\begin{align*}
\max_{\mu \in \Ext(B^\alpha)} \,  \langle p^k,\mu\rangle & =\max_{\substack{(v,s)\in X \times S,\\ v\in  \Ext(B^\alpha_s)}} \langle p^k, v \delta_s \rangle  = \max_{\substack{(v,s)\in X \times S,\\ v\in  \Ext(B^\alpha_s)}}   \langle -A_* \left( \sum_{i=1}^{N_k} c^k_i A v^k_i - f\right), v \rangle\,.
\end{align*}
Therefore, in the insertion step we seek to solve 
\begin{equation}\label{eq:insertion_step2}
(v^k_{N_k+1},s^k_{N_k+1}) \in \argmax_{\substack{(v,s)\in X \times S,\\ v\in  \Ext(B^\alpha_s)}}   \langle -A_* \left( \sum_{i=1}^{N_k} c^k_i A v^k_i - f\right), v \rangle \,.
\end{equation}

\begin{rem}\label{rem:comp}
We remark that the insertion step is usually the main computational bottleneck of every generalized conditional gradient method, also in finite dimensions. Not surprisingly the same computational challenges affect \eqref{eq:insertion_step2}. On one hand, the implementability of \eqref{eq:insertion_step2} relies on an explicit characterization of $\Ext(B^\alpha_s)$, that, depending on the regularizer $J$, could be a challenging task. Moreover, even if such characterization is available, the optimization problem \eqref{eq:insertion_step2} is neither convex or concave and therefore the computation of its global maximizer is usually problematic, above all if $S$ and $\Ext(B^\alpha_s)$ are high dimensional. In our numerical experiments, c.f. Section \ref{sec:numerics}, we consider Fourier-type regularizers that have a simple characterization for $\Ext(B^\alpha_s)$ and allow for a closed formula for
\begin{align}\label{eq:rr}
\argmax_{\substack{v\in  \Ext(B^\alpha_s)}}   \langle -A_* \left( \sum_{i=1}^{N_k} c^k_i A v^k_i - f\right), v \rangle 
\end{align}
for every $s \in S$. Even with this simplifying choice, \eqref{eq:insertion_step2} remains computational challenging.
 We refer to Section \ref{sec:computationalissues} for more details about our computational choices for the specific numerical examples contained in this paper and the limitations we encountered.
\end{rem}

The weights optimization step can be also rewritten as the quadratic problem
\begin{equation}\label{eq:optimizationcoefficient2_}
\hat c^k \in \argmin_{c \in \R_+^{N_k+1}} \frac{1}{2}\left\|\sum_{i=1}^{N_k+1} c_i Av^k_i - f\right\|_Y^2 +  \alpha\sum_{i=1}^{N_k+1} c_i J(v^k_i,s^k_i)
\end{equation}
and the next iterate $\mu^{k+1} \in \mathcal{M}(S;X)$ is produced as
$\mu^{k+1} = \sum_{i=1}^{N_k+1} \hat c^k_i v^k_i \delta_{s^k_i}$.

In order to rewrite Algorithm \ref{alg:liftedalgorithm} in $X \times \mathcal{M}_+(S)$, we define the iterate $(v^k , \sigma^k) \in X \times \mathcal{M}_+(S)$ as 
 $v^{k} = \sum_{i = 1}^{N_{k}} c_i^{k} v_i^{k}$ and 
 $\sigma^k = \sum_{s \in \{s^k_i : \,i =1,\ldots,N_k\}} \big\|\sum_{j \in \{1,\ldots, N_k :\, s^k_j = s\}}   c^k_j v^k_j\big\|_X\delta_{s}$.
In Algorithm \ref{alg:corealgorithm} we summarize the generalized conditional gradient algorithm we use to solve \eqref{eq:classical}. 
Similarly to Algorithm \ref{alg:liftedalgorithm}, we keep track of the pairs $(v^k_i, s^k_i)$ that represent $v^k$ and $\sigma^k$
by defining the sets $\Xi^k = \{(v^k_1, s^k_1), \ldots, (v^k_{N_k}, s^k_{N_k})\}$. At each iteration we add the solution of \eqref{eq:insertion_step2} to $\Xi^k$ and, after the weights optimization step, we remove the pairs $(v^k_i, s^k_i)$ for which the corresponding weights $\hat c^k_i$ are equal to zero.

Thanks to Proposition \ref{prop:equivalence} and Theorem \ref{thm:extremal}, the stopping criterion we used in Algorithm \ref{alg:liftedalgorithm} can be  reformulated for Algorithm \ref{alg:corealgorithm} as follows. We set the algorithm to terminate at the $k$-th iteration, for $k\geq 1$, if 
 \begin{align}\label{eq:stoppingspec}
\argmax_{\substack{(v,s)\in X \times S,\\ v\in  \Ext(B^\alpha_s)}} \langle p^k, v\rangle \leq 1\,.
 \end{align}

  \begin{algorithm}[ht!]
  \caption{Generalized conditional gradient method for infinite infimal convolutions}\label{alg:corealgorithm}
  \begin{algorithmic}
    \STATE Initialize: $v^0 = 0$, $\sigma^0 = 0$,
$\Xi^0 = \emptyset$

    \STATE \textbf{for} \(k=0,1,2\dots \ \textbf{do}\)
  \STATE \qquad $p^k \leftarrow  -A_* \left( A v^k- f\right)$
  \STATE \qquad \((v^k_{N_k+1}, s^k_{N_k+1}) \in  \argmax_{\substack{(v,s)\in X \times S,\\ v\in  \Ext(B^\alpha_s)}} \langle p^k, v\rangle\) 
     \STATE \qquad \textbf{if}  $k \geq 1$, $\langle p^k,v^k_{N_k+1}\rangle  \leq 1$ or $(v^k_{N_k+1}, s^k_{N_k+1}) \in  \Xi^k$ \textbf{then}
      \STATE \qquad \qquad \textbf{return} $(v^k,\sigma^k)$
         \STATE \qquad  \textbf{end if}
        \STATE \qquad $(\hat c^k_1, \ldots, \hat c^k_{N_k +1}) \in \argmin_{c \in \R_+^{N_k+1}}\frac{1}{2}\|\sum_{i=1}^{N_k+1} c_i Av^k_i - f\|_Y^2 +  \alpha \sum_{i=1}^{N_k+1} c_i J(v^k_i,s^k_i)$
        \STATE \qquad $v^{k+1} \leftarrow \sum_{i = 1}^{N_{k}+1} \hat c^k_i v_i^{k}$, \ $\sigma^{k+1} \leftarrow  \sum_{s \in \{s^k_i : \,i =1,\ldots,N_k+1\}} \big\|\sum_{j \in \{1,\ldots, N_k +1 :\, s^k_j = s\}}   \hat c^k_j v^k_j\big\|_X\delta_{s}$
         \STATE \qquad $\Xi^{k+1}  \leftarrow \big\{(v_i^k, s_i^k) \in \Xi^{k} \cup \{(v^k_{N_k+1},s^k_{N_k+1})\} : \hat c^k_i >0\big\}$
         \STATE \qquad $N_{k + 1} \leftarrow \# \Xi^{k+1}$, \ $k \leftarrow k+1$
   \STATE  \textbf{end for}
   \STATE  \textbf{return} $(v^k,\sigma^k)$
  \end{algorithmic}
  \end{algorithm}

\subsection{Sublinear convergence}\label{sec:convergence}

The goal of this section is to prove sublinear convergence for the iterates generated by Algorithms~\ref{alg:liftedalgorithm} and~\ref{alg:corealgorithm}. To this end, we define the functional distance associated to the objective functional $\mathcal{G}_\alpha$, see \eqref{eq:inverseproblemconditional}, as
 \begin{equation} \label{def:residual}
  r_{\mathcal{G}_\alpha}(\mu):=\mathcal{G}_\alpha(\mu) - \min_{\tilde \mu \in \mathcal{M}(S;X)} \mathcal{G}_\alpha(\tilde \mu) \,
  \end{equation}
  for all $\mu \in \mathcal{M}(S;X)$ and the functional distance associated to the original inverse problem \eqref{eq:inverseproblem} as
  \begin{equation} \label{def:residualinverse}
  r_{\mathcal{F}_\alpha}(v,\sigma):=\mathcal{F}_\alpha(v,\sigma) - \min_{\substack{\tilde v \in X, \\ \tilde \sigma \in \mathcal{M}_+(S)}} \mathcal{F}_\alpha(\tilde v, \tilde \sigma) \,
  \end{equation}
  for all $v \in X$ and $\sigma \in \mathcal{M}_+(S)$,
  where 
  \begin{align*}
  \mathcal{F}_\alpha(v,\sigma) :=  \frac{1}{2}\|Av - f\|^2_Y + \alpha R(v,\sigma)\,.
  \end{align*}
The following convergence result for $r_{\mathcal{G}_\alpha}$ holds.  
   \begin{thm} \label{thm:convergence}
 Suppose that hypotheses \ref{ass:convexlsc}, \ref{ass:lsc} and \ref{ass:coercivity} hold. Additionally  assume that
 \begin{enumerate}[label=(I\arabic*)]
\item \label{ass:domain} for
every sequence $\{v^k\}_k$ in $X$ such that there exists  $\sigma^k \in \mathcal{M}_+(S)$ for which $v^k \in \domain (R(\cdot,\sigma^k))$ and $v^k \weakstar v$, $v \in X$, it holds that $\lim_{k\rightarrow +\infty} Av^k = Av$.
 \end{enumerate}
 Let $\{\mu^k\}_k$ be a sequence generated by Algorithm \ref{alg:liftedalgorithm}. Then either Algorithm \ref{alg:liftedalgorithm} terminates after a finite number of steps returning a minimizer of $\mathcal{G}_\alpha$, or  there exists $C>0$ depending only on $f$ such that 
    \begin{equation}\label{eq:decaytheorem}
    r_{\mathcal{G}_\alpha}(\mu^k) \leq \frac{C}{k+1} \,, \,\, \text{ for all } \,\, k \in \N\,.
    \end{equation}
Moreover, each accumulation point $\mu^*$ of $\{\mu^k\}_k$ with respect to the weak* topology of $\mathcal{M}(S;X)$ is a minimizer for $\mathcal{G}_\alpha$. If $\mathcal{G}_\alpha$ admits a unique minimizer $\mu^*$, then $\mu^k \weakstar \mu^*$ along the whole sequence.
  \end{thm}  
\begin{proof}  
The proof of \cref{thm:convergence} proceeds by specializing  \cite[Theorem 3.4]{bredies2021linear} in our setting for the generalized conditional gradient method applied to the functional $\mathcal{G}_\alpha$. We need to verify that the assumptions $(A1), (A2), (A3)$ provided in Section $2$ in \cite{bredies2021linear} hold in our setting.
First note that $\mathcal{M}(S;X)$ is a Banach space when equipped with the total variation norm and $C(S;\predual{X})$ is its predual (see \cref{thm:duality}). Additionally, the functional $\hat F : \mathcal{M}(S;X) \rightarrow [0,\infty]$ is convex, positively one-homogeneous, weakly-* lower semicontinuous and coercive in the sense that its sublevel sets are weakly-* compact (see \cref{prop:convexityintegral} and \cref{thm:lsc}).
This ensures the validity of $(A1)$ and $(A2)$. It remains to prove that the linear operator $K = A \circ T$, where $T\mu = \mu(S)$, is weak*-to-weak continuous and it satisfies assumption $(A3)$ in \cite{bredies2021linear} as a consequence of  \ref{ass:domain}. 
Applying Remark \ref{rem:consweakstar}, we immediately deduce that $K$ is weak*-to-weak continuous since $A$ is weak*-to-weak continuous.  Moreover, given a sequence $\{\mu^n\}_n$ in $\mathcal{M}(S;X)$ such that $\mu^n \in \domain(\hat F)$ we have that $\mu^n(S) \in \domain (R(\cdot,|\mu^n|))$. Indeed, by the definition of $R(\mu^n(S),|\mu^n|)$ and the fact that $\int_S \frac{\mu^n}{|\mu^n|}(s)\, d|\mu^n|(s) = \mu^n(S)$ we have that $R(\mu^n(S),|\mu^n|)  \leq \hat F(\mu^n)< \infty$. Therefore, considering a sequence $\{\mu^n\}_n$ in $\domain (\hat F)$ such that $\mu^n \weakstar \mu$, using Assumption \ref{ass:domain} and recalling Remark \ref{rem:consweakstar}, we conclude that $K\mu^n \rightarrow K\mu$ strongly. 
We can therefore apply Theorem 3.4 in \cite{bredies2021linear} and conclude the proof.
\end{proof}
\begin{rem}
Assumption (I1) is specifically used to satisfy Assumption (A3) in \cite[Theorem 3.4]{bredies2021linear}. In particular, we remark that Assumption (I1) is implied by the weak*-to-strong continuity of the operator $A$. However, it can be useful to require (I1) instead of simply assuming the weak*-to-strong continuity of $A$ in cases where the finiteness of the regularizer along a given sequence $v^k$ could improve the convergence behaviour of $Av^k \rightarrow Av$ to weak*-to-strong. An example could be found in inverse problems for dynamic measures regularized with the Benamou-Brenier energy \cite{bcfr}. Here the finiteness of the regularizer is implying the absolutely continuity of the curve of measures, allowing for operators $A$ to be defined pointwise in time.
\end{rem}

By applying Proposition \ref{prop:equivalence} and Theorem \ref{thm:convergence}, we can deduce from the previous theorem a convergence result for Algorithm \ref{alg:corealgorithm} as follows.
\begin{cor}\label{cor:convergencenotlifted}
Let 
\begin{equation}
v^k  = \sum_{i=1}^{N_k} c^k_i v^k_i \in X\, ,\qquad \sigma^k = \sum_{s \in \{s^k_i : \,i =1,\ldots,N_k\}} \bigg\|\sum_{j \in \{1,\ldots, N_k :\, s^k_j = s\}}   c^k_j v^k_j\bigg\|_X\delta_{s}
\end{equation}
be the iterates generated by Algorithm \ref{alg:corealgorithm}. If Algorithm \ref{alg:corealgorithm} terminates after a finite number of iterations returning $(v^*, \sigma^*) \in X\times \mathcal{M}_+(S)$, then $(v^*, \sigma^*)$ is a minimizer of $\mathcal{F}_\alpha$.
Moreover, there exists $C>0$ depending only on $f$ such that 
\begin{equation}
 r_{\mathcal{F}_\alpha}(v^k,\sigma^k) \leq \frac{C}{k+1}\,, \,\, \text{ for all } \,\, k \in \N\,.
\end{equation}
\end{cor}
\begin{proof}
If Algorithm \ref{alg:corealgorithm}  terminates after a finite number of steps returning $(v^*, \sigma^*) $, the proof follows from \cref{thm:convergence} by noticing that, given a minimizer for $\mathcal{G}_\alpha$ denoted by $\mu^* \in \mathcal{M}(S;X)$,  \cref{prop:equivalence} ensures that $(\frac{\mu^*}{|\mu^*|}, |\mu^*|)$ is a minimizer for \eqref{eq:inverseproblem2} and thus, $(\mu^*(S), |\mu^*|)$ is a minimizer for \eqref{eq:inverseproblem}.
Now note that 
\begin{align*}
\mathcal{F}_\alpha(v,\sigma) \leq \mathcal{G}_\alpha(\mu)
\end{align*}
for every $\mu$ such that $\mu(S) = v$ and $|\mu| = \sigma$. Therefore, applying \cref{prop:equivalence}, we deduce that 
\begin{equation}
r_{\mathcal{F}_\alpha}(v^k,\sigma^k) \leq r_{\mathcal{G}_\alpha}(\mu^k)\,.
\end{equation}
An application of \cref{thm:convergence} concludes the proof.
\end{proof}

%
%
%

\section{Illustrative applications and numerics}\label{sec:numerics}

The goal of this section is to provide illustrative examples of inverse problems with infinite infimal convolution regularization and to show that the generalized conditional gradient method introduced in Section \ref{sec:gcg} can be employed to solve them efficiently without discretizing the set $S$. In Section \ref{sec:1dadaptive} we consider the infimal convolution of fractional-Laplacian-type functionals with different fractional orders. In Section \ref{sec:anisotropic} we focus  on the infimal convolution of anisotropic fractional-Laplacian-type regularizers. 

All  experiments in this paper are carried out on Python3 on a MacBook Pro with 8 GB RAM and an Intel®Core™ i5, Quad-Core, 2.3 GHz.

\subsection{\texorpdfstring{Adaptively selecting the orders of $1$-dimensional fractional Laplacians in an\\ infinite infimal convolution}{Adaptively selecting the orders of 1-dimensional fractional Laplacians in an infinite infimal convolution}} \label{sec:1dadaptive}

In this section we apply our setting to an inverse problem regularized by the infimal convolution of fractional-Laplacian-type operators. We refer to  \cref{app:fraclap} for  basic definitions and results regarding $L^2$-periodic functions and fractional-Laplacian-type operators needed in this section.
Our goal is to define an infinite infimal convolution of functionals that regularizes a given observation by penalizing combinations of fractional order derivatives, where the orders themselves are learnt and depend on the frequencies of the observation. The higher the frequency of the observation, the lower the order of the regularization applied to that frequency. We use this regularizer to denoise periodic signals. The goal of this example is to assess the performance of our regularization functional in a simple setting and to see how the optimization procedure takes advantage of sparsity in the space of parameters.

\subsubsection{Setting}
We denote by $T^1$ the $1$-dimensional torus and  consider $S = [0,1]$ and $X = L_\circ^2(T^1)$
endowed with the $L^2$-norm. Here, $L_\circ^2(T^1)$ denotes the space of $L^2$-functions on $T^1$ with zero mean, see  \cref{sec:periodicl2}.
Notice that $X$ is a Hilbert space, being a closed subset of $L^2(T^1)$, and therefore self-dual by the Riesz representation theorem.
We  define the functional $J : L^2_\circ(T^1) \times [0,1] \rightarrow [0,+\infty]$ as follows
\begin{equation}\label{eq:fraclaplaadaptive}
\displaystyle J(v,s) := \left\{\begin{array}{ll}
\displaystyle \frac{1}{s^\eta}\sqrt{\sum_{m \in \Z} |m|^{4s} |\hat v(m)|^2} & \text{if } s > 0,\\
0 & \text{if } v = 0, s = 0,\\
+\infty & \text{otherwise,}
\end{array}    
\right.
\end{equation}
where $\eta > 0$ is a fixed parameter. Here and subsequently $\hat v(m)$ denotes the $m$-th Fourier coefficient of $v$. The infinite infimal convolution of the functionals $J(\cdot,s)$ has the desired effect of penalizing lower regularity. Indeed, due to the factor $1/ s^\eta$ in the definition of $J$,  
a higher smoothness in the observation receives a lower weight in the regularization.


\begin{rem}\label{rem:coercivityrem}
We remark that if $s>0$ and $v \in H^{2s}(T^1)$, then $J(v,s) =  \frac{1}{s^\eta} \frac{1}{\sqrt{2\pi}} \|(-\Delta^s) v\|_{L^2(T^1)}$ is the rescaled $L^2$-norm of the $s$-fractional Laplacian (see Remark \ref{rem:frac-Hs}). 
\end{rem}
We  consider the following regularizer given by an infinite infimal convolution of $J(\cdot,s)$
\begin{align*}
R(v, \sigma) & =  \inf \left\{\int_0^1 J(u(s),s) \,  d\sigma(s)  : u \in L^1(([0,1],\sigma);L^2_\circ(T^1)) \text{ with } \int_0^1 u(s)\, d\sigma(s) = v\right\}
\end{align*}
and the corresponding denoising problem 
\begin{equation}\label{eq:inv_frac_1dim}
\inf_{\substack{v\in  L_\circ^2(T^1), \\ \sigma\in \mathcal{M}_+(S)}} \frac{1}{2}\|v - f\|^2_{L^2(T^1)} + \alpha R(v, \sigma) \,,
\end{equation}
where $f \in L^2(T^1)$ is a given observation and $\alpha >0$.
We now verify that this problem satisfies the assumptions  made in Section  \ref{sec:inverse}.
\begin{lemma}\label{lem:verass1dlaplacian}
The functional $J$ defined in \eqref{eq:fraclaplaadaptive}
 satisfies \ref{ass:convexlsc}, \ref{ass:lsc} and \ref{ass:coercivity} in Section \ref{sec:inverse}.
\end{lemma}
\begin{proof}
Assumption \ref{ass:convexlsc} is straightforward from the definition of $J(v,s)$ in \eqref{eq:fraclaplaadaptive}. Let us show Assumption \ref{ass:lsc}.
Let $\{(v^n, s_n)\}_n$ be a sequence in
$L_\circ^2(T^1) \times [0,1]$ such that $v^n \rightharpoonup v$ in $L_\circ^2(T^1)$ and $s_n \rightarrow s$.
Without loss of generality we can assume that 
\begin{align}\label{eq:fi}
\liminf_{n\rightarrow +\infty} J(v^n,s_n) = \lim_{n\rightarrow +\infty} J(v^n,s_n) < \infty\,.
\end{align}
In particular, for all but finitely many $n$ either $s_n>0$ or $(v^n,s_n) = (0,0)$. Since, if $(v^n,s_n) = (0,0)$ for all but finitely many the statement is clear, we can suppose without loss of generality that $s_n >0$ for every $n$.    
Note that $\lim_{n\rightarrow +\infty}|m|^{2s_n}\hat v^n(m) = |m|^{2s}\hat v(m)$ for every $m\in \Z$. Indeed, for every $m\in \Z$ it holds
\begin{align*}
    \lim_{n \rightarrow +\infty}|m|^{2s_n}\hat v^n(m) = \lim_{n \rightarrow +\infty} \frac{|m|^{2s_n}}{2\pi} \int_0^{2\pi} e^{-imx}v^n(x)\, dx = \frac{|m|^{2s}}{2\pi} \int_0^{2\pi}  e^{-imx}v(x)\, dx = |m|^{2s}\hat v(m)\,.
\end{align*}
Therefore, if $s>0$, by Fatou's Lemma we have 
\begin{align*}
   \liminf_{n\rightarrow +\infty} J(v^n,s_n) =  \liminf_{n\rightarrow +\infty} \frac{1}{s_n^\eta} \sqrt{\sum_{m \in \Z} |m|^{4s_n} |\hat v^n(m)|^2} \geq \frac{1}{s^\eta} \sqrt{\sum_{m \in \Z} |m|^{4s} |\hat v(m)|^2} = J(v,s)\,.
\end{align*}
In the case that $s = 0$, \eqref{eq:fi} implies
\begin{align*}
   0 \leq \lim_{n\rightarrow +\infty} \sqrt{\sum_{m \in \Z} |m|^{4s_n} |\hat v^n(m)|^2} \leq \left(\sup_{n \in \N} J(v^n,s_n)\right) \left( \lim_{n \rightarrow +\infty}s_n^\eta \right)= 0\,,
\end{align*}
and consequently $\lim_{n\rightarrow +\infty}\sum_{m \in \Z}|\hat v^n(m)|^2 = 0$. Since $v^n \rightharpoonup v$ we infer that $v = 0$ and 
$J(v,s) = 0$, showing that $\liminf_{n\rightarrow +\infty} J(v^n,s_n) \geq J(v,s)$.

Finally, let us verify Assumption \ref{ass:coercivity}. If $s >0$ and $v \in H^{2s}(T^1)$, then 
\begin{align}\label{eq:poiap}
J(v,s) =\frac{1}{s^\eta}\frac{1}{\sqrt{2\pi}} \|(-\Delta^s) v\|_{L^2(T^1)} \geq \frac{1}{\sqrt{2\pi}} \|v\|_{L^2(T^1)}  \,,  
\end{align}
thanks to Remark \ref{rem:frac-Hs} and the Poincar\'e inequality in Theorem \ref{thm:poinc}. If instead $v \in L^2(T^1) \smallsetminus H^{2s}(T^1)$ or $s=0$, the inequality $J(v,s) \geq \frac{1}{\sqrt{2\pi}} \|v\|_{L^2(T^1)}$ is immediate from the definition of $J(v,s)$.
\end{proof}

%

\subsubsection{The generalized conditional gradient method}
We now design a generalized conditional gradient method to solve \eqref{eq:inv_frac_1dim} following the general procedure outlined in Section \ref{sec:gcg}.
Define $B^\alpha_s = \{ v \in L^2_\circ(T^1) : \alpha J(v,s) \leq 1\}$. \cref{lem:extfractionallaplacian} with the choice $f(m) = \alpha^2|m|^{4s}$ provides the following characterization of the extremal points of $B^\alpha_s$:
\begin{align}\label{eq:ppp}
\Ext(B^\alpha_s) = \left\{ v\in L^2(T^1) : \alpha \frac{1}{s^\eta}\sqrt{\sum_{m \in \Z} |m|^{4s} |\hat v(m)|^2} = 1 \right\}\,.
\end{align}
Therefore, given an iterate $\mu^k = \sum_{i=1}^{N_k} c^k_i v^k_i \delta_{s^k_i}$ for $N_k \in \N$, $c_i^k >0$, $v_i^k \in \Ext(B^\alpha_s)$ and $s_i^k \in [0,1]$, the insertion step in \eqref{eq:insertion_step2} amounts to solving
\begin{equation}\label{eq:insertionlapl1dim}
\argmax_{\substack{s \in S, \\  v \in L^2_\circ(T^1)}} \Big\langle f -\sum_{i =1}^{N_k} c_i^k v_i^k ,v\Big\rangle_{L^2(T^1)} + \mathbbm{1}_{\{\frac{\alpha}{s^\eta}\sqrt{\sum_{m \in \Z} |m|^{4s} |\hat v(m)|^2} = 1\}}(v)\,.
\end{equation}
The following lemma provides an explicit formula for the inserted $v \in L_\circ ^{2}(T^1)$  for a fixed $s \in [0,1]$.

\begin{lemma}\label{lem:explixit1d}
Given $w \in L_\circ^2(T^1)$ and $s \in [0,1]$, consider the variational problem
\begin{equation}\label{eq:linearexplicit-1}
\max_{ v \in L^2_\circ(T^1)} \langle w,v\rangle_{L^2(T^1)} + \mathbbm{1}_{\{\frac{\alpha}{s^\eta}\sqrt{\sum_{m\in \Z} |m|^{4s}  |\hat  v(m)|^2} = 1\}}(v)\,.
\end{equation}
A solution of \eqref{eq:linearexplicit-1}  can be computed as 
\begin{equation}\label{eq:resull}
v^*(x) = \frac{1}{A}\sum_{m \in \Z\setminus \{0\}}|m|^{-4s}\hat w(m) e^{im\cdot x},
\end{equation}
where $A = \frac{\alpha}{s^\eta}\sqrt{\sum_{m \in \Z\setminus \{0\}} |m|^{-4s}|\hat w(m)|^2}$.
\end{lemma}
\begin{proof}
First note that existence of a minimizer for \eqref{eq:linearexplicit-1} follows by specializing the arguments of Section \ref{sec:insertions} where a general insertion step is discussed.
Since $w \in L_\circ^2(T^1)$, we can represent $w$ and every $v \in L_\circ ^{2}(T^1)$ as 
\begin{equation}
w(x) =  \sum_{m \in \Z\setminus \{0\}} \hat w(m) e^{im\cdot x}\,,\quad v(x) =   \sum_{m \in \Z\setminus \{0\}} \hat v(m) e^{im\cdot x} \,.
\end{equation}
Therefore, with Plancherel's formula \eqref{eq:plancherel}, the variational problem \eqref{eq:linearexplicit-1} can be rewritten equivalently in the Fourier domain as 
\begin{equation}\label{eq:infourier-1}
\argmax_{ \hat v \in\ell^2(\Z) }     \langle \hat w, \hat v\rangle_{\ell^2(\Z)} \quad \text{s.t.} \ \ \ \frac{\alpha^2}{s^{2\eta}} \sum_{m \in \Z}   |m|^{4s}  |\hat v(m)|^2  = 1 \,.
\end{equation}
By the substitutions 
\begin{align}\label{eq:subs1}
    \tilde w(m) = \hat w(m) \frac{s^{\eta}}{\alpha } |m|^{-2s} \quad  \forall m \in \Z \setminus \{0\} \text{ and } \tilde w(0) = 0\,,
\end{align}
\begin{align}\label{eq:subs2}
    \tilde v(m) = \hat v(m) \frac{\alpha}{s^{\eta}} |m|^{2s} \quad  \forall m \in \Z\,,
\end{align}    
we can rewrite \eqref{eq:infourier-1} as $\argmax_{ \tilde v \in\ell^2(\Z) }  \langle \tilde w, \tilde v\rangle_{\ell^2(\Z)}$ under the constraint $\|\tilde v(m)\|_{\ell^2(\Z)}  = 1$ that has the unique solution $\tilde v^*(m) = \tilde w(m)/\|\tilde w\|_{\ell^2(\Z)}$. 
Therefore, by using \eqref{eq:subs1}, \eqref{eq:subs2}  and the inverse Fourier transform, a simple computation leads to \eqref{eq:resull}.
\end{proof}

In Algorithm \ref{alg:1dlaplalgorithm} we rewrite the generalized conditional gradient method  specialized for~\eqref{eq:inv_frac_1dim}. 
%
In the insertion step at the $k$-th iteration of the algorithm, the newly inserted atom is computed by first applying \cref{lem:explixit1d} with $w = f - \sum_{i =1}^{N_k} c_i^k v_i^k $ to find the optimal solution for a fixed $s \in [0,1]$, denoted by $v^k_{N_k+1}(s)$. Then, $v^k_{N_k+1}(s)$ is optimized with respect to $s$ by solving 
\begin{align}\label{eq:ott}
\max_{s \in [0,1]}\, \langle  p^k , v^k_{N_k+1}(s)\rangle\,.
\end{align}
We remark that the computation of \eqref{eq:ott} is performed using a Python implementation of a basinhopping–type algorithm whose performance is good enough to compute $\argmax_{s \in [0,1]}\, \langle  p^k , v^k_{N_k+1}(s)\rangle$
efficiently and with satisfactory accuracy.

  \begin{algorithm}[ht!]
  \caption{Generalized conditional gradient method for the infinite infimal convolution of fractional Laplacians with adaptive order}\label{alg:1dlaplalgorithm}
  \begin{algorithmic}
    \STATE  \STATE Initialize: $v^0 = 0$, $\sigma^0 = 0$,
$\Xi^0 = \emptyset$
%
   \STATE \textbf{for} \(k=0,1,2\dots \ \textbf{do}\)
  \STATE \qquad $p^k \leftarrow  f -\sum_{i=1}^{N_k}  c^k_i v^k_i$
  \STATE \qquad $v^k_{N_k+1}(s)(x) =  \frac{s^\eta}{\alpha\sqrt{\sum_{m \in \Z \setminus \{0\}} |m|^{-4s}|\hat p^k(m)|^2}}\sum_{m \in \Z\setminus \{0\}}  |m|^{-4s}\hat p^k(m) e^{im\cdot x}$
  \STATE \qquad \(s^k_{N_k+1} \in \argmax_{s \in [0,1]} \langle  p^k , v^k_{N_k+1}(s)\rangle\)
  \STATE \qquad $v^k_{N_k+1} \leftarrow v^k_{N_k+1}(s^k_{N_k+1})$
   \STATE \qquad \textbf{if}  $k \geq 1$, $\langle p^k,v^k_{N_k+1}\rangle  \leq 1$ or $( v^k_{N_k+1}, s^k_{N_k+1}) \in \Xi^k$ \textbf{then}
      \STATE \qquad \qquad \textbf{return} $(v^k,\sigma^k)$
         \STATE \qquad  \textbf{end if}
        \STATE \qquad $(\hat c^k_1, \ldots, \hat c^k_{N_k +1}) \in  \argmin_{c \in \R_+^{N_k+1}}\frac{1}{2}\|\sum_{i=1}^{N_k+1} c_i v^k_i - f\|_{L^2(T^1)}^2 +  \alpha\sum_{i=1}^{N_k+1} c_i$
   \STATE \qquad $v^{k+1} \leftarrow \sum_{i = 1}^{N_{k}+1} \hat c_i^k v_i^{k}$, \  $\sigma^{k+1} \leftarrow  \sum_{s \in \{s^k_i : \,i =1,\ldots,N_k+1\}} \big\|\sum_{j \in \{1,\ldots, N_k +1 :\, s^k_j = s\}}  \hat c^k_j v^k_j\big\|_X\delta_{s}$
        \STATE \qquad $\Xi^{k+1}  \leftarrow \big\{(v_i^k, s_i^k) \in \Xi^{k} \cup \{(v^k_{N_k+1},s^k_{N_k+1})\} : \hat c^k_i >0\big\}$
         \STATE \qquad $N_{k + 1}  \leftarrow \# \Xi^{k+1}$, \ $k \leftarrow k+1$
   \STATE  \textbf{end for}
   \STATE \textbf{return} $(v^k,\sigma^k)$
  \end{algorithmic}
  \end{algorithm}
  
\subsubsection{Bandlimited signals}\label{sec:band}
In practice, the implementation of Algorithm \ref{alg:1dlaplalgorithm} is achieved by truncating the Fourier expansion of the signals and summing only over a finite amount of frequencies in the computation of the fractional Laplacian, to avoid infinite sums. 
Therefore, at implementation time, the signals space $X$ is essentially finite-dimensional. We stress, however, that this discretization is the necessary consequence of the characterization of $\Ext(B^\alpha_s)$ given in \eqref{eq:ppp}. For other types of regularizers, such as TV-type penalties, a pure infinite-dimensional algorithm (both in $X$ and $S$) could be implemented.  We also remind the reader that while $X$ is implemented as a finite  dimensional space, $S$ is never discretized and our algorithm can be interpreted as an off-the-grid approach on the parameters $s\in S$. \\
In the experiment in Section \ref{sec:numexp1} we set the bound on the considered frequencies to $F = 256$.

\subsubsection{Numerical experiments}\label{sec:numexp1}
Consider a signal constructed in the following way.
Define
\begin{equation}
 v(x) = \frac{1}{2\pi}  \sum_{m \in \Z\setminus \{0\}} \hat v(m) e^{im\cdot x}\,,
\end{equation}
where $\hat v(m)$ is such that $\hat v(3) = \hat v(-3) = 8$, $\hat v(20) = \hat v(-20) = 2$ and $\hat v(m) = 0$ otherwise. We then construct the data $f$ by corrupting $v$ with Gaussian noise of standard deviation $7 \cdot 10^{-2}$. With these choices the relative noise level amounts to $\|f - v\|_{L^2}/\|v\|_{L^2} = 8.7 \cdot 10^{-2}$.
In the definition of $J$ and in the inverse problem in \eqref{eq:inv_frac_1dim} we make the parameter choices, $\eta = 2$, $\alpha = 1.5\cdot 10^{-3}$.
Figure \ref{fig:1d} shows the reconstruction obtained by Algorithm \ref{alg:1dlaplalgorithm} and Figure \ref{fig:1d_sparse} shows the orders $s \in [0,1]$ selected by the conditional gradient method and a convergence graph where we plot an approximation of the residual $r_{\mathcal{F}_\alpha}$ \eqref{def:residualinverse} at each iteration in logarithmic scale. More precisely, denoting by $(v^k,\sigma^k) = (\sum_{i = 1}^{N_{k}} c_i^{k} v_i^{k}$, $\sum_{s \in \{s^k_i : \,i =1,\ldots,N_k\}} \big\|\sum_{j \in \{1,\ldots, N_k :\, s^k_j = s\}}  c^k_j v^k_j\big\|_X\delta_{s})$ the iterate generated by Algorithm \ref{alg:1dlaplalgorithm} and by $(v^{\bar k},\sigma^{\bar k})$ the output, we plot the approximate residual $\hat r_{\mathcal{F}_\alpha}(\mu_k)$ defined as
\begin{align}\label{eq:approxresidual}
    \hat r_{\mathcal{F}_\alpha}(\mu_k) =\frac{1}{2}\|v^k - f\|_{L^2(T^1)}^2 + \alpha \sum_{i=1}^{N_k} c_i^k J(v^{k}_i,s^{k}_i) -  \frac{1}{2}\|v^{\bar k} - f\|_{L^2(T^1)}^2 - \alpha \sum_{i=1}^{N_{\bar k}} c_i^{\bar k} J(v^{\bar k}_i,s^{\bar k}_i)\,. 
\end{align}
Compared to \eqref{def:residualinverse} we approximate the minimum of \eqref{eq:inv_frac_1dim} by the energy of the last iterate $(v^{\bar k},\sigma^{\bar k})$.
Moreover, we replace the infinite infimal convolutions $R(v^{k},\sigma^{k})$ and $R(v^{\bar k},\sigma^{\bar k})$ by the upper bounds $\sum_{i=1}^{N_{k}} c_i^{k} J(v^{k}_i,s^{k}_i)$ and $\sum_{i=1}^{N_{\bar k}} c_i^{\bar k}J(v^{\bar k}_i,s^{\bar k}_i)$. This is justified, since \cite[Theorem 4.4]{bredies2021linear} guarantees that $\frac{1}{2}\|v^k - f\|_Y^2 + \alpha \sum_{i=1}^{N_k} c_i^k J(v^{k}_i,s^{k}_i)$ well approximates \eqref{eq:inv_frac_1dim} if $k$ is big enough.

\begin{figure}[ht!]
\begin{minipage}{0.33\textwidth}
\vspace*{3mm}
\centering
\includegraphics[width=5cm,height=4.8cm]{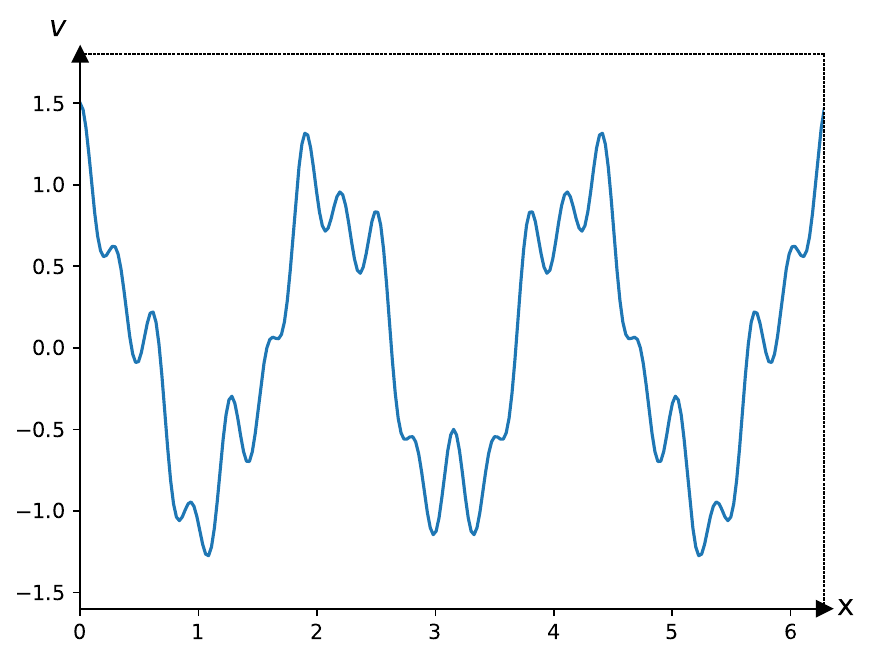}
\end{minipage} 
\begin{minipage}{0.33\textwidth}
\centering
\vspace*{3mm}
\includegraphics[width=5cm,height=4.8cm]{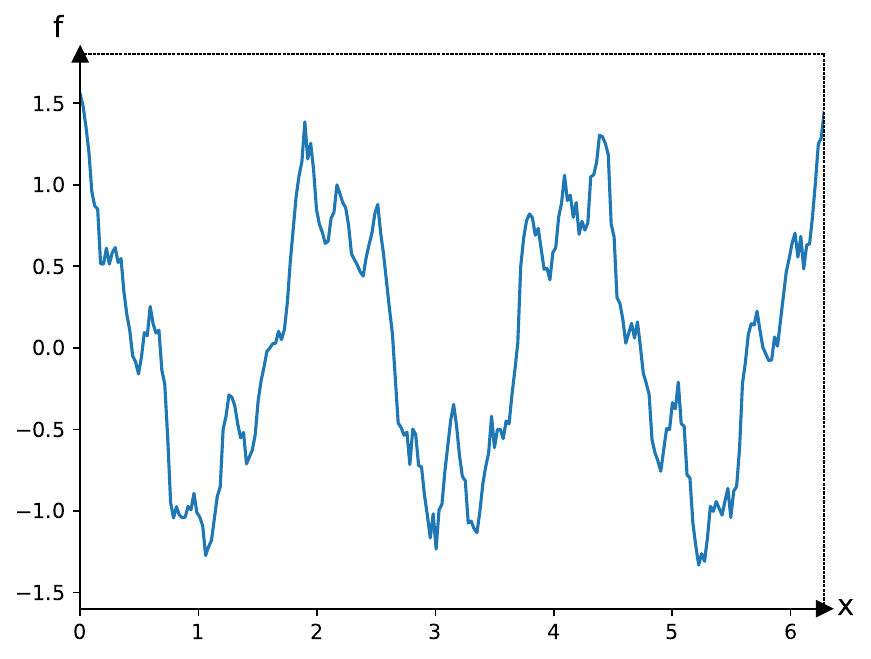}
\end{minipage} 
\begin{minipage}{0.33\textwidth}
\centering
\vspace*{3mm}
\includegraphics[width=5cm,height=4.8cm]{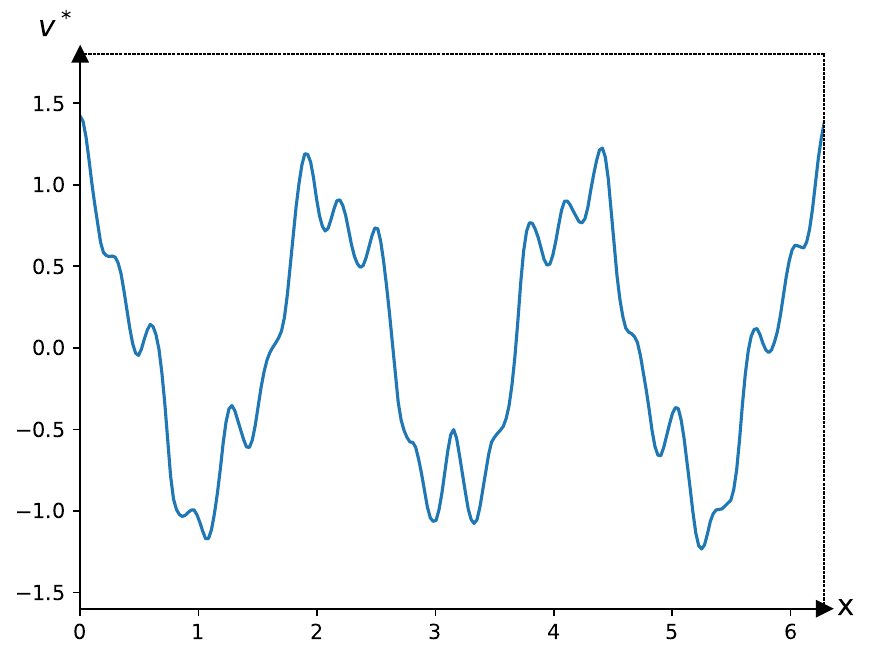}
\end{minipage}\\
\begin{minipage}{0.50\textwidth}
\vspace*{3mm}
\flushright
\includegraphics[width=4.4cm,height=3.6cm]{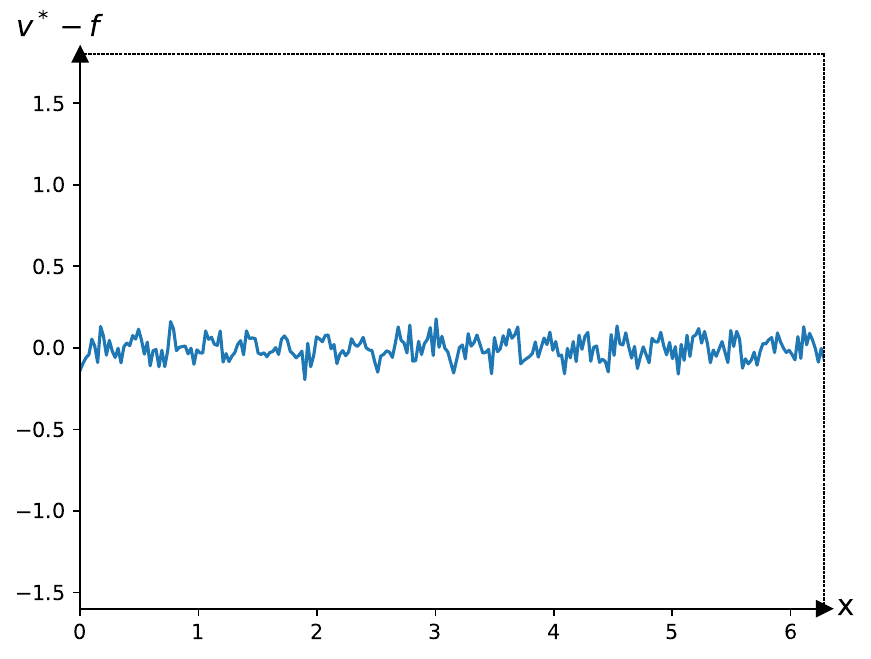}
\end{minipage} 
\begin{minipage}{0.50\textwidth}
\flushleft
\vspace*{3mm}
\includegraphics[width=4.4cm,height=3.6cm]{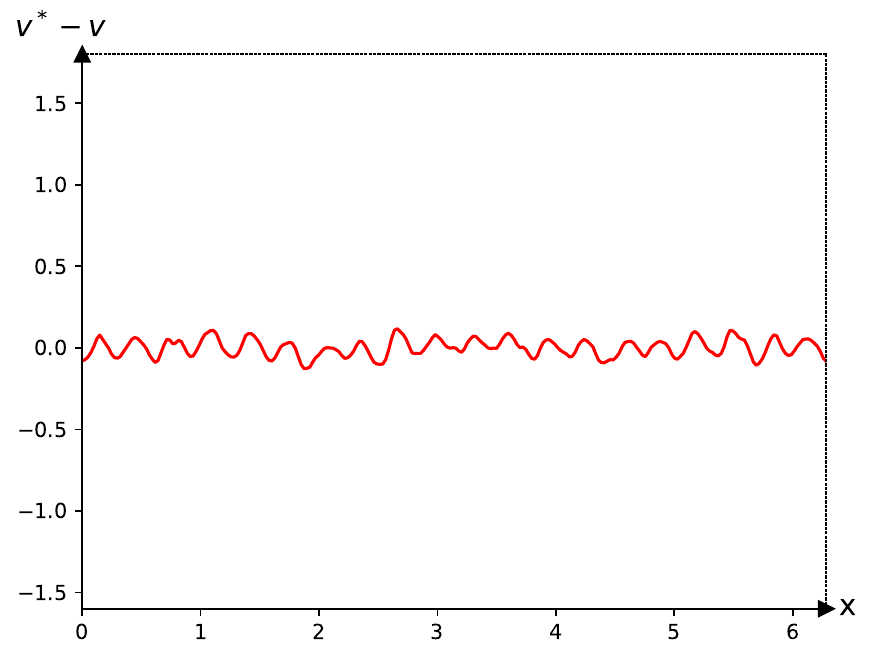}
\end{minipage} 
\caption{\small Above, from left to right: the ground truth $v$; the data $f$, i.e., the ground truth $v$ corrupted by Gaussian noise; the reconstruction $v^*$ provided by Algorithm \ref{alg:1dlaplalgorithm}. Below, from left to right: $v^* - f$, the noise removed by the algorithm; $v^* - v$, the natural bias of the Laplacian regularization that is favouring constant signals.}\label{fig:1d}
\end{figure}

Note that the obtained reconstruction resembles a Laplacian regularization of a noisy signal \cite{bartels2020parameter, antil2017spectral} and it is hard to give a clear interpretation on the optimal orders $s \in [0,1]$ selected by our algorithm. We stress however that the goal of this first example is solely to illustrate the main features of our general algorithm and demonstrate numerically the expected sublinear convergence of the residuals. We also note that in this example the rate of convergence appears to be linear, at least asymptotically. This has to be expected \cite{bredies2021linear}, but the theoretical verification of this behaviour is outside of the scope of this work. 

\textbf{Computational time:} For this experiment, Algorithm \ref{alg:1dlaplalgorithm} terminates in 3.22 seconds. While, slower than a standard higher-order denoising method for signals, our adaptive approach is still feasible for many real-world applications. 

\begin{figure}[ht!]
\begin{minipage}{0.49\textwidth}
\centering
\includegraphics[width=6cm,height=5cm]{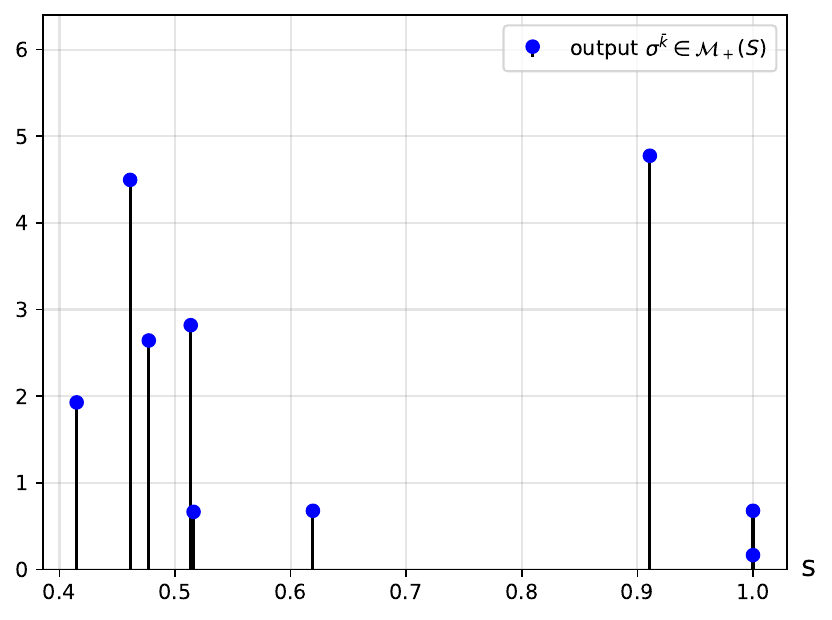}
\end{minipage}\hfill%
\begin{minipage}{0.49\textwidth}
\centering
\includegraphics[width=6cm,height=5cm]{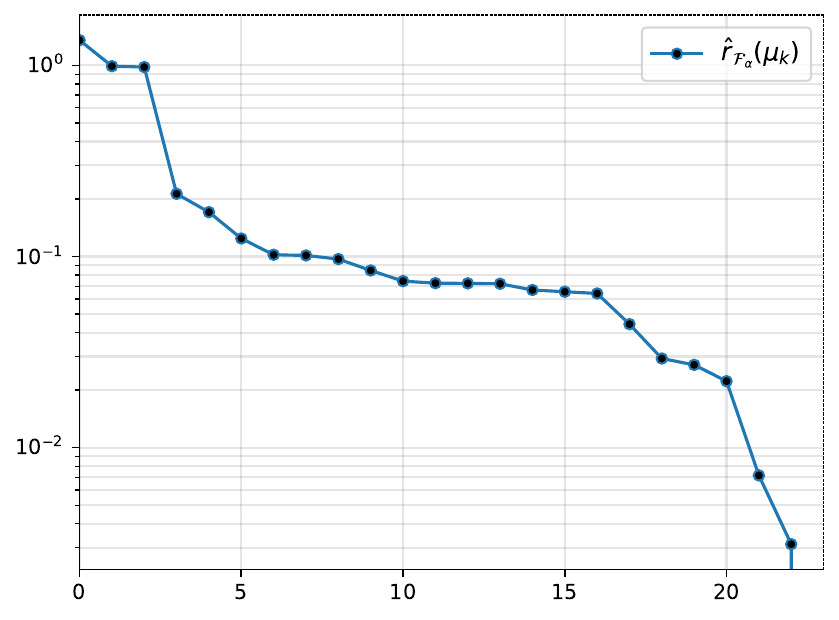}
\end{minipage}  
\caption{\small On the left: the output $\sigma^{\bar k} \in \mathcal{M}_+(S)$ of Algorithm \ref{alg:1dlaplalgorithm}; its support consists of the orders $s\in [0,1]$ selected by Algorithm \ref{alg:1dlaplalgorithm}.
On the right: a convergence graph, showing the approximate residual $\hat r_{\mathcal{F}_\alpha}$ \eqref{eq:approxresidual} computed at each iteration.}\label{fig:1d_sparse}
\end{figure}

\subsection{\texorpdfstring{Adaptively selecting the anisotropies of fractional Laplacians in an infinite\\ infimal convolution}{Adaptively selecting the anisotropies of fractional Laplacians in an infinite infimal convolution}}\label{sec:anisotropic}

In this section we design an infinite infimal convolution-type regularizer that is based on a variant of the anisotropic fractional Laplacian. The anisotropy  is determined by a parameter $s \in [0,\pi]$ associated to the direction $(\cos(s), \sin(s)) \in \mathbb{S}^1$. We construct an infinite infimal convolution regularizer that automatically selects the directions of higher oscillation of the noisy image $f$ and applies a lower regularization in such directions.
In this way, the regularizer learns how to preserve the anisotropy of the data and removes the noise that is assumed to be isotropic. 

\subsubsection{Setting}
We denote by $T^2$ the $2$-dimensional torus and  consider $S = [0,\pi]$ and $X = L^2(T^2)$ endowed with the $L^2$-norm. 
We define the operator $J : L^2(T^2) \times [0,\pi] \rightarrow [0,+\infty]$ as follows
\begin{equation}\label{eq:fraclapla}
J(v,s) := \sqrt{\sum_{m \in \Z^2} (|m \cdot (\cos(s),\sin(s))| + \zeta|m| + \omega)^{4\gamma} |\hat v(m)|^2 }\,,
\end{equation}
where $\gamma \in [0,1]$, $\zeta \geq 0$ and $\omega >0$ are fixed parameters.  
\begin{rem}
The parameter $\omega >0$ in the definition of $J$ is  essential for ensuring the coercivity of $J$ at frequencies $m\in \Z^2$ such that $m \cdot (\cos(s),\sin(s)) = 0$. 
Since the penalization $ \sqrt{\sum_{m \in \Z^2} (\zeta |m|)^{4\gamma} |\hat v(m)|^2}$ for $\zeta >0$ models isotropic regularization of a specific rate, the corresponding term in $J$ can therefore be useful to denoise images that exhibit isotropic behavior in some regions. If $\zeta$ is large then our regularizer  acts similarly to a standard fractional Laplacian of order $\gamma$, since the anisotropic part is negligible. Further variants of \eqref{eq:fraclapla} are possible. For example, one could allow for different orders of smoothness in a direction and its orthogonal, by considering the penalization 
\begin{align*}
\sqrt{\sum_{m \in \Z^2} \left(|m \cdot (\cos(s),\sin(s))|^{4\gamma_1} + |m \cdot (-\sin(s),\cos(s))|^{4\gamma_2} + \omega\right) |\hat v(m)|^2 } 
\end{align*}
for $\gamma_1>\gamma_2 > 0$.
\end{rem}

We consider the following regularizer given by the infinite infimal convolution of $J(\cdot,s)$:
\begin{equation}
R(v, \sigma) =  \inf\left\{ \int_0^\pi J(u(s),s)\, d\sigma(s) : u \in L^1(([0,\pi],\sigma);L^2(T^2)) \text {with } \int_0^\pi u(s)\, d\sigma(s) \right\}\,,
\end{equation}
and the corresponding denoising problem 
\begin{equation}\label{eq:denoising-anosotr-laplace}
\inf_{\substack{v\in  L^2(T^2),\\ \sigma\in \mathcal{M}_+(S)}} \frac{1}{2}\|v - f\|^2_{L^2(T^2)} + \alpha R(v, \sigma) \,,
\end{equation}
where $f \in  L^2(T^2)$ is a given measurement. 

We  verify that this problem satisfies the assumptions made in Section  \ref{sec:inverse}.
\begin{lemma}
The functional $J$ defined in \eqref{eq:fraclapla}
 satisfies \ref{ass:convexlsc},  \ref{ass:lsc} and \ref{ass:coercivity} in Section \ref{sec:inverse}.
\end{lemma}

\begin{proof}
Assumption \ref{ass:convexlsc} is straightforward from the definition of $J$. Assumption \ref{ass:lsc} can be proven similarly to Lemma \ref{lem:explixit1d} using Fatou's lemma.  
Assumption \ref{ass:coercivity} follows immediately from the definition of $J$ since $\omega >0$.
\end{proof}

\subsubsection{The generalized conditional gradient method}

We now describe a generalized conditional gradient method to solve \eqref{eq:denoising-anosotr-laplace}, following the general procedure outlined in Section \ref{sec:gcg}. For $B^\alpha_s = \{ v \in L^2(T^2) : \alpha J(v,s) \leq 1\}$, \cref{lem:extfractionallaplacian2} with the choice $f(m) = \alpha^2 (|m \cdot (\cos(s),\sin(s))| + \zeta|m| + \omega)^{4\gamma}$ provides the following characterization of the extremal points of $B^\alpha_s$:
\begin{align}\label{eq:exta}
\Ext(B_s) = \left\{ v \in L^2(T^2) : \alpha \sqrt{\sum_{m \in \Z^2} (| m \cdot (\cos(s),\sin(s)) | + \zeta|m| + \omega)^{4\gamma} |\hat v(m)|^2} = 1\right\}\,.
\end{align}
Therefore, given an iterate $\mu^k = \sum_{i=1}^{N_k} c^k_i v^k_i \delta_{s^k_i}$ with $N_k \in \N$, $c_i^k >0$, $s_i^k \in [0,\pi]$ and $v_i^k \in \Ext(B^\alpha_s)$ as in \eqref{eq:exta}, the insertion step in \eqref{eq:insertion_step2} becomes
\begin{equation}\label{eq:insertionlapl1dim2}
\argmax_{\substack{s \in [0,\pi], \\  v \in L^{2}(T^2)}} \Big\langle f-\sum_{i =1}^{N_k} c_i^k v_i^k,v\Big\rangle_{L^2(T^2)} + \mathbbm{1}_{{\{\alpha\sqrt{\sum_{m\in \Z^2} (| m \cdot (\cos(s),\sin(s)) | + \zeta|m| + \omega)^{4\gamma} |\hat v(m)|^2}} = 1\}}(v)\,.
\end{equation}
The following lemma gives an explicit formula for the inserted $v \in H^{2\gamma}(T^2)$  for a fixed $s \in [0,\pi]$. 
\begin{lemma}\label{lem:insanis}
Given $w \in L^2(T^2)$ and $s \in [0,\pi]$ consider the variational problem
\begin{equation}\label{eq:linearexplicitanis}
\max_{ v \in L^{2}(T^2)} \langle w,v\rangle_{L^2(T^2)} + \mathbbm{1}_{\{\alpha\sqrt{\sum_{m \in \Z^2} (|m \cdot (\cos(s),\sin(s)) | + \zeta|m| + \omega)^{4\gamma} |\hat v(m)|^2} = 1\}}(v)\,.
\end{equation}
A solution of \eqref{eq:linearexplicitanis}  can be computed as 
\begin{equation}
v^*(x) = \frac{1}{A}\sum_{m \in \Z^2} (| m \cdot (\cos(s),\sin(s)) | + \zeta|m| + \omega)^{-4\gamma}\hat w(m) e^{im\cdot x}\,
\end{equation}
where 
\begin{equation}
A = \alpha\sqrt{\sum_{m \in \Z^2}  (|m \cdot (\cos(s),\sin(s)) | + \zeta|m| + \omega)^{-4\gamma} |\hat w(m)|^2}\,.
\end{equation}
\end{lemma}
\begin{proof}
Since the proof is similar to that of Lemma \ref{lem:explixit1d} we omit it.
\end{proof}

  \begin{algorithm}[t!]
  \caption{Generalized conditional gradient method for infinite infimal convolution of fractional Laplacians with adaptive anisotropies}\label{alg:anisalgorithm}
  \begin{algorithmic}
    \STATE \STATE Initialize: $v^0 = 0$, $\sigma^0 = 0$,
$\Xi^0 = \emptyset$
   \STATE \textbf{for} \(k=0,1,2\dots \ \textbf{do}\)
  \STATE \qquad $p^k \leftarrow   f -\sum_{i=1}^{N_k}  c^k_i v^k_i$
  \STATE \qquad $v^k_{N_k+1}(s)(x) = \frac{\alpha^{-1}}{\sqrt{\sum_{m \in \Z^2} \frac{|\hat p^k(m)|^2}{(|m \cdot (\cos(s),\sin(s)) | + \zeta|m| + \omega)^{4\gamma}}}}\sum_{m \in \Z^2}\frac{\hat p^k(m)}{(|m \cdot(\cos(s),\sin(s)) | + \zeta|m| + \omega)^{4\gamma}}e^{im \cdot x}$
  \STATE \qquad \(s^k_{N_k+1} \in \argmax_{s \in [0,\pi]} \langle  p^k , v^k_{N_k+1}(s)\rangle\)
 \STATE \qquad $v^k_{N_k+1} \leftarrow v^k_{N_k+1}(s^k_{N_k+1})$
   \STATE \qquad \textbf{if} $k \geq 1$,  $\langle p^k,v^k_{N_k+1}\rangle  \leq 1$ or $(v^k_{N_k+1}, s^k_{N_k+1}) \in \Xi^k$ \textbf{then}
      \STATE \qquad \qquad \textbf{return} $(v^k,\sigma^k)$
         \STATE \qquad  \textbf{end if}
    \STATE \qquad $(\hat c_1, \ldots, \hat c_{N_k +1}) \in  \argmin_{c \in \R_+^{N_k+1}}\frac{1}{2}\|\sum_{i=1}^{N_k+1} c_i v^k_i - f\|_{L^2(T^2)}^2 +  \alpha\sum_{i=1}^{N_k+1} c_i$
   \STATE \qquad $v^{k+1} \leftarrow \sum_{i = 1}^{N_{k}+1} \hat c_i v_i^{k}$, \ $\sigma^{k+1} \leftarrow  \sum_{s \in \{s^k_i : \,i =1,\ldots,N_k+1\}} \big\|\sum_{j \in \{1,\ldots, N_k +1 :\, s^k_j = s\}}   \hat c^k_j v^k_j\big\|_X\delta_{s}$
         \STATE \qquad $\Xi^{k+1}  \leftarrow \big\{(v_i^k, s_i^k) \in \Xi^{k} \cup \{(v^k_{N_k+1},s^k_{N_k+1})\} : \hat c^k_i >0\big\}$
         \STATE \qquad $N_{k + 1}  \leftarrow \# \Xi^{k+1}$, \ $k \leftarrow k+1$
   \STATE  \textbf{end for}
   \STATE  \textbf{return} $(v^k,\sigma^k)$
  \end{algorithmic}
  \end{algorithm}

In Algorithm \ref{alg:anisalgorithm} we rewrite the generalized conditional gradient method  specialized to~\eqref{eq:denoising-anosotr-laplace}. 
%
In the insertion step at the $k$-th iteration of the algorithm, the newly inserted atom is computed by first applying \cref{lem:insanis} to find the optimal solution for a fixed $s \in [0,\pi]$, denoted by $v^k_{N_k+1}(s)$. Then, similarly to Section \ref{sec:1dadaptive}, $v^k_{N_k+1}(s)$ is optimized with respect to $s$ by solving 
\begin{align}\label{eq:ott2}
\argmax_{s \in [0,\pi]}\, \langle  p^k , v^k_{N_k+1}(s)\rangle\,.
\end{align}
Similarly to Section \ref{sec:1dadaptive}, we remark that the computation of \eqref{eq:ott2} is performed using a Python implementation of a basin-hopping--type algorithm whose performance is good enough to compute $\argmax_{s \in [0,\pi]} \langle  p^k , v^k_{N_k+1}(s)\rangle$ efficiently and with satisfactory accuracy.

\subsubsection{Bandlimited images}
Similarly to Algorithm \ref{alg:1dlaplalgorithm} (see Section \ref{sec:band}), the implementation of Algorithm \ref{alg:anisalgorithm} is also achieved by truncating the Fourier expansion of the images and summing only over a finite amount of frequencies in the computation of the fractional Laplacian, to avoid infinite sums. 
Similar considerations to the ones on Section \ref{sec:band} can be made. At implementation time, the image space $X$ is essentially-finite dimensional, while the direction space $S$ is never discretized. Therefore, for this numerical example, our algorithm can be interpreted as an off-the-grid approach on the parameters $s\in S$. \\
In first two experiments in Section \ref{sec:numm} we set the bound on the considered frequencies to $F = 128$, while in the last experiment we set the bound to $F = 256$.

\subsubsection{Numerical experiments}\label{sec:numm}

In this section we present several numerical examples of the applications of Algorithm \ref{alg:anisalgorithm} for denoising images with high frequencies in certain directions. 
We show that the regularizer is able to learn such directions and preserve the related frequencies in the reconstruction. For each example we plot the output measure $\sigma^{\bar k} \in \mathcal{M}_+(S)$ that is supported on the $s \in [0,\pi]$ selected by our algorithm. Note that the points $s= 0$ and $s=\pi$ correspond to the same direction $(\cos(s), \sin(s))$, and thus they should be identified. Moreover, similarly to the numerical experiments in Section \ref{sec:numexp1}, we plot an approximation of the residual $r_{\mathcal{F}_\alpha}$ \eqref{def:residualinverse} at each iteration in logarithmic scale. More precisely, denoting by $(v^k,\sigma^k) = (\sum_{i = 1}^{N_{k}} c_i^{k} v_i^{k}$, $\sum_{s \in \{s^k_i : \,i =1,\ldots,N_k+1\}} \big\|\sum_{j \in \{1,\ldots, N_k +1 :\, s^k_j = s\}}   \hat c^k_j v^k_j\big\|_X\delta_{s})$ the iterate generated by Algorithm \ref{alg:anisalgorithm} and by $(v^{\bar k},\sigma^{\bar k})$ the output, we plot the approximate residual $\hat r_{\mathcal{F}_\alpha}(\mu_k)$ defined as in \eqref{eq:approxresidual}. We refer to the discussion in Section \ref{sec:numexp1} as a justification for this choice.
In all examples we observe a convergence rate that is at least sublinear, as ensured by Theorem \ref{thm:convergence}.

In the first  example we aim to denoise a black-and-white grid of size $128 \times 128$ that is corrupted with Gaussian noise with standard deviation  $0.3$, see Figure \ref{fig:grid}. With these choices the relative noise level amounts to $\|f - v\|_{L^2}/\|v\|_{L^2} = 3.6 \cdot 10^{-1}$. In the definition of $J$ in \eqref{eq:fraclapla} we set  $\alpha = 5.5$, $\gamma = 0.25$, $\omega = 10^{-3}$, and $\zeta = 10^{-3}$. We compare the reconstruction obtained by our algorithm with a reconstruction produced using as regularizer the $L^2$-norm of the $\frac{1}{4}$-fractional Laplacian. We observe that our method reconstructs sharper edges. This is because our functional reduces the amount of regularization in the directions of the anisotropies of the image.

\begin{figure}[ht!]
\begin{minipage}{0.23\textwidth}
\centering
\includegraphics[width=3.5cm,height=3.5cm]{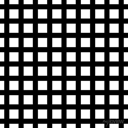}
\end{minipage}\ \ \ 
\begin{minipage}{0.23\textwidth}
\centering
\includegraphics[width=3.5cm,height=3.5cm]{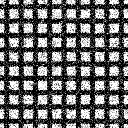}
\end{minipage}\ \ \ 
\begin{minipage}{0.23\textwidth}
\centering
\includegraphics[width=3.5cm,height=3.5cm]{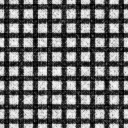}
\end{minipage} \ \ 
\begin{minipage}{0.23\textwidth}
\centering
\includegraphics[width=3.5cm,height=3.5cm]{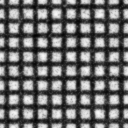}
\end{minipage} \ \ 
\caption{\small  From left to right:  ground truth;  data $f$, i.e., ground truth corrupted by Gaussian noise;  reconstruction obtained with our algorithm;  reconstruction obtained with  $\frac{1}{4}$-fractional Laplacian regularization.} \label{fig:grid}
\end{figure}

\begin{figure}[ht!]
\begin{minipage}{0.49\textwidth}
\centering
\vspace*{-1mm}
\includegraphics[width=6cm,height=5cm]{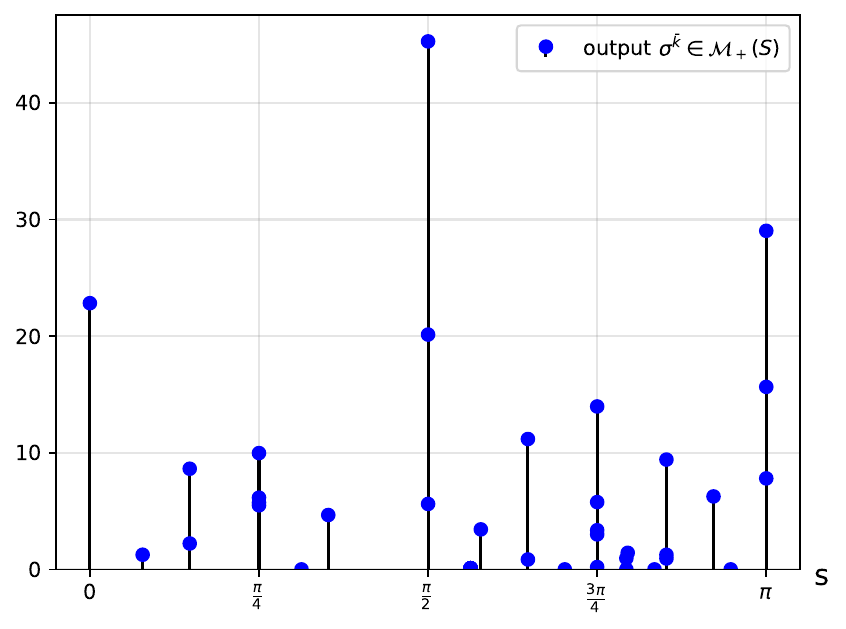}
\end{minipage}\ \ %
\begin{minipage}{0.49\textwidth}
\centering
\includegraphics[width=6cm,height=5cm]{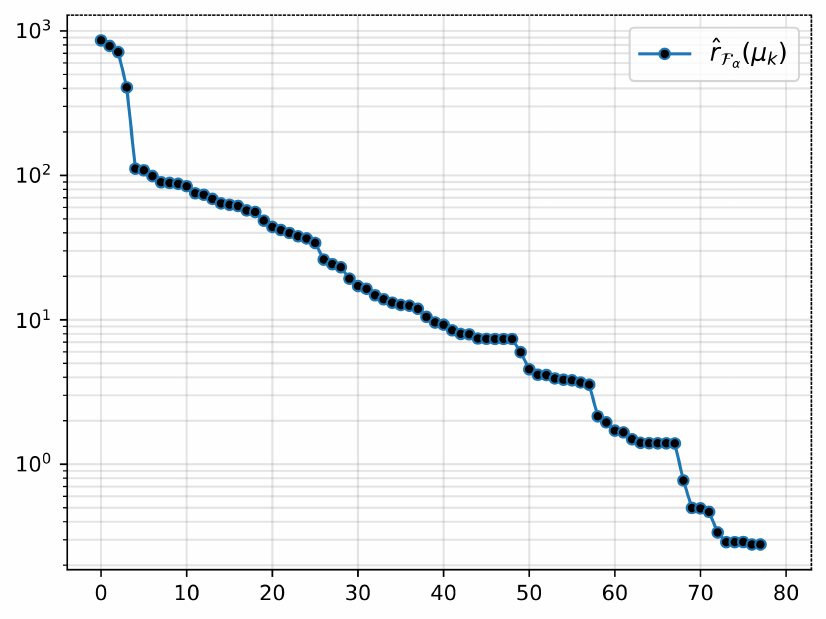}
\end{minipage} 
\caption{\small On the left: the measure $\sigma^{\bar k} \in \mathcal{M}_+(S)$ obtained in output by Algorithm \ref{alg:anisalgorithm}; its support consists of the directions of the dominant anisotropies of the image $s\in [0,\pi]$ selected by Algorithm \ref{alg:anisalgorithm}.
On the right: a convergence graph, showing the approximate residual $\hat r_{\mathcal{F}_\alpha}$ according to \eqref{eq:approxresidual}
computed at each iteration.}\label{fig:convgrid}
\end{figure}

In Figure \ref{fig:convgrid} we show the output measure $\sigma^{\bar k} \in \mathcal{M}_+(S)$ supported on the angles $s \in [0,\pi]$ selected by our algorithm and we plot the decay of the approximate residual $\hat r_{\mathcal{F}_\alpha}$ according to \eqref{eq:approxresidual}. We see that the selected angles concentrate around~$0$, ~$\frac{\pi}{2}$ and $\pi$ which are the most evident anisotropies of the image. Note again that the points $s= 0$ and $s=\pi$ correspond to the same direction $(\cos(s), \sin(s))$, and thus they should be identified.
We also observe that due to the imperfect anisotropy of the grid, the optimal $\sigma \in \mathcal{M}_+(S)$ also has significant mass at $\frac{\pi}{4}$ and $\frac{3}{4}\pi$. The convergence plot illustrates the sublinear convergence  of our algorithm. 
 \\
\textbf{Computational time:} For this experiment, Algorithm \ref{alg:anisalgorithm} terminates in $658.37$ seconds. On the other hand the reconstruction obtained with the $\frac{1}{4}$-fractional Laplacian terminates in $22.87$ seconds. We remark that the algorithm for performing the fractional Laplacian regularization is simply obtained by removing from Algorithm \ref{alg:anisalgorithm} the computation of the argmax in $s$. It is however possible to achieve considerably faster computational times with other methods \cite{bartels2020parameter}.

In the second example we denoise a more complicated black-and-white geometric pattern corrupted by Gaussian noise with standard deviation equal to $0.3$, see Figure \ref{fig:inc}. With these choices the relative noise level amounts to $\|f - v\|_{L^2}/\|v\|_{L^2} = 2.3 \cdot 10^{-1}$. In the definition of $J$ in \eqref{eq:fraclapla} we set  $\alpha = 5.5$, $\gamma = 0.25$, $\omega = 10^{-3}$, and $\zeta = 10^{-3}$.
Similarly to the first example, we observe  that our method reconstructs sharper edges thanks to its adaptivity. 
In Figure \ref{fig:inc} we show the directions selected by the algorithm and the convergence plot of the approximate residual \eqref{def:residualinverse}. In this example the selected $s \in [0,\pi]$  cluster around the values~$\frac{1}{4} \pi$ and~$\frac{3}{4} \pi$ that represent the  anisotropies of the original image. The approximate convergence rate $\hat r_{\mathcal{F}_\alpha}$ appears again sublinear.

\begin{figure}[ht!]
\begin{minipage}{0.30\textwidth}
\centering
\includegraphics[width=4cm,height=4cm]{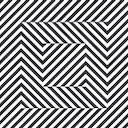}
\end{minipage} \quad 
\begin{minipage}{0.30\textwidth}
\centering
\includegraphics[width=4cm,height=4cm]{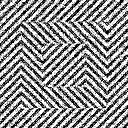}
\end{minipage}\quad 
\begin{minipage}{0.30\textwidth}
\centering
\includegraphics[width=4cm,height=4cm]{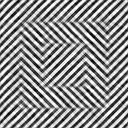}
\end{minipage} \quad 
\caption{\small  From left to right: the ground truth; the data $f$, i.e., the ground truth corrupted by Gaussian noise; the reconstruction obtained with our algorithm.}\label{fig:inc}
\end{figure}

\begin{figure}[ht!]
\begin{minipage}{0.49\textwidth}
\centering
\includegraphics[width=6cm,height=5cm]{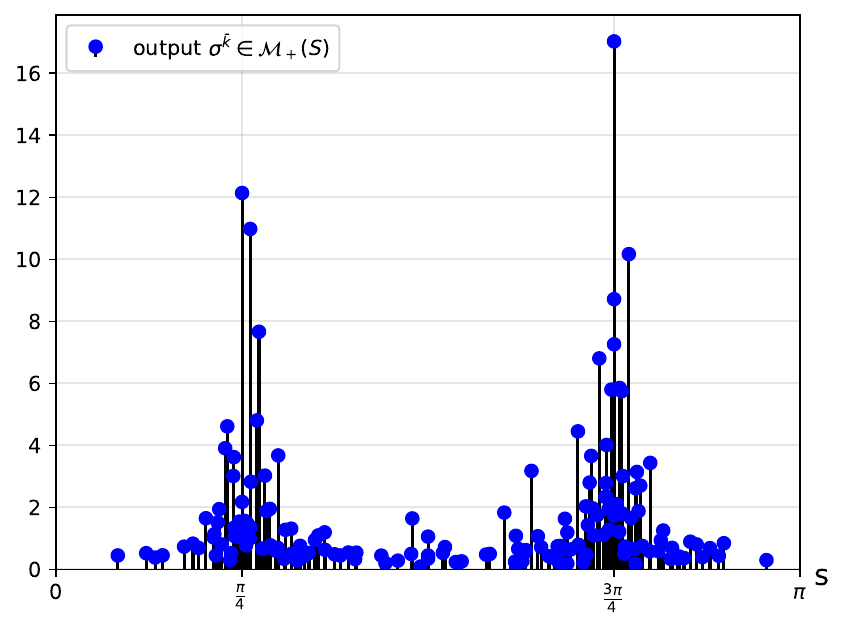}
\end{minipage} 
\begin{minipage}{0.49\textwidth}
\centering
\vspace*{-1mm}
\includegraphics[width=6cm,height=5cm]{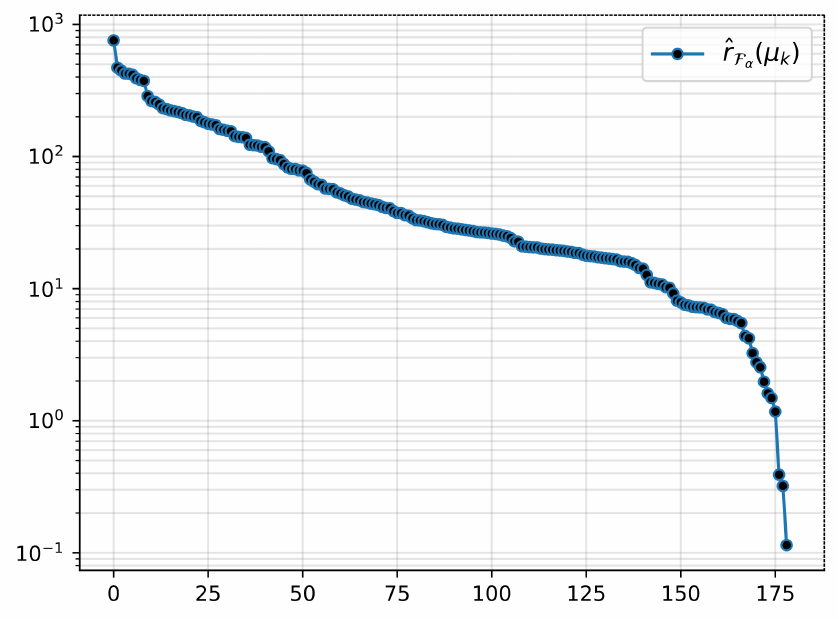}
\end{minipage}
\caption{ \small On the left: the measure $\sigma^{\bar k} \in \mathcal{M}_+(S)$ obtained in output by Algorithm \ref{alg:anisalgorithm}; its support consists of the directions of the dominant anisotropies of the image $s\in [0,\pi]$ selected by Algorithm \ref{alg:anisalgorithm}.
On the right: a convergence graph, showing the approximate residual $\hat r_{\mathcal{F}_\alpha}$ according to \eqref{eq:approxresidual} computed at each iteration.}
\end{figure}
\noindent \textbf{Computational time:} For this experiment, Algorithm \ref{alg:anisalgorithm} terminates in $2720.98$ seconds. Note that the computational time is higher than in the previous example. This is likely due to the more complicated geometric pattern of the data.

In the last experiment we denoise a $253$ $\times$ $253$ grayscale image of a bamboo fence corrupted by Gaussian noise with standard deviation $0.2$, see Figure \ref{fig:fence}. With these choices the relative noise level amounts to $\|f - v\|_{L^2}/\|v\|_{L^2} = 2.3 \cdot 10^{-1}$. In the definition of $J$ in \eqref{eq:fraclapla} we set $\alpha = 6.5$, $\gamma = 0.25$, $\omega = 10^{-3}$, and $\zeta = 5\cdot 10^{-3}$. 
Similarly to  previous examples, we observe that our method  reconstructs sharper edges. For this experiment we compare the reconstructions obtained by our method with classical model-based denoising methods. Firstly, we consider methods that do not impose preferred directions of regularization on the image: Total Variation (TV) \cite{rof} and Total Generalized Variation (TGV) \cite{bkp}, see Figure \ref{fig:fencecomparison}.

\begin{figure}[ht!]
\begin{minipage}{0.30\textwidth}
\centering
\includegraphics[width=4cm,height=4cm]{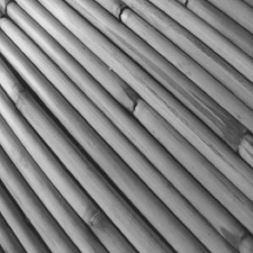}
\end{minipage} \quad 
\begin{minipage}{0.30\textwidth}
\centering
\includegraphics[width=4cm,height=4cm]{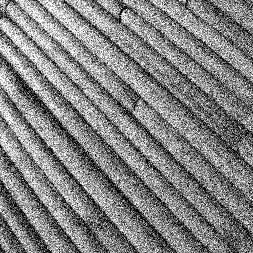}
\end{minipage}\quad 
\begin{minipage}{0.30\textwidth}
\centering
\includegraphics[width=4cm,height=4cm]{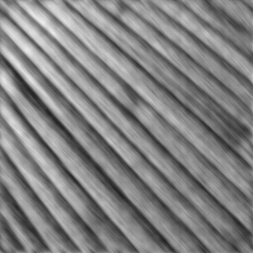}
\end{minipage} \quad 
\caption{ \small From left to right:  ground truth;  data $f$, i.e., ground truth corrupted by Gaussian noise with standard deviation $0.2$ (PSNR = 14.0);  reconstruction obtained with our algorithm (PSNR = 25.8).}\label{fig:fence}
\end{figure}

\begin{figure}[ht!]
\includegraphics[scale=0.35]{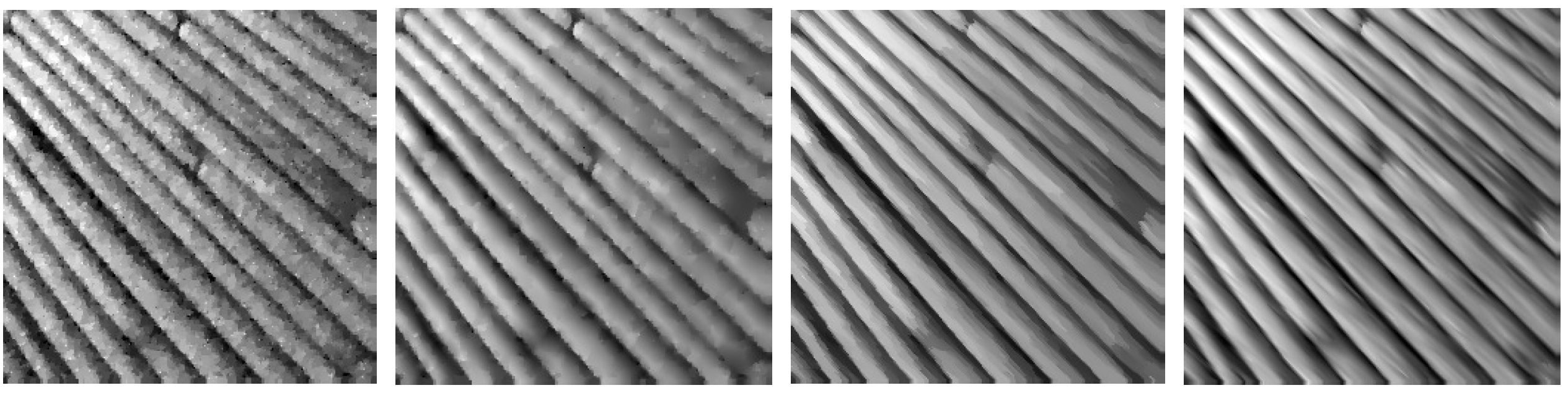}
\caption{Comparison with classical denoising regularizers. From left to right: TV  (PSNR = 23.8), TGV  (PSNR = 24.7), Directional TV  (PSNR = 26.8), and Directional TGV  (PSNR = 28.2)}\label{fig:fencecomparison}
\end{figure}

Our method outperforms them in terms of the PSNR and visual quality. Secondly, we compare our method to approaches that are aware of preferred directions in the image: directional TV  \cite{kongskov2019directional} and  directional TGV  \cite{kongskov2019directional}.
In this case, we achieve slightly worse performance in terms of PSNR and visual quality. 
\begin{figure}[ht!]
\begin{minipage}{0.49\textwidth}
\centering
\includegraphics[width=6cm,height=5cm]{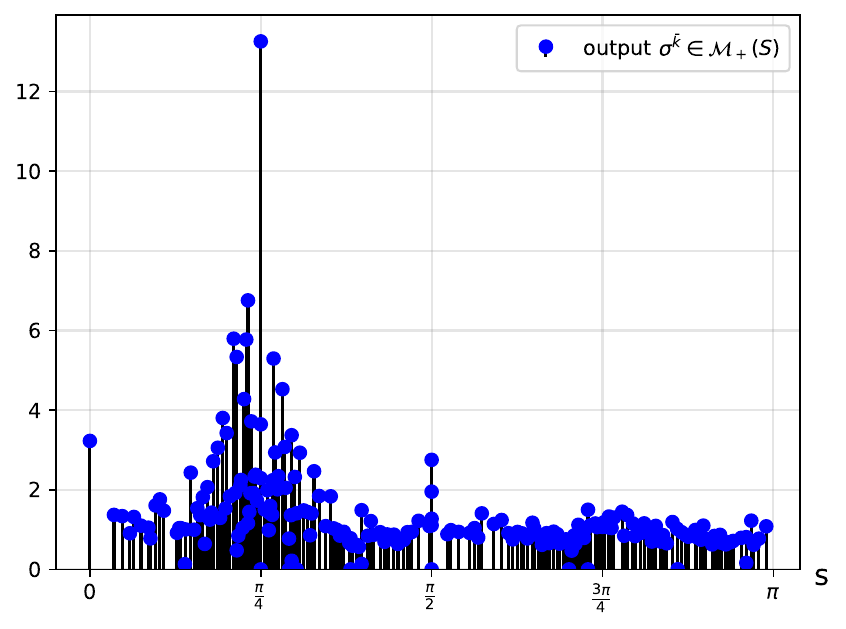}
\end{minipage}\ \  
\begin{minipage}{0.49\textwidth}
\centering
\vspace*{-1mm}
\includegraphics[width=6cm,height=5cm]{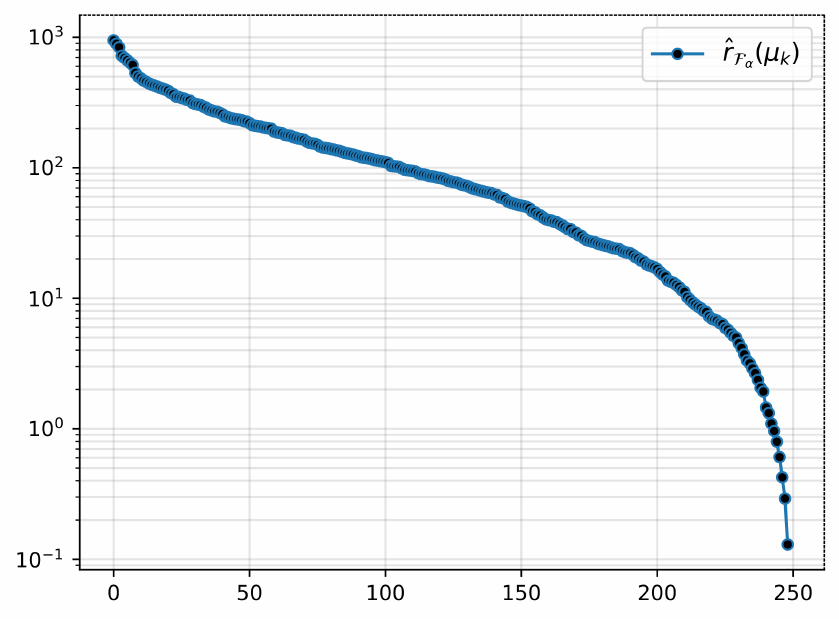}
\end{minipage}
\caption{\small  On the left: the measure $\sigma^{\bar k} \in \mathcal{M}_+(S)$ obtained in output by Algorithm \ref{alg:anisalgorithm}; its support consists of the directions of the dominant anisotropies of the image $s\in [0,\pi]$ selected by Algorithm \ref{alg:anisalgorithm}.
On the right: a convergence graph, showing the approximate residual $\hat r_{\mathcal{F}_\alpha}$ according to \eqref{eq:approxresidual} computed at each iteration.}\label{fig:convfence}
\end{figure}
This is due to the fact that TV and TGV regularization produce better reconstructions than the (isotropic) fractional Laplacian when the ground truths exhibit piecewise constant behavior \cite{bartels2020parameter}. Therefore, for images like the bamboo fence, we cannot expect to outperform directional TV and directional TGV. 
However, we remark that, in contrast to \cite{kongskov2019directional} our method is able to select automatically the directions of the anisotropies of the image that are encoded into the output measure $\sigma \in \mathcal{M}_+(S)$. 
Figure \ref{fig:convfence} shows the directions selected by our algorithm and the convergence plot of the residual. As expected, the selected $s \in [0,\pi]$ concentrate around $\frac{\pi}{4}$ and we again observe a sublinear convergence rate for the residual.\\
\textbf{Computational time:} For this experiment, Algorithm \ref{alg:anisalgorithm} terminates in $12245.31$ seconds. 
In comparison, standard numerical implementations of TV and TGV regularization for denoising terminate in $9$ seconds and $29$ seconds, respectively \cite{kongskov2019directional}. Implementations of Directional TV and Directional TGV are also much faster and terminate in $12$ seconds and $37$ seconds \cite{kongskov2019directional}.  \\
The computational time for adaptively searching in the continuous space of directions is starting to be prohibitive for images of this resolution, also in comparison with non-adaptive methods. However, we would like to remark that a careful optimization of our algorithm could lead to much faster computational times. Moreover, it would be possible to include acceleration steps and better insertion procedures, c.f. Section \ref{sec:inss}, that would allow to ease the computational burden of the algorithm making it also applicable to more diverse imaging tasks.

\subsection{On the computational complexity of the insertion steps}\label{sec:computationalissues}\label{sec:inss}

As anticipated in Remark \ref{rem:comp}, the optimization of \eqref{eq:ott} and \eqref{eq:ott2} is the main computational bottleneck of our numerical experiments. This is due to the fact that problems \eqref{eq:ott}, \eqref{eq:ott2} are neither convex or concave and thus finding the global optimum is a challenging task. For this reason, several strategies have been employed to ease the computational hardness of insertion steps in generalized conditional gradient methods such as multi-start gradient descents, variants of the Newton method and dynamic programming approaches \cite{denoyelle2019sliding, bredies2022generalized, boyd2017alternating, duval2021dynamical,pieper2021linear}. Moreover, acceleration strategies designed to reduce the number of required insertion steps have also been proposed \cite{denoyelle2019sliding, bredies2022generalized}. \\
In our numerical simulations, we have employed a Python implementation of a basin-hopping--type algorithm whose performance is good enough to compute $\argmax_{s \in S} \langle  p^k , v^k_{N_k+1}(s)\rangle$ efficiently and with satisfactory accuracy. Moreover, we have not implemented any acceleration strategy, since our main goal was to illustrate how generalized conditional gradient methods offer a natural framework to design grid-free algorithms with guarantees that can be used to solve inverse problems regularized with infinite infimal convolution operators. Having said that, we believe that targeted optimization schemes for solving insertion steps, together with the addition of suitable acceleration steps, would lead to much faster computational times, allowing for regularizers different from fractional Laplacian-type energies and consequently higher applicability potential.

\section{Conclusions}

\sloppy We introduced the infinite infimal convolution operator for the regularization of ill-posed inverse problems.
Such operator is constructed through an infimal convolution of a continuously parametrized family of convex, positively one-homogeneous and coercive functionals and the regularization depends on a Radon measure defined on the set of parameters $S$ that prescribes which functionals are active in the infimal convolution. The variational  problem is  then formulated with respect to both the measure $\sigma$ in $S$ and the reconstruction is obtained by the integration operation $v = \int_S u(s)\, d\sigma(s)$. Using the positive one-homogeneity of each functional in the infimal convolution, we proposed a convex lifting of the problem in the space of measures with values in Banach space, which allows us to prove well-posedness. Moreover, since this reformulation leads to a convex optimization problem in the space of Radon measures, we could  analyze the sparsity of solutions, obtain a representer theorem and propose a generalized conditional gradient method to compute the minimizers. This allows us to exploit sparsity both in the parameters space and in the reconstruction domain.
Finally, we applied the model and the generalized conditional gradient method to several illustrative examples involving fractional-Laplacian-type regularizers.

The analysis of other types of regularizers (not of fractional-Laplacian type) is left to future work. Successful design of a generalized conditional gradient method relies on the possibility to solve the insertion step \eqref{eq:insertion_step2} in a fast an accurate way for a given $s \in S$. This is true for the fractional Laplacian (see Lemma \ref{lem:explixit1d} and Lemma \ref{lem:insanis}); however, in case on more complicated regularizers, such as directional BV-norms and directional total generalized variations, the resolution of the insertion step could be impractical already for two-dimensional domains and high resolution images.
Since there are no general recipes for solving \eqref{eq:insertion_step2} efficiently, the feasibility of Algorithm \ref{alg:corealgorithm} and strategies of solving it need to be studied case-by-case.

\section{Acknowledgments}

KB gratefully acknowledges support by the Austrian Science Fund (FWF) through project P 29192 ``Regularization graphs for variational imaging''.
MC is supported by the Royal Society (Newton International Fellowship NIF\textbackslash R1\textbackslash 192048 ``Minimal partitions as a robustness boost for neural network classifiers''). The Institute of Mathematics and Scientific Computing, to which KB and MH are affiliated, is a member of NAWI Graz (\texttt{https://www.nawigraz.at/en/}). KB and MH are further members of BioTechMed Graz (\texttt{https://biotechmedgraz.at/en/}). 
YK acknowledges support of the EPSRC (Fellowship EP/V003615/1 and Programme Grant EP/V026259/1) and the National Physical Laboratory. 
CBS acknowledges support from the Philip Leverhulme Prize, the Royal Society Wolfson Fellowship, the EPSRC advanced career fellowship EP/V029428/1, EPSRC grants EP/S026045/1 and EP/T003553/1, EP/N014588/1, EP/T017961/1, the Wellcome Innovator Awards 215733/Z/19/Z and 221633/Z/20/Z, the European Union Horizon 2020 research and innovation programme under the Marie Skodowska-Curie grant agreement No. 777826 NoMADS, the Cantab Capital Institute for the Mathematics of Information, and the Alan Turing Institute. 
MC, YK, and CBS would like to thank the Isaac Newton Institute for Mathematical Sciences in Cambridge (supported by EPSRC grant no EP/R014604/1) for support and hospitality during the programme ``Mathematics of Deep Learning'' where part of the work on this paper was undertaken.

\AtNextBibliography{\small}

\printbibliography

\appendix 

\section{Measures with values in Banach spaces}\label{sec:preliminaries}


In this section we give the necessary definitions and results about measures with values in Banach spaces that are needed in this paper and we formulate them in our setting. We refer to \cite{DiestelUhl,Ryan} to a more detailed overview about the existing theory.

\subsection{\texorpdfstring{$L^p$-functions with values in Banach spaces}{Lp-functions with values in Banach spaces}}\label{app:bochner}

Let $d \in \N$ and $X$ be a Banach space. Given a non-empty compact set $S \subset \R^d$, $\mu \in \mathcal{M}_+(S)$ a Radon measure on $S$ and $p \in [1,\infty)$ we denote by $L^p((S,\mu);X)$ the space of $L^p$ functions with values in $X$ defined as
\begin{equation}
L^p((S,\mu);X) = \left\{u :S \rightarrow X \text{ measurable } : \int_S \|u(s)\|^p_X\, d\mu(s) <\infty\right\}\,,
\end{equation}
where we say that a function $u :S \rightarrow X$ is measurable if it is the $\mu$-a.e. limit of a sequence of simple functions; see \cite[Section 2.3]{Ryan} for more details. Note also that each $u\in L^p((S,\mu);X)$ is $p$-integrable in the sense of Bochner. It is well-known that the space $L^p((S,\mu);X)$ is a Banach space with the norm 
\begin{equation*}
\|u\|_{L^p((S,\mu);X)} = \left(\int_S \|u(s)\|^p_X\, d\mu(s)\right)^{1/p}\,.
\end{equation*}

\subsection{Definition and Radon--Nikod\'ym property}\label{sec:rrd}
We will consider finite Radon measures defined in the compact set $S \subset \R^d$ with values in $X$ that is the dual of a separable Banach space $\predual{X}$ and we denote this set as $\mathcal{M}(S;X)$. The formal definition is similar to the classical notion of measure with values in $\R$. Denoting by $\Sigma$ the Borel $\sigma$-algebra of $S$ a function $\mu : \Sigma \rightarrow X$ is a Borel measure if
\begin{itemize}
\item[i)] $\mu(\emptyset) = 0$
\item[ii)] For any countable family $(E_i)_{i=1,}^\infty$ of pairwise disjoint Borel sets we have that 
\begin{equation}\label{eq:countableadditivity}
\mu\left(\bigcup_{i=1}^\infty E_i\right) = \sum_{i=1}^\infty \mu(E_i),
\end{equation}
 where the limit in \eqref{eq:countableadditivity} is taken with respect to the strong topology of $X$.
\end{itemize}
The classical definition of variation can be easily extended to measures with values in Banach spaces. Given  a Borel measure $\mu$ and $E \in \Sigma$ we define 
\begin{equation}\label{eq:totalvariation}
|\mu|(E) = \sup \left\{\sum_{i=1}^n  \|\mu(A_i)\|_X: \ n \in \N,\,  (A_i)_{i=1}^n \ \text{partition of } E\right\}\,.
\end{equation}
We denote by $\mathcal{M}(S;X)$ the set of Borel measures with finite variation and we call their elements Radon measures. It can be shown that for $\mu \in \mathcal{M}(S;X)$, $|\mu|$ belongs to the set of finite positive Borel measures $\mathcal{M}_+(S)$ \cite[Proposition 9]{DiestelUhl} and $\mathcal{M}(S;X)$ is a Banach space with the total variation norm defined as $\|\mu\|_{\mathcal{M}(S;X)} = |\mu|(S)$ for every  $\mu \in \mathcal{M}(S;X)$. 
We remark that every $\mu \in \mathcal{M}(S;X)$ is regular, in the sense that $|\mu|$ is regular in the classical sense, being a finite positive Borel measure defined on a compact set in $\R^d$.
Moreover, for every $\mathcal{M}(S;X)$ we have $\mu \ll |\mu|$ in the sense that for every Borel set $E\subset S$ such that $|\mu|(E) = 0$ it holds that $\mu(E) = 0$.

\begin{rem}
We point out that the total variation norm defined above can be replaced by the total semi-variation \cite{DiestelUhl} to endow the space $\mathcal{M}(S;X)$ with an alternative Banach space structure. As we consider the integral of strongly measurable functions with respect to measures in $\mathcal{M}(S;X)$ the natural norm to be considered is the total variation norm. We refer to \cite{DiestelUhl, Ryan} for the definition of the total semi-variation and the corresponding theory of weak integration.
\end{rem}
In the next proposition  we deal with vector measures constructed as 
\begin{equation}\label{eq:density1}
\mu(E) = \int_E u(s)\, d\sigma(s)\,,
\end{equation}
for all $E \subset S$ Borel,  where $\sigma \in \mathcal{M}_+(S)$ and $u \in L^1((S,\sigma); X)$. Note that the integral above is defined as the Bochner integral of $u \in L^1((S,\sigma); X)$. Following the notation for classical measures,
given $\mu \in \mathcal{M}(S;X)$ defined as in \eqref{eq:density1} we write $\mu = u \sigma$. Moreover, we call $u$ the density of $\mu$ with respect to $\nu$ and we denote it by $u = \dfrac{\mu}{\sigma}$. Note that $u$ is unique $\sigma$ almost everywhere due to \cite[Corollary 5]{DiestelUhl}. 

\begin{prop}\label{prop:totalvariationproduct}
Let $\Sigma$ be the Borel $\sigma$-algebra of $S$, $\sigma \in \mathcal{M}_+(S)$ and $u \in L^1((S,\sigma); X)$. Then, the map $\mu : \Sigma \rightarrow X$ defined in \eqref{eq:density1}
belongs to  $\mathcal{M}(S;X)$ and its variation is given by
\begin{equation}\label{eq:totalvariationproduct}
|\mu|(E) =  \int_E \|u(s)\|_X \, d\sigma
\end{equation}
for every $E \subset S$ measurable set.
In particular, $\|u(s)\|_X > 0$ for $|\mu|$-a.e. $s \in S$. Moreover, $\mu \ll |\mu|$ and
\begin{equation}\label{eq:density}
\mu = \hat u |\mu|\,,
\end{equation}
where $\hat u \in L^1((S,|\mu|); X)$ is defined as
$\hat u(s) = \frac{u(s)}{\|u(s)\|_X}$ for $|\mu|$-a.e. $s \in S$.
\end{prop}
\begin{proof}
First the verification that $\mu$ belongs to  $\mathcal{M}(S;X)$ is straightforward and \eqref{eq:totalvariationproduct} follows directly from Proposition 5.4 in \cite{Ryan}.
Notice that by the properties of the Bochner integral we have the following inequality
\begin{equation}
\|\mu(E)\|_X \leq \int_E \|u(s)\|_X \, d\sigma(s) = |\mu|(E) 
\end{equation}
implying that $\mu \ll |\mu|$. 
We now show \eqref{eq:density}. 
 Indeed, thanks to  \eqref{eq:totalvariationproduct} we first notice that $\|u(s)\|_X > 0$ for $|\mu|-a.e.$ $s \in S$. Thus, for every measurable set $E$ there holds
\begin{align*}
\int_{E} \frac{u(s)}{\|u(s)\|_X} \, d|\mu|(s) =  \int_{E}  u(s) \, d\sigma(s) = \mu(E)\,,
\end{align*}
as we wanted to prove.
\end{proof}

An important property for measures with values in a Banach space $X$ is the validity of the Radon--Nikod\'ym theorem. 
\begin{dfnz}[Radon--Nikod\'ym property]\label{def:rd}
The Banach space $X$ has the Radon--Nikod\'ym property if for every set $\Omega$, every $\sigma$-algebra $\Sigma$ on $\Omega$ and every measure $\mu :\Sigma \rightarrow \R$ of bounded variation such that $\mu \ll \sigma$ for $\sigma :\Sigma \rightarrow \R$ a finite positive measure, there exists $u \in L^1((\Omega,\sigma); X)$ such that $\mu =u \sigma$.
\end{dfnz}
The Radon--Nikod\'ym property is true for separable dual Banach spaces due to the classical Dunford--Pettis theorem, see Theorem 1 \cite[III.3]{DiestelUhl}.
\begin{thm}[Dunford--Pettis]\label{eq:DP}
Separable dual Banach spaces have the Radon--Nikod\'ym property.
\end{thm}

We conclude this section with the Lebesgue decomposition theorem adapted to measures with values in Banach spaces \cite[Theorem 9]{DiestelUhl}.

\begin{thm}[Lebesgue decomposition theorem]\label{thm:lebesguethm}
Given $\mu \in \mathcal{M}(S;X)$ and $\sigma \in \mathcal{M}_+(S)$, there exist unique $\mu^a, \mu^s \in   \mathcal{M}(S;X)$ such that 
\begin{itemize}
\item $\mu^a \ll \sigma$,
\item $\varphi \circ \mu^s \in \mathcal{M}(S)$ and $\sigma$ are  mutually singular  for every $\varphi \in X^*$,
\item $\mu = \mu^a + \mu^s$.
\end{itemize}
Additionally, $|\mu| = |\mu^a| + |\mu^s|$, and $|\mu^s|$ and $\sigma$ are mutually singular.
\end{thm}

\subsection{Duality theory}\label{subsec:duality}
Define the space $C(S;X)$ to be the set of continuous functions on $S$ with values in $X$ endowed with the norm topology given by
\begin{equation}
\|T\|_{C(S;X)} := \sup_{s \in S} \|T(s)\|_X
\end{equation}
for $T \in C(S;X)$. Here we are interested in the identification of $C(S;X)^*$ with the space of vector measures $\M(S;X^*)$, if $S$ is a non-empty compact set of $\R^d$. This is a classical result due to Singer \cite{singer1957linear} that we state here for the reader's convenience. See also, for instance, \cite{Meziani09} or \cite[Section 6.5]{DiestelUhl}. Note that since $S$ is non-empty and compact, the space $\mathcal{M}(S;X^*)$ coincides with the space of regular $X^*$-valued Borel measures of finite total variation. 
\begin{thm}\label{thm:duality}
Let $S$ be a compact Hausdorff space and $X$ a Banach space. Then, there exists an isometric isomorphism $T$ between $\mathcal{M}(S;X^*)$ and $C(S;X)^*$ such that
\begin{align*}
T(\mu)(f)= \int_S f \,d\mu , \quad \text{for every }\  \mu \in \mathcal{M}(S, X^*),\  f \in C(S;X)
\end{align*}
and
\begin{align*}
 \|\mu\|_{\mathcal{M}(S,X^*)} = \|T(\mu)\|_{C(S,X)^*} = \sup_{\|f\|_{C(S;X)} \leq 1} \int_S f \, d\mu\, ,\quad \text{for every }\  \mu \in \mathcal{M}(S, X^*)\,.
\end{align*}
\end{thm}

In light of \cref{thm:duality} we say that a sequence of regular $X^*$-valued vector measures on $S$ denoted by $\{\mu^n\}_n$ weakly* converges to $\mu$ if
\begin{align}\label{eq:weakconver}
\lim_{n\rightarrow +\infty} \int_S f \,d\mu^n = \int_S f \,d\mu \quad \text{for every } f  \in C(S;X)\,.
\end{align}
We also have the following remark.
\begin{rem}\label{rem:consweakstar}
\cref{thm:duality} and the definition of weak* convergence in \eqref{eq:weakconver} imply that the total variation norm of a regular $X^*$-valued Borel measure on $S$ is  weakly-* lower semicontinuous.
Moreover choosing $f(s) = v \in X$ in \eqref{eq:weakconver} we obtain that 
 \begin{align*}
 \lim_{n\rightarrow +\infty} \langle \mu^n (S), v \rangle = \lim_{n\rightarrow +\infty} \int_S v\, d\mu^n = \int_S v\,  d\mu  = \langle \mu (S), v \rangle
\end{align*}
for every $v \in X$, implying that $\mu^n(S)$ converges weakly* to $\mu(S)$.
\end{rem}
%
%
%
%
%
%
%
%
%
%

\section{Fractional-Laplacian-type operators}\label{app:fraclap}
In this section of the appendix we recall basic results and definitions about fractional Sobolev spaces, fractional Laplacian operators as well as necessary functional analytic tools required in the paper. We also consider the energies obtained as the $L^2$-norm of fractional-Laplacian-type and characterize the extremal points of their balls. For more details about fractional Sobolev spaces and fractional Laplacians we refer to \cite{di2012hitchhikers}.

\subsection{\texorpdfstring{Periodic $L^2$-functions}{Periodic L2-functions}}\label{sec:periodicl2}

For $q \in \N$ we denote by $T^q$ the $(2\pi)$-periodic torus defined by a quotient procedure as $(\R / 2\pi\Z)^q$. We denote by $L^2(T^q)$ the set of $L^2$ functions on $T^q$ defined as
\begin{equation}
L^2(T^q) = \left\{v : T^q \rightarrow \R \ \text{measurable}: \int_{T^q} |v(x)|^2\, dx < \infty\right\}\,.
\end{equation}
We identify the space $L^2(T^q)$ with the space of equivalence classes of $[0,2\pi)^q$-periodic functions $v : \R^q \rightarrow \R$  such that 
\begin{equation}
 \int_{(0,2\pi)^q} |v(x)|^2\, dx < \infty\,.
\end{equation}
We endow $L^2(T^q)$ with a Hilbert structure by means of the scalar product
\begin{equation}
\langle v , w \rangle_{L^2(T^q)} = \int_{(0,2\pi)^q} v(x) w(x)\, dx \quad \text{for } v, w \in L^2(T^q)
\end{equation}
inducing the norm
\begin{equation}
\|v\|_{L^2(T^q)} = \sqrt{\int_{(0,2\pi)^q} |v(x)|^2\, dx}\,.
\end{equation}
The space $L^2(T^q)$ is a separable Hilbert space and its predual is $L^2(T^q)$ itself.
We also introduce the $L^2$-periodic functions with zero mean, defined as
\begin{equation}
L_\circ ^2(T^q) = \left\{u\in L^2(T^q) : \int_{(0,2\pi)^q} v(x)\, dx = 0\right\}\,.
\end{equation}
Notice that $L_\circ ^2(T^q)$ is a closed subspace of $L^2(T^q)$. Therefore, $L_\circ ^2(T^q)$ is a separable Hilbert space as well. 

\subsection{\texorpdfstring{Fourier series of periodic $L^2$-functions}{Fourier series of periodic L2-functions}}

For $m \in \Z^q$ we define the $m$-th Fourier coefficient of $v \in L^2(T^q)$ as
\begin{equation}
\mathcal{F} v(m) = \hat v(m) = \frac{1}{(2\pi)^q} \int_{(0,2 \pi)^q} v(x) e^{-ix\cdot m}\, dx\,.
\end{equation}
It is standard to show that for every $v \in L^2(T^q)$ the Fourier series 
$\sum_{m \in \Z^q} \hat v(m)e^{ix\cdot m}$
converges to $v$ in $L^2(T^q)$ and we write
$v(x) = \sum_{m \in \Z^q} \hat v(m)e^{ix\cdot m}$.
Moreover, Parseval's theorem holds for every $v,w \in L^2(T^q)$, i.e.,
\begin{equation}\label{eq:plancherel}
\langle v ,w \rangle_{L^2(T^q)} = (2\pi)^q \langle \hat v , \hat w \rangle_{\ell^2(\Z^q)}\,.
\end{equation}

\subsection{Periodic fractional Sobolev functions}
We now define fractional Sobolev functions on the periodic domain $T^q$ for $q \in \N$.
\begin{dfnz}[Fractional Sobolev functions on periodic domains]
Given a function $v: T^q \rightarrow \R$ and $s \geq 0$,  we say that $v \in H^s(T^q)$ if $v \in L^2(T^q)$  and
\begin{equation}
[v]^2_{H^s(T^q)} := \sum_{m \in \Z^q} |m|^{2s} |\hat v(m)|^2 < \infty\,.
\end{equation}
The space $H^s(T^q)$ is a Hilbert space with the scalar product defined as 
\begin{equation}
\langle v, w\rangle_{H^s(T^q)} = \langle v, w\rangle_{L^2(T^q)} + \sum_{m \in \Z^q} |m|^{2s} \hat v(m) \overline{\hat w(m)}
\end{equation}
inducing a norm defined as $\|v\|_{H^s(T^q)} = \sqrt{\|v\|^2_{L^2(T^q)} + [v]^2_{H^s(T^q)}}$.
\end{dfnz}

\begin{rem}
Notice that for $s = 0$ the space $ H^s(T^q)$ is isomorphic to $L^2(T^q)$ and $\|v\|_{H^0(T^q)} = \sqrt{1 + (2\pi)^{-q}} \|v\|_{L^2(T^q)}$ for every $v \in L^2(T^q)$. 
\end{rem}

\subsection{The fractional Laplacian}

We define the fractional Laplacian in periodic domains as follows.
\begin{dfnz}[Fractional Laplacian]
Given $v \in  H^{2s}(T^q)$ and $s \in (0,1]$ we define the fractional Laplacian of $v$ as
\begin{equation}
(-\Delta^{s})v(x) = \sum_{m \in \Z^q} |m|^{2s} \hat v(m) e^{ix \cdot m}  \,.
\end{equation} 
\end{dfnz}

\begin{rem}\label{rem:frac-Hs}
According to the previous definitions it is easy to check that $(-\Delta^{s})v \in L^2(T^q)$ with
\begin{equation}
\|(-\Delta^{s})v\|_{L^2(T^q)} = (2\pi)^{q/2} [v]_{H^{2s}(T^q)} \,.
\end{equation}
\end{rem}
Using Parseval's theorem it is easy to show that the following version of the Poincar\'e inequality holds.
\begin{thm}[Poincar\'e inequality]\label{thm:poinc}
For every $0 \leq s \leq 1$ there holds
\begin{equation}\label{eq:po}
\|v\|_{L^2(T^q)} \leq \|(-\Delta^{s})v\|_{L^2(T^q)}
\end{equation}
for every $v \in H^{2s}(T^q)$ with zero mean.
\end{thm}
\subsection{\texorpdfstring{Extremal points of $L^2$-balls of fractional-Laplacian-type operators}{Extremal points of L2-balls of fractional-Laplacian-type operators}}
In this paper we will consider $L^2$-norms of fractional-Laplacian-type operators. This section is devoted to characterize the extremal points of their $L^2$ balls.
Let $f  : \Z^q \rightarrow [0,\infty)$ be a function such that 
\begin{enumerate}
\item [i)] For every $m \neq 0$ it holds that $f(m) > 0$.
\end{enumerate} 
We consider the convex, positively one-homogeneous functional $J_f : L^{2}(T^q) \rightarrow [0,\infty)$ defined as 
\begin{equation}\label{eq:jayf}
J_f(v) := \sqrt{\sum_{m \in \Z^q} f(m) |\hat v(m)|^2}\,.
\end{equation}
We now characterize the extremal points of the unit ball of $J_f$ in $ L^{2}_\circ(T^q)$.
\begin{lemma}\label{lem:extfractionallaplacian}
Consider the set $B = \{v \in L^{2}_\circ(T^q) : J_f(v) \leq 1\}$. We have that 
\begin{equation}
\Ext(B) = \left\{v \in L^{2}_\circ(T^q): J_f(v) = 1\right\}\,.
\end{equation}
\end{lemma}

\begin{proof}
First note that $J_f(v) > 0$ whenever $v \neq 0$ thanks to Assumption $i)$ of $f$.
Setting
\begin{equation}
\mathscr{B} = \left\{v \in L^{2}_\circ(T^q): J_f(v) = 1\right\}\,,
\end{equation}
we want to prove that $\Ext(B) = \mathscr{B}$.
We first proceed to prove the inclusion $\mathscr{B} \subset \Ext(B)$. Given $v \in \mathscr{B}$ consider the convex combination 
\begin{equation}\label{eq:convex12}
v = \lambda v_1 + (1 - \lambda) v_2
\end{equation}
with $\lambda \in (0,1)$ and $v_1,v_2 \in B$.
Applying the functional $J_f$ to both sides of \eqref{eq:convex12}, using the convexity of $J_f$ and that $v_1,v_2 \in B$ we obtain that $\lambda J_f(v_1) + (1 - \lambda) J_f(v_2) = 1$ 
implying that $J_f(v_1) = J_f(v_2) = 1$. Therefore, from \eqref{eq:convex12} we get 
\begin{align*}
1 =  &\lambda^2\sum_{m \in \Z^q } f(m) |\hat v_1(m)|^2 + (1-\lambda )^2\sum_{m \in \Z^q} f(m) |\hat v_2(m)|^2   +2\lambda (1-\lambda )\Rea \sum_{m \in \Z^q} f(m)\overline{\hat v_1(m)} \hat v_2(m)
\end{align*}
implying, since $\lambda \neq 0$ and $\lambda \neq 1$, that 
\begin{equation*}
\Rea \sum_{m \in \Z^q } f(m)\overline{\hat v_1(m)} \hat v_2(m) = 1\,.
\end{equation*}
 Recalling that $J_f(v_1) = J_f(v_2) = 1$, a simple contradiction argument that uses the sharpness of Cauchy-Schwarz inequality shows that 
$\hat v_1(m) =  \hat v_2(m)$ for every $m \in \Z^q$. In particular, $v_1 = v_2 = v$
implying that $v$ is an extremal point for $B$.

Conversely, let us show that $\Ext(B) \subset \mathscr{B}$.
Suppose that such inclusion does not hold true. Then, there exists $v \in \Ext(B)$ such that $J_f(v) \neq 1$. In particular, since $v \in B$, there exists $\e > 0$ such that   $J_f(v) < 1 - \e$. Consider $w \in L^{2}_\circ(T^q) \setminus \{0\}$ such that $
J_f(w) \leq \frac{\e}{2}$.
Notice that such $w$ exists since $\domain\,J_f \neq \{0\}$ and $J_f$ is positively one-homogeneous. Consider then the convex combination
\begin{equation}
v = \frac{1}{2} (v + w) + \frac{1}{2}(v-w)\,.
\end{equation}
Since $w \neq 0$, we immediately have that $v+ w \neq v - w$. Moreover
\begin{equation}
J_f(v \pm w)  \leq J_f(v) + J_f(w)  < 1\,,
\end{equation}
showing that $v$ is not an extremal point.
\end{proof}
If $f$ is strictly positive, then we can characterize the extremal points of $J_f$ in $L^{2}(T^q)$ as follows.
\begin{lemma}\label{lem:extfractionallaplacian2}
Suppose that $f$ satisfies 
\begin{itemize}
\item [i')] For every $m \in \Z^q$ it holds that $f(m) > 0$.
\end{itemize}
Then given the set $\tilde B = \{v \in L^{2}(T^q) : J_f(v) \leq 1\}$ we have that 
\begin{equation}
\Ext(\tilde B) = \left\{v \in L^{2}(T^q) : J_f(v) = 1\right\}\,.
\end{equation}
\end{lemma}
Since the proof follows closely the one of Lemma \ref{lem:extfractionallaplacian} we omit it.

\end{document}